\documentclass[a4paper, 12pt, twoside]{article}

\newcommand{\version}{2.8}

\author{Antongiulio Fornasiero%
\thanks
{University of M\"unster.
\Email{antongiulio.
fornasiero@googlemail.com}}}
\def\shorttitle{Dimensions, matroids, and dense pairs. V. \version}
\title{Dimensions, matroids, and dense pairs of first-order structures
\normalsize{Version \version}}

\newcommand{\Email}[1]{\href{mailto:#1}{\tt<#1>}}

\newlength{\offset}
\usepackage{ifpdf}

\usepackage{amsmath, amsthm, amssymb}
\usepackage{xspace}

\usepackage{tocloft}

\usepackage{paralist}

\usepackage{fancyhdr}

\usepackage[T1]{fontenc}

\ifpdf
 \usepackage[pdftex, linktocpage=true, pdfborder={0 0 0}, colorlinks=false,
  plainpages=false, pdfpagelabels, pdfdisplaydoctitle=true, bookmarks=true, 
  bookmarksnumbered=true]{hyperref}
 \hypersetup{pdfauthor={A. Fornasiero},
   pdftitle={Dimensions,  v. \version},
   pdfkeywords={Topological structure, pregeometry, existential matroid,
    dimension, dense pair, lovely pair, d-minimal},
   pdfsubject={Existential matroids  and dense pairs}
   pdfduplex=DuplexFlipLongEdge
}
\else
 \usepackage[plainpages=false,pdfpagelabels,colorlinks=false,hypertex]{hyperref}
\fi

\usepackage[notref, notcite]{showkeys}

\makeatletter
\providecommand{\cftdotfill}{\@cftdotfill}
\makeatother

\newcommand{\intro}[1]{\textbf{#1}}
\newcommand*{\Pa}[1]{\bigl( #1 \bigr)}
\newcommand*{\PA}[1]{\Bigl( #1 \Bigr)}
\newcommand*{\set}[1]{\{#1\}}
\newcommand*{\abs}[1]{\lvert#1\rvert}
\newcommand*{\card}[1]{\lvert#1\rvert}
\newcommand{\N}{\mathbb{N}}

\newcommand{\Z}{\mathbb{Z}}

\newcommand{\inter}[1]{\mathring{#1}}
\DeclareMathOperator{\interior}{int}
\newcommand{\cl}[1]{\overline{#1}}

\DeclareMathOperator{\bd}{bd}
\newcommand{\rest}{\upharpoonright}
\DeclareMathOperator{\dcl}{dcl}
\DeclareMathOperator{\acl}{acl}
\newcommand{\et}{\ \&\ }

\newcommand{\sdiff}{\mathbin\Delta}
\newcommand{\av}{\bar a}
\newcommand{\bv}{\bar b}
\newcommand{\cv}{\bar c}
\newcommand{\dv}{\bar d}
\newcommand{\ev}{\bar e}
\newcommand{\x}{\bar x}
\newcommand{\y}{\bar y}
\newcommand{\z}{\bar z}
\DeclareMathOperator{\tp}{tp}

\newcommand{\Lang}{\mathcal L}
\newcommand{\monster}{\mathbb M}
\newcommand{\Gm}{\mathbb G}
\newcommand{\Am}{\mathbb A}
\newcommand{\Bm}{\mathbb B}
\newcommand{\Cm}{\mathbb C}
\newcommand{\Dm}{\mathbb D}
\DeclareMathOperator{\SU}{U}
\newcommand{\SUrk}{$\SU$-rank\xspace}
\DeclareMathOperator{\mat}{cl}
\newcommand{\dimat}{\dim^{\mat}}
\newcommand{\rkmat}{\RK^{\mat}}
\DeclareMathOperator{\aut}{Aut}
\newcommand{\Power}{\mathcal P}
\DeclareMathOperator{\RK}{rk}
\newcommand{\rk}{\RK}
\def\Ind#1#2{#1\setbox0=\hbox{$#1x$}\kern\wd0\hbox to
  0pt{\hss$#1\mid$\hss}\lower.9\ht0\hbox to 0pt{\hss$#1\smile$\hss}\kern\wd0}
\newcommand*{\ind}[1][]{\mathop{\mathpalette\Ind{}^{\!\!\!\!\rlap{$\scriptscriptstyle\textnormal{#1}$}\,\,\,\,}}}
\def\mind{\ind[M]}

\def\tind{\ind[\th]}
\def\indmat{\ind[$\mat$]}
\def\indf{\ind[f]}
\newcommand*{\notind}[1][]{\mathrel{\not\mkern-7mu{\ind[#1]}}}
\newcommand*{\notindmat}{\notind[$\mat$]}
\newcommand{\elem}{\equiv}

\DeclareMathOperator{\id}{id}
\newcommand*{\pair}[1]{\langle #1 \rangle}

\newcommand{\F}{\mathbb{F}}
\newcommand{\K}{\mathbb{K}}
\newcommand*{\eq}[1]{{#1}^{\mathrm{eq}}}
\newcommand{\Teq}{\eq{T}}
\newcommand{\monstereq}{\eq{\monster}}
\newcommand{\matt}{\mateq}
\newcommand{\rkmatt}{\mathop{\tilde{\RK}}}
\newcommand{\mateq}{\mathop{\eq{\mat}}}
\newcommand{\acleq}{\mathop{\eq{\acl}}}
\newcommand{\dcleq}{\mathop{\eq{\dcl}}}
\newcommand{\Kfam}{\mathcal K}
\newcommand{\Afam}{\mathcal A}
\DeclareMathOperator{\charact}{char}
\newcommand{\Utp}{\ensuremath{P\text{-}\tp}}
\newcommand{\Uindependent}{$P$\hyph independent\xspace}
\newcommand{\Utype}{$P$\hyph type\xspace}

\def\hyph{\nobreakdash-\hspace{0pt}\relax}
\newcommand{\Wlog}{W.l.o.g\mbox{.}\xspace}
\newcommand{\wloG}{w.l.o.g\mbox{.}\xspace}
\newcommand{\wrt}{w.r.t\mbox{.}\xspace}
\newcommand{\eg}{e.g\mbox{.}\xspace}
\newcommand{\ie}{i.e\mbox{.}\xspace}

\newcommand{\cf}{cf\mbox{.}\xspace}
\newcommand{\Cf}{Cf\mbox{.}\xspace}
\newcommand{\tfae}{t.f.a.e\mbox{.}\xspace}
\newcommand{\Tfae}{T.f.a.e\mbox{.}\xspace}
\newcommand{\aka}{a.k.a\mbox{.}\xspace}

\newcommand{\dminimal}{d-minimal\xspace}

\newcommand{\dminimality}{d-minimality\xspace}
\newcommand{\prefree}{pre-independence\xspace}
\newcommand{\Prefree}{Pre-independence\xspace}

\newcommand{\xnarrow}{$x$-narrow\xspace}
\newcommand{\cldense}{dense\xspace}
\newcommand{\clclosed}{$\mat$\hyph closed\xspace}
\newcommand{\clclosure}{$\mat$\hyph closure\xspace}
\newcommand{\clelimination}{$\mat$\hyph elimination\xspace}
\newcommand{\Zclosed}{$Z$\hyph closed\xspace}
\newcommand{\clbasis}{$\mat$\hyph basis\xspace}
\newcommand{\clindependent}{$\mat$\hyph independent\xspace}
\newcommand{\clminimal}{$\mat$\hyph minimal\xspace}
\newcommand{\clminimality}{$\mat$\hyph minimality\xspace}
\newcommand{\DSF}{definable Skolem functions\xspace}
\DeclareMathOperator{\zcl}{Zcl}
\DeclareMathOperator{\zrk}{Zrk}
\DeclareMathOperator{\scl}{Scl}
\DeclareMathOperator{\sdim}{Sdim}
\DeclareMathOperator{\srk}{Srk}
\newcommand{\Ltwo}{\Lang^2}
\newcommand{\Ttwo}{T^2}
\newcommand{\Td}{T^d}
\newcommand{\Ln}{\Lang^n}
\newcommand{\Tn}{T^n}
\newcommand{\de}{\mathrm d}
\newcommand{\app}{\leadsto}
\newcommand{\TR}{T_{R \nmid 0}}
\newcommand{\LR}{\Lang_R}
\newcommand*{\canon}[1]{\ulcorner #1 \urcorner}

\newcommand{\Zapplication}{Z-application\xspace}
\newcommand{\basic}{basic\xspace}
\newcommand{\Basic}{Basic\xspace}

\newcommand{\Case}[1]{\par\smallskip\noindent\textsc{Case #1.}}

\newtheorem{lemma}{Lemma}[section]
\newtheorem{thm}[lemma]{Theorem}
\newtheorem{corollary}[lemma]{Corollary}
\newtheorem{conjecture}[lemma]{Conjecture}
\newtheorem{proposition}[lemma]{Proposition}

\newtheorem*{fact}{Fact}
\newtheorem*{hypothesis*}{Hypothesis}
\newtheorem*{proviso*}{Proviso}

\theoremstyle{remark}

\newtheorem{claim}{Claim}
\newtheorem*{claim*}{Claim}

\theoremstyle{definition}
\newtheorem{definizione}[lemma]{Definition}

\newtheorem{remark}[lemma]{Remark}
\newtheorem{final remark}[lemma]{Final remark}
\newtheorem{example}[lemma]{Example}
\newtheorem{examples}[lemma]{Examples}
\newtheorem{question}[lemma]{Question}

\begin{document}

\maketitle

\begin{abstract}
A structure $M$ is pregeometric if the algebraic closure is a
pregeometry in all $M'$ elementarily equivalent to~$M$.
We define a generalisation: structures with an existential matroid.
The main examples are superstable groups of \SUrk a power of $\omega$ 
and d-minimal expansion of fields. 
Ultraproducts of pregeometric structures expanding a field, while
not pregeometric in general, do have an unique existential matroid.

Generalising previous results by L.~van den Dries, we define dense elementary
pairs of structures expanding a field and with an existential matroid, and we
show that the corresponding theories have natural completions, whose models
also have a unique existential matroid.
We also extend the above result to dense tuples of structures.
\end{abstract}


\tableofcontents

\section{Introduction}
A theory $T$ is called pregeometric~\cite{hp,gagelman} if, in every model~$\K$ of~$T$, $\acl$ satisfies the Exchange Principle (and, therefore, $\acl$ is a
pregeometry on~$\K$); if $T$ is complete, it suffices to check that $\acl$
satisfies EP in one $\omega$-saturated model of~$T$.
$T$ is geometric if it is pregeometric and eliminates 
the quantifiers $\exists^\infty$.
We call a structure $\K$ (pre)geometric if its theory is (pre)geometric
(thus, $\K$ is pregeometric iff there exists an $\omega$-saturated elementary
extension $\K'$ of $\K$ such that $\acl$ satisfies EP in~$\K'$).

Note that a pregeometric expansion of a field is geometric 
(\cite[1.18]{DMS}, see also Lemma~\ref{lem:cl-ring-function}).

In the remainder of this introduction, all theories and all structures expand
a field; in the body of the article we will sometimes state definitions and
results without this assumption.

Geometric structures are ubiquitous in model theory: if $\K$ is either
o-minimal, or strongly minimal, or a $p$-adic field, or a pseudo-finite field
(or more generally a perfect PAC field, see \cite{cdm} and \cite[2.12]{hp}), 
then $\K$ is geometric.

However, ultraproducts of geometric structures (even strongly minimal ones) are not geometric in general.
We will show that there is a more general notion, structures with existential
matroids, which instead is robust under taking ultraproducts.
More in details, we consider structures $\K$ with a matroid $\mat$ that
satisfies some natural conditions ($\mat$ is an ``existential matroid'').
Under our hypothesis that $\K$ is a field, then there is at most one
existential matroid on~$\K$.
An (almost) equivalent notion has already been studied by van den
Dries~\cite{dries89}: we will show that, if $\monster$ is a monster model, an existential matroid on $\monster$ induces a (unique)
dimension function on the definable subset of $\monster^n$, satisfying the
axioms in~\cite{dries89}, and conversely, any such dimension function,
satisfying a slightly stronger version of the axioms, will be induced by a
(unique) existential matroid.
Moreover, a superstable group $\K$ of \SUrk a power of $\omega$ is
naturally endowed by an existential matroid (van den Dries
\cite[2.25]{dries89} noticed this already in the case when $\K$ is a
differential field of characteristic~0). 

Given a geometric structure~$\K$, there is an abstract notion of dense
subsets of~$\K$, which specialises to the usual topological notion in
the case of o-minimal structures or of formally $p$-adic
fields.
More precisely, a subset $X$ of $\K$ is dense in $\K$ if every \emph{infinite}
$\K$-definable subset of $\K$ intersect $X$~\cite{macintyre}. 
If $T$ is a complete geometric theory, then the theory of dense elementary
pairs of models of $T$ is complete and consistent (the proof of this fact was
already in~\cite{vdd-dense}, but the result was stated there only for
o-minimal structures). 

We consider here the more general case when $T$ is a complete theory  and
such that a monster model of $T$ has an existential matroid.
We show that there is a corresponding abstract notion of density in models 
of~$T$.
Given $T$ as above, consider the theory of pairs $\pair{\K, \K'}$, where 
$\K \prec \K' \models T$ and $\K$ is dense in $\K'$;
the theory of such pairs will not be complete in general, but we will show
that it will become complete (and consistent) if we
add the additional condition that $\K$ is \clclosed in $\K'$ 
(that is, $\mat(\K) \cap \K' = \K)$; we thus obtain the (complete) theory~$\Td$.
Moreover $\Td$ also has an existential matroid. 
This allows us to iterate the above construction, and consider dense \clclosed
pairs of models of~$\Td$, which turn out to coincide with nested dense \clclosed
triples of models of~$T$; iterating many times, we can thus study nested dense
\clclosed $n$-tuples of models of~$T$.

Of particular interest are two cases of structures with an existential matroid:
the \clminimal case and the \dminimal one.

A structure $\K$ (with an existential matroid) is \clminimal if there is only
one ``generic'' 1-type over every subset of~$\K$; the prototypes of such
structures are given by strongly minimal structures and connected superstable
groups of \SUrk a power of~$\omega$. 
If $T$ is the theory of~$\K$, we show that the condition that $\K$ is dense in
$\K'$ is superfluous in the definition of $\Td$, and that $\Td$ is also
\clminimal.

An first-order topological structure $\K$ (expanding a topological field) is
\dminimal if it is Hausdorff, it has an $\omega$-saturated elementary
extension $\K'$ such that every definable (unary!)
subset of $\K'$ is the union of an open set and finitely many discrete sets,
and it satisfies a version of Kuratowski-Ulam's theorem for definable subset
of~$\K^2$ (the ``$d$'' stands for ``discrete'').
Examples of \dminimal structures are $p$-adic fields, o-minimal structures, 
and \dminimal structures in the sense of Miller.
We show that a \dminimal structure has a (unique) existential matroid, and that
the notion of density given by the matroid coincides with the topological one.
Moreover, if $T$ is the theory of a \dminimal structure, then $\Td$ is the
theory of dense elementary pairs of models of $T$ (the condition that $\K$ is
a \clclosed subset of $\K'$ is superfluous); hence, in the case when $T$ is
o-minimal, we recover van den Dries' Theorem~\cite{vdd-dense}.
However, if $T$ is \dminimal, $\Td$ will not be \dminimal.
Moreover, while ultraproducts of o-minimal structures and of formally $p$-adic
fields are \dminimal, ultraproducts of \dminimal structures are not \dminimal
in general.

We show that if $\K$ has an existential matroid, then $\K$ is a
perfect field: therefore, the theory exposed in this article does not apply to
differential fields of finite characteristic, or to separably closed
non-perfect fields.


\section{Notations and conventions}
Let $T$ be a complete theory in some language~$\Lang$, with only infinite 
models.
Let $\kappa > \card T$ be a ``big'' cardinal.
We work inside a $\kappa$-saturated and strongly $\kappa$-homogeneous 
model $\monster$ of~$T$: we call $\monster$ a monster model of~$T$.

$A$, $B$, and $C$, subsets of $\monster$ of cardinality less than~$\kappa$; 
by $\av$, $\bv$, and $\cv$, finite tuples of elements of~$\monster$; 
by $a$, $b$, and $c$, elements of~$\monster$.
As usual, we will write, for instance, $\av \subset A$ to say that $\av$ is a
finite tuple of elements of~$A$, and $A \bv$ to denote the union of $A$ with
the set of elements in~$\bv$.

Given a set~$X$ and $m \leq n \in N$, 
denote by $\Pi^n_m$ the projection from $X^n$ onto the first $m$
coordinates.
Given $Y \subseteq X^{n+m}$, $\x \in X^n$, and $\z \in X^m$,
denote the sections 
$Y_{\x} := \set{\bar t \in X^m: \pair{\x, \bar t} \in Y}$ and
$Y^{\z} := \set{\bar t \in X^m: \pair{\bar t, \z} \in X}$.


\section{Matroids}\label{sec:matroid}


Let $\mat$ be a (finitary) closure operator on $\monster$:
that is, $\mat: \Power(\monster) \to \Power(\monster)$ satisfies, 
for every $X \subseteq \monster$:
\begin{description}
\item[extension] $X \subseteq \mat(X)$;
\item[monotonicity] $X \subseteq Y$ implies $\mat(X) \subseteq \mat (Y)$;
\item[idempotency] $\mat(\mat X) = \mat(X)$;
\item[finitariness] $\mat(X) = \bigcup \set{\mat(A): A \subseteq X \et A \text{ finite}}$.
\end{description}
$\mat$ is a (finitary) \intro{matroid} (\aka \intro{pregeometry})
if moreover it satisfies the Exchange Principle:
\begin{description}
\item[EP] $a \in \mat(X c) \setminus \mat(X)$ implies $c \in \mat(X a)$.
\end{description}

\begin{proviso*}
For the remainder of this section, $\mat$~is a finitary matroid on~$\monster$.
\end{proviso*}

As is well-known from matroid theory, $\mat$~defines
notions of rank (which we denote by $\rkmat$), generators, independence, 
and basis.%
\footnote{Sometimes in geometric model theory the ``rank'' is called
``dimension'' and/or the ``dimension'' (defined later) is called ``rank'';
however, since in many interesting cases (\eg algebraically closed fields, and
o-minimal structures, with the $\acl$ matroid) what we call the
dimension of a definable set induced by the matroid coincides with the usual
notion of dimension given geometrically, our choice of nomenclature is clearly
better.}  

\begin{definizione}
A set $A$ generates $C$ over $B$ if $\mat(A B) = \mat(C B)$.
A subset $A$ of $\monster$ is independent over $B$ if, 
for every $a \in A$, $a \notin \mat(B a': a \neq a' \in A)$.
\end{definizione}

\begin{lemma}[Additivity of rank]
\label{lem:Lascar}
\[
\rkmat(\av \bv / C) = \rkmat(\av / \bv C) + \rkmat(\bv / C).
\]
\end{lemma}

For the axioms of independence relations, we will use the nomenclature 
in~\cite{adler}.
\begin{definizione}
Given an infinite set $X$, a \intro{\prefree relation}%
\footnote{\Prefree relations as defined here are slightly different
  than the ones defined in~\cite{adler}.
However, as we will see later, if $\mat$ is definable, then $\indmat$ is a
\prefree relation in Adler's sense.} 
on $X$ is a the ternary relation $\ind$ on $\Power(X)$
satisfying the following axioms:
\begin{description}
\item[Monotonicity:]
If $A \ind_C B$, $A' \subseteq A$, and $B' \subseteq B$, then
$A' \ind_C B'$.
\item[Base Monotonicity:]
If $D \subseteq C \subseteq B$ and $A \ind_D B$, then $A \ind_C B$.
\item[Transitivity:]
If $D \subseteq C \subseteq B$, $B \ind_C A$, and $C \ind_D A$, 
then $B \ind_D A$.
\item[Normality:]
If $A \ind_C B$, then $AC \ind_C B$.
\item[Finite Character:]
If $A_0 \ind_C B$ for every finite $A_0 \subseteq A$, then
$A \ind_C B$.
\end{description}
$\ind$ is \intro{symmetric} if moreover it satisfies the following axiom:
\begin{description}
\item[Symmetry:]
$A \ind_C B$ iff $B \ind_C A$.
\end{description}
\end{definizione}

\begin{definizione}
The \prefree relation on $\monster$ induced by $\mat$ is the ternary relation
$\indmat$ on $\Power(\monster)$ defined by:
$X \indmat_Y Z$ if for every $Z' \subset Z$, if $Z'$ is independent over~$Y$, 
then  $Z'$ remains independent over $Y X$.
If $X \indmat_Y Z$, we say that $X$ and $Z$ are independent over~$Y$
(\wrt~$\mat$).
\end{definizione}

\begin{remark}
If $X \indmat_Y Z$, then $\mat(X Y) \cap \mat(Z Y) = \mat (Y)$.
\end{remark}

\begin{lemma}
$\indmat$ is a symmetric \prefree relation.
\end{lemma}
\begin{proof}
The same given in \cite[Lemma 1.29]{adler}.
\end{proof}

\begin{remark}
$\indmat$ also satisfies the following version of anti-reflexivity:
\begin{itemize}
\item
$A \indmat_C B$ iff $\mat(A) \indmat_{\mat(C)} \mat(B)$;
\item 
$a \indmat_X a$ iff $a \in \mat(X)$.
\end{itemize}
\end{remark}

\begin{remark}
$X \indmat_Y Y$. 
\end{remark}

\begin{lemma}
\Tfae:
\begin{enumerate}
\item $X \indmat_Y Z$;
\item
$\forall Z'$ such that $Y \subseteq Z' \subseteq \mat(YZ)$, 
we have $\mat(X Z') \cap \mat(Y Z) = \mat(Z')$;
\item there exists $Z' \subseteq Z$ which is a basis of $Z/Y$,
such that $Z'$ remains independent over $X Y$;
\item for every $Z' \subseteq Z$ which is a basis of $Z/Y$,
$Z'$ remains independent over $Y X$;
\item if $X' \subseteq X$ is a basis of $Y X /Y$ and $Z' \subseteq Z$ is a
basis of $Y Z / Y$, then $X'$ and $Z'$ are disjoint, and $X' Z'$ is a basis of
$X Z$ over~$Y$;
\item $\rkmat(X / Y Z) = \rkmat(X / Y)$.
\end{enumerate}
\end{lemma}

\begin{lemma}\label{lem:ind-ternary}
Let $\ind$ be a symmetric \prefree relation on some infinite set~$X$.
Assume that $\av \ind_C \dv$ and $\av \dv \ind_C \bv$.
Then, $\av \ind_C \bv \dv$ and $\dv \ind_C \bv \av$.
\end{lemma}
\begin{proof}
\Cf \cite[1.9]{adler}.
$\av \ind_C \bv \dv$ implies $\av \ind_{C \bv} \dv$,  
which implies $\av \ind_{C \bv} \bv \dv$, which, 
together with $\av \ind_C \dv$,  implies $\av \ind_C \bv \dv$.
\end{proof}

\begin{lemma}\label{lem:ind-sequence}
Let $\ind$ be a symmetric \prefree relation on some infinite set~$X$.
Let $\pair{I, \leq}$ be a linearly ordered set, $\Pa{\av_i: i \in I}$ be a
sequence of tuples in~$X^n$, and $C \subset X$.
Then, \tfae:
\begin{enumerate}
\item For every $i \in I$, we have $\av_i \ind_C (\av_j: j < i)$;
\item For every $i \in I$, we have $\av_i \ind_C (\av_j: j \neq i)$.
\end{enumerate}
\end{lemma}

\begin{proof}
Assume, for contradiction, that (1) holds, 
but $a_i \notind_C (\av_j: j \neq i)$, for some $i \in I$.
Since $\indmat$ satisfies finite character, \wloG $I = \set{1, \dotsc, m}$ 
is finite. 
Let $m'$ such that $i < m' \leq m$ is minimal with 
$\av_i \notind_C (\av_j: j \leq m' \et j \neq i)$; \wloG, $m = m'$.

Let $\dv := (a_j: j \neq i \et j < m)$.
By assumption, $\av_i \ind_C \dv$ and $\dv \av_i \ind_C \av_{m}$.
Then, by Lemma~\ref{lem:ind-ternary}, we have $\av_i \ind_C \dv \av_{m}$, absurd.
\end{proof}

\begin{definizione}
We say that a sequence $(\av_i: i \in I)$ satisfying one of the above
equivalent conditions is an \intro{independent sequence} over~$C$.
\end{definizione}

\begin{remark}
Let $(a_i: i \in I)$ be a sequence of elements of~$\monster$.
There is a clash with the previous definition of independence; 
more precisely, let $J := \set{i \in I: a_i \notin \mat(C)}$;
then, $(a_i : i \in I)$ is an independent sequence
over $C$ according to $\indmat$ iff all the $a_j$ are pairwise distinct for 
$j \in J$, and the set $\set{a_j : j \in J}$ 
is independent over $C$ according to~$\mat$. 
Hopefully, this will not cause confusion.
\end{remark}

\subsection{Definable matroids}

\begin{definizione}\label{def:definable}
Let $\phi(x,\y)$ be an $\Lang$-formula.
We say that $\phi$ is \intro{\xnarrow{}} if, for every $\bv$ and every~$a$, 
if $\monster \models \phi(a,\bv)$, then $a \in \mat(\bv)$.
We say that $\mat$ is \intro{definable} if, for every~$A$,
\[
\mat(A) = \bigcup \set{\phi(\monster,\av): \phi(x, \y) \text{ is \xnarrow},
\av \in A^n, n \in \N}.
\]
\end{definizione}

\begin{proviso*}
For the rest of the section, $\mat$ is a definable matroid.
\end{proviso*}

\begin{remark}
For every $A$ and every $\sigma \in \aut(\monster)$,
$\sigma(\mat(A)) = \mat(\sigma(A))$.
\end{remark}

\begin{lemma}\label{lem:cl-definable}
\begin{enumerate}
\item 
$\indmat$ satisfies the Invariance axiom:
if $A \indmat_B C$ and $\pair{A', B', C'} \equiv \pair{A, B, C}$, then
$A' \indmat_{B'} C'$.
\item
$\indmat$ satisfies the Strong Finite Character axiom:
if $A \notindmat_C B$, then there exist finite tuples $\av \subset A$, $\bv
\subset B$, and $\cv \subset C$, and a formula $\phi(\x, \y, \z)$ without
parameters, such that
\begin{itemize}
\item $\monster \models \phi(\av, \bv, \cv)$;
\item $\av' \notindmat_C B$ for all $\av'$ satisfying
$\monster \models \phi(\av', \bv, \cv)$.
\end{itemize}
\item 
For every $\av$, $B$, and~$C$,
if $\tp(\av/BC)$ is  finitely satisfied in~$B$, then $\av \indmat_B C$.
\item 
$\indmat$ satisfies the Local Character axiom:
for every $A$, $B$ there exists a subset $C$ of $B$ such that 
$\card C \leq \card T + \card A$ and $A \indmat_C B$.
\end{enumerate}
\end{lemma}
\begin{proof}
(1) is obvious.

(2) Assume that $A \notindmat_C B$.
Hence, there exists $\bv \in B^n$ independent over~$C$, such that $\bv$ is not
independent over~$A C$.
Hence, there exists $\av \subset A$ and $\cv \subset C$ finite tuples, such
that, \wloG, $b_1 \in \mat(\cv \av \tilde b)$, 
where $\tilde b := \pair{b_2,  \dotsc, b_n}$.
Let $\alpha(x, \tilde x, \y, \z)$ be an \xnarrow formula, such that
$\monster \models \alpha(b_1, \tilde b, \cv, \av)$.
If $\av' \subset \monster$ satisfies $\alpha(\bv, \cv, \av')$, then
$\av' \notindmat_C B$.

(3) and (4) follow as in ~\cite[2.3--4]{adler}.
Here is a direct proof of the Local Character axiom: let $A$ and $B$ be given.
Let $B' \subseteq B$ be a basis of $A B$ over~$A$,
$A' \subseteq A$ be a basis of~$A$,
and  $C \subseteq B$ be a basis of $B$ over~$B'$.
Notice that $C B'$ is a basis of~$A B$ and $A' B'$ is a set of generators
of $A B$; hence, by the Exchange Principle, $\card{C} \leq \card{A'} =
\rkmat(A) \leq \card A$.
Moreover, $A \indmat_{C} B$.
\end{proof}

\begin{definizione}
Let $\ind$ be a \prefree relation on $\monster$.
We say that $\ind$ is an independence relation on $\monster$ if it moreover
satisfies Invariance, Local Character, and
\begin{description}
\item[Extension:]
If $A \ind_C B$ and $D \supseteq B$, then there exists $A' \equiv_{BC} A$ such
that $A' \ind_C D$.
\end{description}
We also define the following axiom:
\begin{description}
\item[Existence:]
For any $A$, $B$, and $C$, there exists $A' \equiv_C A$ such that $A' \ind_C B$.
\end{description}
\end{definizione}

The following result follows from~\cite{adler}.

\begin{corollary}
If $\indmat$ satisfies either the Extension or the Existence axiom, then it is
an independence relation (and satisfies the Existence axiom).
\end{corollary}
\begin{proof}
See \cite[Thm.~2.5]{adler}.
\end{proof}

\begin{definizione}\label{def:existence}
$\mat$ satisfies Existence if:
\begin{itemize}
\item[] 
For every $a$, $B$, and $C$, if $a \notin \mat B$, 
then there exists $a' \equiv_B a$ such that $a' \notin \mat(B C)$.
\end{itemize}
Denote by $\aut(\monster/B)$ the set of automorphisms of $\monster$ which fix $B$ pointwise.
Denote by $\Xi(a/B)$ the set of conjugates of $a$ over $B$:
$\Xi(a/C) := \set{a^\sigma: \sigma \in \aut(\monster/B)}$.
\end{definizione}

\begin{lemma}\label{lem:cl-E}
\Tfae:
\begin{enumerate}
\item 
$\mat$ satisfies Existence.
\item
For every $a$, $B$, and $C$, if $\Xi(a/B) \subseteq \mat(BC)$, 
then $a \in  \mat(B)$.
\item 
For every $a$, $\bv$, and~$\cv$, if $a \notin \mat(\bv)$, then there exists
$a' \equiv_{\bv} a$ such that $a' \notin \mat(\bv \cv)$.
\item
For every $a$, $\bv$, and~$\cv$, and every \xnarrow formula $\phi(x, \y, \z)$,
if $\monster \models \phi (a', \bv, \cv)$ for every $a' \equiv_{\bv} a$,
then $a \in \mat(\bv)$.
\item
For every formula (without parameters) $\phi(x, \y)$ and every \xnarrow
formula $\psi(x, \y, \z)$, if 
$\monster \models \forall \y\ \exists \z\ \forall x\ \Pa{\phi(x, \y) \rightarrow
\psi(x, \y, \z)}$, then $\phi$ is \xnarrow.
\item 
For every $a$ and $B$, if $\rkmat(\Xi(a/B)$ is finite, then $a \in \mat(B)$.
\item
For every $a$ and $B$, 
if $\rkmat(\Xi(a/B) < \kappa$, then $a \in \mat(B)$.
\item
$\indmat$ is an independence relation.
\end{enumerate}
\end{lemma}

\begin{remark}
If $\mat$ satisfies Existence, then $\acl A \subseteq \mat A$.
\end{remark}

\begin{lemma}\label{lem:cl-Skolem-existence}
Assume that $\mat(A)$ is an elementary substructure of~$\monster$,
for every $A \subset \monster$.
Then, $\mat$ satisfies Existence, and 
therefore $\indmat$ is an independence relation.
Hence, if $T$ has \DSF and $\mat$ extends~$\acl$,
then $\mat$ is satisfies Existence.
\end{lemma}
\begin{proof}
Let $\Xi(a/B) \subseteq \mat(BC)$. We want to prove that $a \in \mat(B)$.
Let $B'$ and $C'$ be elementary substructures of~$\monster$,
such that $B \subseteq B' \subset \mat(B)$,
$B'C \subseteq C' \subset \mat(BC)$,
$\card{B'} < \kappa$, and $\card{C'} < \kappa$
($B'$~and $C'$ exist by hypothesis on~$\mat$).
By substituting $B$ with $B'$ and $C$ with~$C'$,
\wloG we can assume that $B \preceq C \prec \monster$.
By saturation, there exist an \xnarrow formula $\phi(x, \y, \z)$,
$\bv \subset B$, and $\cv \subset C$, such that 
$\Xi(a/B) \subseteq \phi(\monster, \bv, \cv)$.
Let $p := \tp(a/B)$, let $q \in S_1(C)$ be a heir of~$p$,
and $a'$ be a realisation of~$q$.
Since $\phi(x, \bv, \cv) \in p$, there exists $\bv' \in B$ such that
$\phi(x, \bv, \bv') \in q$.
Hence, $a' \in \mat(B)$; since $a' \equiv_B a$, $a \in \mat(B)$.
\end{proof}


\begin{definizione}
The \intro{trivial matroid} $\mat^0$ is given by $\mat^0(X) = X$ 
for every $X \subseteq \monster$.
$\mat^0$ is a definable matroid and satisfies Existence.
It induces the trivial \prefree relation $\ind[0]$, such that
$A \ind[0]_B C$ for every $A$, $B$, and~$C$.
Notice that $\ind[0]$ is an independence relation.
\end{definizione}

\begin{definizione}\label{def:existential}
We say that $\mat$ is an \intro{existential matroid} if $\mat$ is a definable
matroid, satisfies Existence, and is non-trivial (\ie, different from~$\mat^0$).
\end{definizione}

\begin{examples}
\begin{enumerate}
\item
Given $n \in \N$, the uniform matroid of rank $n$ is defined as:
$\mat^n(X) := X$, if $\card X <n$, or $\monster$ if $\card X \geq n$.
$\mat^n$ is a definable matroid, but does not satisfy Existence in general
(unless $n = 0$).
\item
Define $\id(X) := X$.
$\id$~is a definable matroid, but does not satisfy Existence in general.
The \prefree relation induced by $\id$ is given by
$A \ind[$\id$]_B C$ iff $A \cap C \subseteq B$.
\end{enumerate}
\end{examples}

\begin{remark}
Let $\monster'$ be another monster model of~$T$.  
We can define an operator $\mat'$ on $\monster'$ in the following way:
\[
\mat(X') := \bigcup \set{\phi(\monster', \av): \phi(x, \y) \text{ \xnarrow }
  \et \av' \subset X'}.
\]
Then, $\mat'$ is a definable matroid.
If $\mat$ satisfies existence, then $\mat'$ also satisfies existence.
\end{remark}

\begin{remark}\label{rem:general-closure}
Notice that the definitions of ``definable'' (\ref{def:definable}) and
``existential'' (\ref{def:existential} and \ref{def:existence}) 
make sense also for finitary closure operators  (and not only for matroids). 
\end{remark}
However, we will not need such more general definitions.

\begin{proviso*}
For the remainder of this section, $\mat$~is an existential matroid.
\end{proviso*}


Summarising, we have:\\
If $\mat$ is an existential matroid, then $\indmat$ is an independence
relation,  satisfying the strong finite character axiom.
In particular, if $\monster$ is a pregeometric structure,
then $\ind[$\acl$]\,$ is an independence relation.


\subsection{Dimension}

\begin{definizione}\label{def:dimension}
Given a set $V \subseteq \monster^n$, definable with parameters~$A$, the
\intro{dimension} of~$V$ (\wrt to the matroid~$\mat$) is given by 
\[
\dimat(V) := \max\set{ \rkmat(\bv/ A):  \bv \in X},
\]
with $\dimat(V) := - \infty$ iff $V = \emptyset$.
More generally, the dimension of a partial type $p$ with parameters $A$
is given by
\[
\dimat(p) := \max\set{ \rkmat(\bv/ A):  \bv \models p}.
\]
\end{definizione}

The following lemma shows that the above notion is well-posed:
in its proof, it is important that $\mat$ satisfies existence.
\begin{lemma}
Let $V$ be a  type-definable subset of~$\monster^n$.
Then, $\dimat(V) \leq n$, 
and $\dimat(V)$ does not depend on the choice of the parameters.
\end{lemma}

\begin{remark}
For every $d \leq n \in \N$, the set of complete types in $S_n(A)$ of
$\dimat$ equal to~$d$ is closed (in the Stone topology).
That is, $\dimat$ is continuous in the sense of~\cite[\S 17.b]{poizat85}.
\end{remark}


\begin{lemma}
Let $p$ be a partial type over~$A$.
Then,
\[
\dimat(p) := \min\set{\dimat(V): V \text{ is $A$-definable} \et V \in p}.
\]
Moreover, if $p$ is a complete type, then, 
for every $\bv \models p$, $\rkmat(\bv  /A) = \dimat(p)$.
\end{lemma}
\begin{proof}
Let $d :=  \dimat(p)$,
$e := \min\set{\dimat(V): V \text{ is $A$-definable}  \et V \in p}$, 
and $\bv \models p$, such that $d = \rkmat(\bv / A)$.
If $V \in p$, then $\bv \in V$, and therefore
\[
e \geq \dimat(V) \geq \rkmat(\bv / A) = d.
\]
For the opposite inequality, first assume that $p$ is a complete type.
\Wlog $\tilde b := \pair{b_1, \dotsc, b_d}$ are
$\mat$-independent over~$A$, and therefore $b_i \in \mat(A \tilde b)$ for
every $i = d + 1, \dotsc, n$.
For every $i \leq n$, $\phi_i(x, \y, \z)$ be an \xnarrow formula such that
$\monster \models \phi(b_i, \tilde b, \av)$ (where $\av \subset A$), 
$\phi(\x, \y, \z) := \bigcap_{i = 1}^n \phi_i(x_i, x_1, \dotsc, x_d, \z)$,
and $V := \phi(\monster^n, \monster^d, \av)$.
Then, for every $\bv' \in V$, $\rkmat(\bv' / A) \leq d$, and therefore
$\dimat(V) \leq d$.
Moreover, $\bv \in V$, hence $V \in p$, and therefore $e \leq d$.

The general case when $p$ is a partial type follows from the complete case,
the fact that the set of complete types extending $p$ is a closed (and hence
compact) subset of $S_n(A)$, and the previous remark.
\end{proof}

\begin{remark}
$\dimat(\monster^n) = n$.  
Moreover, $\dimat$ is monotone: if $U \subseteq V \subseteq \monster^n$, 
then $\dimat(U) \leq \dimat(V)$.
\end{remark}

\begin{definizione}
Given $p \in S_n(B)$, $q \in S_n(C)$, with $B \subseteq C$, we say that $q$ is
a non-forking extension of~$p$ (w.r.t.~$\mat$), and write $p \sqsubseteq q$, if
$q$ extends $p$ and $\dimat(q) =\dimat(p)$.
We write $q \indmat_B C$ if $q \rest_B \sqsubseteq q$.
\end{definizione}

\begin{remark}
Let $B \subseteq C$ and $q \in S_n(C)$.
Then, $q \indmat_B C$ iff, for some (for all) $\av$ realising~$q$,
$\av \indmat_B C$.
\end{remark}

\begin{remark}
Let $p \in S_n(B)$ and $B \subseteq C$.
Then, for every $q \in S_n(C)$ extending~$p$, $\dimat(q) \leq \dimat(p)$.
Moreover, there exists $q \in S_n(C)$ which is a non-forking extension of~$p$.
\end{remark}

\begin{lemma}\label{lem:ind-heir}
Let $\indf$ be Shelah's forking relation on~$\monster$.
Then, for every $A$, $B$, and $C$ subsets of~$\monster$,
if $A \indf_B C$, then $A \indmat_B C$.
In particular, if
$\K \prec \monster$, $\K \subseteq C$, and $q \in S_n(C)$, and
$q$ is either a heir or a coheir of $q \rest_{\K}$, then $q \indmat_{\K} C$.
\end{lemma}

\begin{proof}
The fact that $\indf$ implies $\indmat$ is a particular case of
\cite[Remark~1.27]{adler}.
For the case when $q$ is a heir of $p := q \rest_{\K}$, see also
\cite[Remark~2.3]{adler}.
\end{proof}
 
\begin{corollary}
Assume that $T$ is super-simple and $p \in S_n(A)$ 
for some $A \subseteq \monster$.
Then, $SU(p) \geq \dimat(p)$, where $SU$ is the $SU$-rank (see~\cite{wagner}).
\end{corollary}

\begin{remark}\label{rem:forking-closed}
Given $B \supseteq A$, let $N_n(B/A)$ be the set of all $n$-types over $B$
that do not fork over~$A$.
$N_n(B,A)$ is closed in~$S_n(B)$.
The same is true for any independence relation~$\ind$, instead of~$\indmat$.
\end{remark}

\begin{lemma}
For every complete type~$p$, $\dimat(p)$ is the maximum of the cardinalities
$n$ of chains of complete types
$p = q_0 \subset q_1 \subset \dotsc \subset q_n$,
such that each $q_{i + 1}$ is a forking extension of~$q_i$.
\end{lemma}
\begin{proof}
Let $A$ be the set of parameters of~$p$, and $\bv \models p$.
Let $d := \dimat(p)$; \wloG, $\tilde b := \pair{b_1, \dotsc, b_d}$ are
independent over~$A$.
For every $i \leq n$ let $A_i := A b_1 \dots b_i$, and 
$q_i := \tp(\bv / A_i)$.
Then, $p = q_0 \subset \dots \subset q_d$, and each $q_{i+1}$ is a forking
extension of~$q_i$.

Conversely, assume that $p = q_0 \subset \dots \subset q_n$, and each
$q_{i+1}$ is a forking extension of~$q_i$, and $A_i$ be the set of parameters
of~$q_i$.
\begin{claim}
For every $i \leq n$, $\dimat(q_{n-1}) \geq i$; in particular, 
$\dimat(p) \geq n$.
\end{claim}
By induction on~$i$.
The case $i = 0$ is clear.
Assume that we have proved the claim for~$i$, we want to show that it holds
for $i + 1$.
Since $q_i$ is a forking extension of 
$q_{i+1}$, $\dimat(q_i) > \dimat(q_{i+1})$, and we are done.
\end{proof}

\begin{lemma}\label{lem:dim-rk}
Let $V \subseteq \monster^n$ be non-empty and definable with parameters~$\av$.
Then, either $\dimat(V) = 0 = \rkmat(V /\av)$, 
or $\dimat(V) > 0$ and $\rkmat(V) \geq \kappa$.
\end{lemma}

\begin{lemma}
A formula $\phi(x, \y)$ is \xnarrow iff, for every $\bv$,
$\dimat\Pa{\phi(\monster, \bv)} = 0$.
\end{lemma}

\begin{lemma}\label{lem:type-def}
Let $\phi(x, \y)$ be a formula without parameters, and $\av \in \monster^n$.
Then, $\dimat(\phi(\monster, \av)) = 0$ iff there exists an \xnarrow formula
$\psi(x, \y)$ such that $\forall x\ \Pa{\phi(x,\av) \rightarrow \psi(x, \av)}$.
Therefore, define
\[\begin{aligned}
\Gamma_\phi(\y) &:= \set{\neg \theta(\y):  \theta(\y) 
\text{ formula without  parameters s.t. }\\
& \qquad \forall \av\ \Pa{\theta(\av) \rightarrow \dimat(\phi(\monster, \av)) = 0}},\\
U_\phi^1 &:= \set{\av \in \monster^n: \dimat(\phi(\monster, \av)) = 1}.
\end{aligned}\]
Then, $U_\phi^1 = \set{\av \in \monster^n: \monster \models \Gamma_\phi(\av)}$,
and in particular $U_\phi^1$ is type-definable (over the empty set).

More generally, let $k \leq n$, $\x := \pair{x_1, \dotsc, x_n}$, and
$\phi(\x, \y)$ be a formula without parameters.
Define 
\[
U_\phi^{\geq k} := \set{\av \in \monster^m: 
\dimat(\phi(\monster^n, \av)) \geq  k}.
\]
Then, $U_\phi^{\geq k}$ is type-definable over the empty set.
\end{lemma}


\begin{lemma}[Fibre-wise dimension inequalities]\label{lem:cl-function}
$U \subseteq \monster^{m_1}$, $V \subseteq \monster^{m_2}$, and
$F: U \to V$ be definable, with parameters~$C$.
Let $X \subseteq U$ and $Y \subseteq V$ be type-definable, such that 
$F(X) \subseteq Y$.
Define $f := F \rest X: X \to Y$.
For every $\bv \in Y$, let $X_{\bv} := f^{-1}(\bv) \subseteq X$, 
and $m := \dimat(Y)$.
\begin{enumerate}
\item 
If, for every $\bv \in Y$, $\dimat(X_{\bv}) \leq n$, 
then $\dimat(X) \leq m + n$.
\item
If $f$ is surjective and, for every $\bv \in Y$, $\dimat(X_{\bv}) \geq n$, 
then $\dimat(X) \geq m + n$.
\item 
If $f$ is surjective, then $\dimat(X) \geq m$.
\item
If $f$ is injective, then $\dimat(X) \leq m$.
\item
If $f$ is bijective, then $\dimat(X) = m$.
\end{enumerate}
\end{lemma}
\begin{proof}
1) 
Assume, for contradiction, that $\dimat(X) > m + n$.
Let $\av \in X$ such that $\rkmat(\av / C) > m + n$, and $\bv := F(\av)$.
Since $\av \in X_{\bv}$, and $X_{\bv}$ is type-definable with parameters 
$C \bv$, $\rkmat(\av / \bv C) \leq n$.
Hence, by Lemma~\ref{lem:Lascar}, $\rkmat(\av / C) \leq \rkmat(\av \bv /C)
\leq m + n$, absurd.

2)
Let $\bv \in Y$ such that $\dimat(\bv / C) = m$.
Let $\av \in X_{\bv}$ such that $\dimat(\av / \bv C) \geq n$.
Then, by Lemma~\ref{lem:Lascar}, $\rkmat(\av \bv / C) \geq m + n$.
However, since $\av = F(\bv)$, $\av \subset \mat(\bv C)$, and therefore
$\rkmat(\bv / C) = \rkmat(\av \bv / C) \geq m + n$.

(3) follows from (2) applied to $n = 0$.
The other assertions are clear.
\end{proof}


\begin{lemma}\label{lem:cl-equality}
Let $\mat'$ be another existential matroid on~$\monster$.
\Tfae:
\begin{enumerate}
\item $\mat \subseteq \mat'$;
\item $\rkmat \geq \RK^{\mat'}$;
\item $\dimat \geq \dim^{\mat'}$ on definable sets;
\item $\dimat \geq \dim^{\mat'}$ on complete types;
\item for every definable set $X\subseteq \monster$, 
if $\dim^{\mat'}(X) = 0$, then $\dimat(X) = 0$.
\end{enumerate}
\Tfae:
\begin{enumerate}
\item $\mat = \mat'$;
\item $\rkmat = \RK^{\mat'}$;
\item $\dimat = \dim^{\mat'}$ on definable sets;
\item $\dimat = \dim^{\mat'}$ on complete types;
\item for every definable set $X \subseteq \monster$, 
$\dimat(X) = 0$ iff $\dim^{\mat'}(X) = 0$.
\end{enumerate}
\end{lemma}

We will show that, for many interesting theories, there is at most
one existential matroid.

Define $\TR$ to be the theory of rings without zero divisors, in the language
of rings $\LR := (0, 1, + , \cdot)$.

\begin{definizione}
If $\K$ expands a ring without zero divisors, define
$F : \K^4 \to \K$ the function, definable without parameters in the language
$\LR$,
\[
\pair{x_1, x_2, y_1, y_2} \mapsto
\begin{cases}
t & \text{if } y_1 \neq y_2 \et t \cdot (y_1 - y_2) = x_1 - x_2;\\
0 & \text{ otherwise.}
\end{cases}
\]
Notice that $F$ is well-defined, because in a ring without zero divisors, if
$y_1 \neq y_2$, then, for every~$x$, there exists at most one $t$ such that
$t \cdot (y_1 - y_2) = x$.
\end{definizione}

\begin{lemma}[{\cite[1.18]{DMS}}]\label{lem:cl-ring-function}
Assume that $T$ expands $\TR$.
Let $A \subseteq \monster$ be definable.
Then, $\dimat(A) = 1$ iff  $\monster = F(A^4)$.
\end{lemma}
\begin{proof}
Same as~\cite[1.18]{DMS}.
Assume for contradiction that 
$\dimat(A) = 1$, but there exists $c \in \monster \setminus F(A^4)$.
Since $c \notin F(A^4)$, the function 
$\pair{x_1, x_2} \mapsto c \cdot x_1 + x_2: A^2 \to \monster$ is injective.
Hence, by Lemma~\ref{lem:cl-function}, $\dimat(\monster) \geq \dimat(A^2) = 2$, absurd.

Conversely, by Lemma~\ref{lem:cl-function} again, if $f(A^4) = \monster$, 
then $\dim(A) = 1$.
\end{proof}

\begin{thm}\label{thm:cl-unique}
If $T$ expands $\TR$, then $\mat$ is the only existential matroid on~$\monster$.
If $S$ is a definable subfield of $\monster$ of dimension~$1$,
then $S = \monster$.
\end{thm}
\begin{proof}
Let $A \subseteq \monster$ be definable.
By the previous lemma, $\dim (A) = 1$ iff $F(A^4) = \monster$.
Since the same holds for any existential matroid
$\mat'$ on~$\monster$, we conclude that, for every definable set $A\subseteq
\monster$, $\dimat(A) = 0$ iff $\dim^{\mat'}(A) = 0$, 
and hence $\dimat = \dim^{\mat'}$.

Given $S$ a subfield of~$\monster$, $F(S^4) = S$.
Hence, if $\dimat(S) = 1$, then $S = \monster$.
\end{proof}


\begin{example}
In the above theorem, we cannot drop the hypothesis that $T$ expands~$\TR$.
In fact, let $\monster_0$ be an infinite connected graph, such that 
$\monster_0$ is a monster model, and $\acl$ is a matroid in $\monster_0$
(\eg, $\monster_0$ equal to a monster model of the theory of random graphs).
Let $\monster$ be the disjoint union of $\kappa$ copies of $\monster_0$:
notice that $\monster$ is a monster model.
For every $a \in \monster$, let $\mat(a)$ be the connected component
of $\monster$ containing~$a$ (it is a copy of $\monster_0$),
and $\mat(A) := \bigcup_{a \in A} \mat(a)$.
Then, $\acl$ and $\mat$ are two different existential matroids on~$\monster$.
\end{example}

\begin{example}
In Lemma~\ref{lem:cl-ring-function} and Theorem~\ref{thm:cl-unique} 
we cannot even relax the hypothesis to 
``$T$ expands the theory of a vector space''. 
In fact, let $\F$ be an ordered field, considered as a vector space over
itself, in the language $\pair{0, 1, +, <, \lambda_c}_{c \in \F}$, and let $T$
be its theory.
Let $\Td$ be the theory of dense pairs of models of~$T$.
\cite[5.8]{DMS} show that $\Td$ has elimination of quantifiers, and $\acl$ is
a matroid on~$\Td$.
However, as the reader can verify, the small closure $\scl$ is another
existential matroid on~$\Td$ (\cf \S\ref{subsec:small-closure}),
and it is different from~$\acl$.
\end{example}

\begin{corollary}
If $\monster$ expands a field, then $\monster$ must be a perfect field.
In particular, the theory of separably closed and non-algebraically closed
fields, and the theory of differentially closed fields of finite
characteristic do not admit an existential matroid.
\end{corollary}
\begin{proof}
Cf.\ \cite[1.6]{dries89}.
If $\monster$ is not perfect, then $\monster ^p$ is a proper definable
subfield of~$\monster$, where $p := \charact(\monster)$, 
and therefore $\dimat(\monster^p) = 0$.
However, the map $x \mapsto x^p$ is a bijection from $\monster$ to
$\monster^p$; therefore, $\dimat(\monster) = 0$, absurd.
\end{proof}

\begin{corollary}
Let $\mat'$ be a non-trivial definable matroid on some monster
model~$\monster'$.
Assume that $\monster'$ expands a model of~$\TR$.
Then, \tfae:
\begin{enumerate}
\item $\mat'$ is an existential matroid;
\item for every formula (without quantifiers) $\phi(x, \y)$,
$\phi$ is \xnarrow iff, for y
every~$\bv$, $F^4(\phi(\monster', \bv) \neq \monster'$.
\end{enumerate}
\end{corollary}
\begin{proof}
$(1 \Rightarrow 2)$ is clear.

$(2 \Rightarrow 1)$ follows from Lemma \ref{lem:cl-E}-5.
\end{proof}

\begin{lemma}
Let $\K$ be a ring without zero divisors definable in~$\monster$, of dimension
$n \geq 1$.
Let $\F \subseteq \K$ be a definable subring such that $\F$ is a skew field.
If $\dimat(\K) = n$, then $\K = \F$.
\end{lemma}
\begin{proof}
Assume, for contradiction, that there exists $c \in \K \setminus \F$.
Define $h: \F \times \F \to \K$, $h(x,y) := x + c y$.
Since $c \notin \F$ and $\F$ is a skew field, $h$ is injective.
Thus, $2 n = \dim(\F^2) \leq \dim(\K) = n$, a contradiction.
\end{proof}

\begin{corollary}
Let $\K \subseteq \monster^n$ be a definable field, such that
$\dimat(\K) \geq 1$.
Then, $\K$~is perfect.
\end{corollary}
The assumption that $\dimat(\K) \geq 1$ is necessary: 
non-perfect definable fields of dimension $0$ can exist.
\begin{proof}
Let $p := \charact \K$, and
$\phi: \K \to \K$ be the Frobenius automorphism $\phi(x)= x^p$.
Since $\phi$ is injective, $\dimat(\K^p) = \dimat(\K)$, 
and therefore $\K^p = \K$.
\end{proof}

\begin{example}\label{ex:group}
Let $\lambda$ be an ordinal, which is an ordinal power of $\omega$ (\eg,
$\lambda = 1$, $\lambda = \omega, \dotsc$).  
Let $\Gm$ be a monster model of a superstable group, such that
$\SU(\Gm) = \lambda$, where $\SU$ is Lascar's rank.
For every $a$ and~$B$, define $a \in \mat(B)$ iff $\SU(a/B) < \lambda$.  
Then, $\mat$~is an existential matroid.
If $X$ is a definable subset of~$\Gm$, then $\dimat(X) = 1$ iff $X$ is
generic, that is
finitely many left translates of $X$ cover~$\Gm$.
\end{example}
\begin{proof}
See~\cite{poizat87}.
%
%
\end{proof}

\begin{example}
Let $\K$ be a monster differentially closed field, and $p \geq 0$ be its
characteristic.
If $p = 0$, then $\K$ is superstable, and $\SU(\K) = \omega$; hence, by  the
previous example, there exists a (unique) existential matroid $\mat$ on~$\K$.
It is easy to see that, if $A$ is a differential subfield of~$\K$ and
$b \in \K$, then $b \in \mat(A)$ iff $b$ is differential-algebraic over~$A$
(that is, iff $b, \de b, \de^2 b, \dotsc$ are algebraically dependent over~$A$);
see \cite{wood} and~\cite[2.25]{dries89}.
On the other hand, if $p > 0$, then there is no existential matroid on~$\K$,
because $\K$ is not perfect.
\end{example}

\begin{definizione}
Let $X \subseteq \K^n$ any $Y \subseteq \K^m$ be definable.
Let $f: X \app Y$ be a definable application
(\ie, a multi-valued partial function), with graph~$F$.
For every $x \in X$, let $f(x) := \set{y\in Y: \pair{x,y} \in F}\subseteq Y$.
Such an application $f$ is a \intro{\Zapplication{}} if, for every $x \in X$, 
$\dimat\Pa{f(x)} \leq 0$.
\end{definizione}

\begin{remark}\label{rem:Zapplication}
Let $A \subseteq \K$, and $b \in \K$.
Then, $b \in \mat(A)$ iff there exists an $\emptyset$-definable
\Zapplication $f: \K^n \app \K$ and $\av \in A$,
such that $b \in f(\av)$.
Moreover, if $\cv \in \K^n$, then $b \in \mat(A\cv)$ iff there exists
an $A$-definable \Zapplication $f: \K^n \to \K$, such that $b \in f(\cv)$.
\end{remark}

\begin{definizione}
We say that $\dimat$ is \intro{definable} if,
for every $X$ definable subset of $\monster^m \times \monster^n$, the set 
$\set{\av \in \monster^m: \dimat(X_{\av}) =  d}$ is definable.
\end{definizione}

\begin{lemma}
\Tfae:
\begin{enumerate}
\item 
$\dimat$ is definable;
\item 
for every $X$ definable subset of $\monster^m \times \monster$,
the set $X^{1, 1} := \set{\av \in \monster^m: \dimat(X_{\av}) = 1}$ is also
definable;
\item
for every $k \leq n$, every~$m$, and every $X$ definable subset of 
$\monster^m \times \monster^n$,
the set $X^{n, k} := \set{\av \in \monster^m: \dimat(X_{\av}) = k}$ is also
definable, with the same parameters as~$X$.
\end{enumerate}
\end{lemma}
\begin{proof}
$(3 \Rightarrow 1 \Rightarrow 2)$ is obvious.

$(2 \Rightarrow 1)$ is obvious.
We will prove by induction on $n$ that, for every
$Y$ definable subset of $\K^n \times \K^m$, the set
$Y^{n, \geq k} := \set{\av \in \monster^m: \dimat(X_{\av}) \geq k}$ is
definable. 
The case $k = 0$ is clear.
The case $k = 1$ follows from the assumption and the observation that, for
every $Z$ definable subset of~$\K^n$, $\dimat(Z) \geq 1$ iff 
$\dimat(\theta(Z)) \geq 1$ for some $\theta$ projection from $\K^n$ to a
coordinate axis.
The inductive step follows from the fact that
\[
X^{n, \geq k} = 
\Pa{\Pi^{n+m}_{n+ m -1}(X)}^{n - 1, \geq k} \cup
\Pa{X^{n + m -1, \geq 1}}^{n-1, \geq k - 1}.
\qedhere
\]

$(1 \Rightarrow 3)$ 
Let $X \subseteq \K^{n + m}$ be definable with parameters~$A$.
Then, $X^{n, k}$ is $\monster$-definable, by assumption.
Moreover, by Lemma~\ref{lem:type-def}, $X^{n, k}$ is type-definable
over~$A$, and therefore invariant under automorphisms that fix $A$ pointwise.
Hence, by Beth's definability theorem, $X^{n, k}$ is definable over~$A$.
\end{proof}


\begin{remark}
If $T$ expands $\TR$, then $\dimat$ is definable.
\end{remark}


\subsection{Morley sequences}

Most of the results of this subsection remain true for an arbitrary
independence relation $\ind$ instead of~$\indmat$.

\begin{definizione}
Let $C \subseteq B$, $p(\x) \in S_n(B)$, and $\pair{I, \leq}$ be a linear order.
A \intro{Morley sequence} over~$C$ indexed by $I$ in $p$
is a sequence $(\av_i: i \in I)$ of tuples
in $\monster^n$, such that $(\av_i: i \in I)$ is order-indiscernibles over~$B$
and independent over~$C$, and every $\av_i$ realises~$p(\x)$.\\
A Morley sequence over $C$ is a Morley sequence over $C$ in some $p \in S_n(C)$.
A Morley sequence in~$p$ is a Morley sequence over $B$ in~$p$.
\end{definizione}

\begin{lemma}\label{lem:Morley-seq}
Let $\pair{I, \leq}$ be a linear order, with $\card{I} < \kappa$.
Let $p(\x) \in S_n(C)$.
Then, there exists a Morley sequence over $C$ indexed by $I$ in~$p(\x)$.
If moreover $\bv \indmat_C \dv$, then there exists a Morley sequence 
$(\av_i: i \in I)$ over~$C$ indexed by~$I$ in~$p(\x)$, such that
$(\bv \av_i: i \in I)$ is order-indiscernibles over $C \dv$ and,
for every $i \in I$, $\bv \av_i \indmat_C \dv (\av_j: i \neq j \in I)$.
\end{lemma}
\begin{proof}
Let $(\x_i: i \in I)$ be a sequence of $n$-tuples of variables.
Consider the following set of $C$-formulae:
\[
\Gamma_1(\x_i: i \in I) := \bigwedge_{i \in I} p(\x_i) \et
\bigwedge_{i \in I} \x_i \indmat_C (\x_j: j < i).
\]
First, notice that, by Remark~\ref{rem:forking-closed}
$\Gamma_1$ \emph{is} a set of formulae.

Consider the following set of $C$-formulae:
\begin{multline*}
\Gamma_2(\x_i: i \in I) := \Gamma_1(\x_i: i \in I) \et \\
(\x_i: i \in I) \text{ is an order-indiscernible sequence of  over } C.
\end{multline*}

By \cite[1.12]{adler}, $\Gamma_2$ is consistent.

We give an alternative proof of the above fact, which does not use 
Erd\"os-Rado.
\begin{claim}
$\Gamma_1$ is consistent.
\end{claim}
It is enough to prove that $\Gamma_1$ is finitely satisfiable; hence, \wloG
$I = \set{0, \dotsc, m}$ is finite.
Let $\av_0$ be any realisation of~$p(\x)$.
Let $\av_1 \elem_C \av_0$ such that $\av_i \indmat_C \av_0$, \dots,
let $\av_m \elem_C \av_0$ such that $\av_m \indmat_C \av_0 \dots \av_{m - 1}$.

By Ramsey's Theorem, $\Gamma_2$ is also consistent.

Since $\card I < \kappa$, there exists a realisation $(\av_i: i \in I)$ 
of~$\Gamma_2$.
Then, by Lemma~\ref{lem:ind-sequence}
$(\av_i: i \in I)$ is a Morley sequence in~$p(\x)$ over~$C$.

If moreover $\bv$ and $\dv$ satisfy $\bv \indmat_C \dv$, let
$q(\x, \y, \z)$ be the extension of $p(\x)$ to $S^*(C \bv \dv)$ satisfying
$\y = \bv$ and $\z = \dv$.
Let $(\av_i \bv \dv: i \in I)$ be a Morley sequence in $q(\x, \y, \z)$.
By Lemma~\ref{lem:ind-ternary}, for every $i \in I$ we have
$\bv \av_i \indmat_C \dv (\av_j: i \neq j \in I)$.
\end{proof}

\begin{lemma}
A type $p \in S_n(A)$ is \intro{stationary} if, for every $B \supseteq A$,
there exists a unique $q \in S_n(B)$ such that $p \sqsubseteq q$.
\end{lemma}

\begin{remark}
Let $p \in S_n(A)$.
If $\dimat(p) = 0$, then $p$ is stationary iff $p$ is realised in $\dcl(A)$.
\end{remark}

Hence, unlike the stable case, if $\mat \neq \acl$, then there are types over
models which are not stationary.

\begin{lemma}\label{lem:indipendent-morley-seq}
Let $C \supseteq B$, and $q \in S_n(C)$ such that $q \indmat_B C$.
Let $(\av_i: i \in I)$ be a sequence of realisations of $q$ independent
over~$C$. 
Then, $(\av_i: i \in I)$ is also independent over~$B$.
If moreover $q$ is stationary, then
\begin{enumerate}
\item $(\av_i: i \in I)$ is a totally indiscernible set over~$C$,
and in particular it is a Morley sequence for $q$ over~$B$.
\item
If $(\av': i \in I)$ is another sequence of realisations of $q$
independent over~$C$, then $(\av_i: i \in I) \elem_C (\av'_i: i \in I)$.
\end{enumerate}
\end{lemma}
\begin{proof}
Standard proof.
More precisely, for every  $i \in I$, let $\dv_i := (a_j: i \neq j \in I)$.
By assumption, $\av_i \indmat_C \dv_i$, and, since $q \indmat_B C$,
$\av_i \indmat_B C$, and therefore $\av_i \indmat_B \dv_i$, proving that 
$(\av_i: i \in I)$ is independent over~$B$.

Let us prove Statement~(2).
By compactness, \wloG $I = \set{1, \dotsc, m}$ is finite.
Assume, for contradiction, that $(\av: i \leq m) \not\elem_C (\av': i \leq m)$;
by induction on~$m$, we can assume 
that $(\av_i: i \leq m-1) \elem_C (\av'_i: i \leq m-1)$, and therefore, \wloG,
that $\av_i = \av'_i$ for $i= 1, \dotsc, m - 1$.
However, since $q$ is stationary, $\av_m \elem_C \av'_m$,
$\av_m \indmat_C (\av_i: i \leq m -1)$,  and 
$\av'_m \indmat_C (\av_i: i \leq m -1)$,  
we have that 
$\av_m \equiv_{C (\av_i: i \leq   m-1)} \av'_m$, absurd.

Finally, it remains to prove that the set $(\av_i: i \in I)$ is
totally indiscernible over~$C$.
If $\sigma$ is any permutation of~$I$, then
$(\av_{\sigma(i)}: i \in I)$ is also a sequence of realisations of $q$
independent over~$C$, and therefore, by Statement~(2),
$(\av_{\sigma(i)}: i \in I) \elem_C (\av_i: i \in I)$.
\end{proof}


\begin{corollary}
Assume that there is a definable linear ordering on~$\monster$.
Then,  $p \in S_n(A)$ is stationary iff $p$ is realised in $\dcl(A)$.
Hence, if $\dimat(p) > 0$, every non-forking extension of $p$ is not stationary.
\end{corollary}
Contrast the above situation to the case of stable theories, where instead
every type has at least one stationary non-forking extension.
\begin{proof}
Assume that $p$ is stationary, but, for contradiction, that $\dimat(p) > 0$.
Then, there is a  Morley sequence in $p$ with at least two elements $\av_0$
and $\av_1$.
Since $\dimat(p) > 0$, $\av_0 \neq \av_1$.
By Lemma~\ref{lem:indipendent-morley-seq}, $\av_0$ and $\av_1$ are
indiscernibles over~$A$, absurd.
\end{proof}

\begin{corollary}\label{cor:ind-morley}
Let $B \subseteq C$ and $q \in S_n(C)$.
Then, \tfae:
\begin{enumerate}
\item $q \indmat_B C$;
\item there exists an infinite sequence 
of realisations of $q$ that are independent over~$B$;
\item every sequence $(\av_i: i \in I)$ of realisations of $q$ that are
independent over~$C$ are independent also over~$B$;
\item there exists an infinite Morley sequence in $q$ over~$B$.
\end{enumerate}
\end{corollary}
\begin{proof}
\Cf \cite[1.12--13]{adler}.
\begin{enumerate}
\item[($1 \Rightarrow 3)$]
Let $(\av_i: i \in I)$ be a sequence of realisations of~$q$ independent 
over~$C$.
For every $i \in I$, let $\dv_i := (\av_j: i \neq j \in I)$.
Since $\av_i \ind_C \dv_i$ and $\av_i \ind_B C$, we have $\av_i \ind_B \dv_i$.

\item[($3 \Rightarrow 4$)]
Let $(\av_i: i \in I)$ be an infinite Morley sequence in $q$ over~$C$: such
sequence exists by Lemma~\ref{lem:Morley-seq} (or by \cite[1.12]{adler}).
Then, $(\av_i: i \in I)$ is independent also over $B$, and hence a Morley
sequence for $q$ over~$B$.

\item[($4 \Rightarrow 2$)] 
is obvious.

\item[($2 \Rightarrow 1$)]
Choose $\lambda < \kappa$ a regular cardinal large enough.
Let $(\av'_i: i < \omega)$ be a sequence of realisations of~$q$ independent 
over~$C$.
By saturation, there exists $(\av_i: i < \lambda)$ a sequence of realisations
of~$q$ independent over~$C$.
By Local Character, and since $\lambda$ is regular,
there exists $\alpha < \lambda$ such that
$\av_\alpha \indmat_{B \dv} C$, where $\dv := (\av_i: i < \alpha)$.
Since moreover $\av_\alpha \indmat_B \dv$, we have $\av_\alpha \indmat_B C$,
and therefore $q \indmat_B C$.
\qedhere
\end{enumerate}
\end{proof}

\subsection{Local properties of dimension}
In this subsection, we will show that the dimension of a set can be checked
locally: what this means precisely will be clear in \S\ref{sec:dmin}, where
the results given here will be applied to a ``concrete'' situation.

\begin{definizione}
A quasi-ordered set $\pair{I, \leq}$ is a \intro{directed set} if every
pairs of elements of $I$ has an upper bound.
\end{definizione}

\begin{lemma}\label{lem:order-ind}
Let $\pair{I, \leq}$ be a directed set, definable in $\monster$ with
parameters~$\cv$.
Then, for every $\av \in I$ and $\dv \subset \monster$ there exists $\bv \in I$
such that $\bv \geq \av$  and $\dv \av \indmat_{\cv} \bv$.
\end{lemma}
\begin{proof}
Fix $\av \in I$ and $\dv \subset \monster$, and assume, for contradiction,
that every $\bv \geq \av$ satisfies $\dv \av \notindmat_{\cv} \bv$.

\Wlog, $\cv = \emptyset$.
Let $\lambda$ be a large enough cardinal; at the price of increasing $\kappa$
if necessary, we may assume that $\lambda < \kappa$.
By Lemma~\ref{lem:Morley-seq},
there exists a Morley sequence $(\dv'\av'_i: i < \lambda)$
in $\tp(\dv\av/ \emptyset)$ over~$\emptyset$.
Consider the following set of formulae over 
$\set{\av'_i: i < \lambda}$:
\[
\Lambda(\x) := \set{\x \in I, \x \geq \av'_i : i < \lambda}.
\]
Since $\pair{I, \leq}$ is a directed set, $\Lambda$ is consistent: let 
$\bv \in I$ be a realisation of~$\Lambda$.
By Erd\"os-Rado's Theorem, there exists a Morley sequence  
$(\dv_i \av_i: i < \omega)$ in $\tp(\dv\av/ \emptyset)$ over~$\emptyset$,
such that all the $\dv_i \av_i$ satisfy the same type $q(\x, \y)$ over~$\bv$.
Therefore, by Corollary~\ref{cor:ind-morley}, $q \indmat \bv$, and in
particular $\av_0 \indmat \bv$.
Since $\av_0 \equiv \av$, there exists $\bv' \geq \av$ such that
$\av \elem \bv'$, a contradiction.
\end{proof}


\begin{lemma}\label{lem:neighbourhood-abstract}
Let $X \subseteq \monster^n$ be definable with parameters~$\cv$
and $\Pa{U_{\bar t}}_{\bar t \in I}$ be a family of subsets of $\monster^n$,
such that each $U_{\bar t}$ is definable with parameters $\bar t \cv$.
Let $d \leq n$, and assume that, for every $\av \in X$ there exists $\bv \in I$
such that $\av \in U_{\bv}$, $\av \mind_{\cv} \bv$, 
and $\dimat(X \cap U_{\bv}) \leq d$.
Then, $\dimat(X) \leq d$.
\end{lemma}
\begin{proof}
Assume, for contradiction, that $\dimat(X) > d$; let $\av \in X$ such
that $\rkmat(\av/\cv) > d$.
Choose $\bv$ as in the hypothesis of the lemma; then, 
$\rkmat(\av / \bv\cv) > d$, absurd.
\end{proof}

\begin{lemma}\label{lem:cl-U}
Let $I \subseteq \monster^n$ be definable  and $<$ be a definable linear
ordering on~$I$.
Let $\Pa{X_{\bv}}_{\bv \in I}$ be a definable increasing family of subsets
of~$\K^m$ and $X := \bigcup_{\bv \in I} X_{\bv}$.
Let $d \leq m$, and assume that, for every $\bv \in I$, 
$\dimat(X_{\bv}) \leq d$.
Then, $\dimat(X) \leq d$.
\end{lemma}
\begin{proof}
Let $\cv$ be the parameters used to define $I$, $<$, and 
$\Pa{X_{\bv}}_{\bv \in I}$.
Let $\av \in X$ such that $\rkmat(\av/\cv) = \dimat(X)$.
Let $\bv \in I$ such that $\av \in X_{\bv}$.
Choose $\av', \bv' \subset \monster$ such that
$\av' \bv' \equiv_{\cv} \av \bv$ and $\av' \bv' \indmat_{\cv} \av \bv$.
\Wlog, $\bv' \geq \bv$; hence, $\av \in X_{\bv'}$ and 
\[
d \geq \dimat(X_{\bv'}) \geq \rkmat(\av/ \cv \bv') = \rkmat(\av/\cv) =
\dimat(X). \qedhere
\]
\end{proof}


We can extend the above lemma to directed families.

\begin{lemma}
Let $\pair{I, \leq}$ be a definable directed set.
Let $\Pa{X_{\bv}}_{\bv \in I}$ be a definable increasing family of subsets
of~$\monster^m$ and $X := \bigcup_{\bv \in I} X_{\bv}$.
Let $d \leq m$, and assume that, for every $\bv \in I$, 
$\dimat(X_{\bv}) \leq d$.
Then, $\dimat(X) \leq d$.
\end{lemma}
\begin{proof}
\Wlog, $\pair{I, \leq}$ and the family $\Pa{X_{\bv}}_{\bv \in I}$ are
definable without parameters.
Let $\av \in X$ such that $\rkmat(\av) = \dimat(X)$, and let
$\bv_0 \in I$ such that $a \in X_{\bv_0}$.
By the Lemma~\ref{lem:order-ind}, there exists  $\bv \in I$ such that 
$\bv \geq \bv_0$ and $\av \bv_0 \indmat \bv$.
Hence, $\av \in X_{\bv}$ and $\av \indmat \bv$, and therefore
\[
d \geq \dimat(X_{\bv}) \geq \rk(\av / \bv) = \rk(\av) = \dimat(X).
\qedhere
\]
\end{proof}

\begin{remark}
The above lemma is not true if $\Pa{X_{\bv}}_{\bv \in I}$ be a definable
\emph{decreasing} family of subsets of $\monster^m$, instead of increasing.
For instance, let $\K$ be a real closed field, $\cl = \acl$,
$I := \K^{< 0} \times \K \cup \set{\pair{0,0}}$;
define $\pair{x, y} \leq \pair{x', y'}$ if $x \leq x'$ and $y = y'$, or $x = 0$.
Let $I_{b_1, b_2} := \set{\pair{x,y} \in I: \pair{x,y} \geq \pair{b_1, b_2}}$.
Then, $\pair {I, \leq}$ is a directed set, $\dim^{\acl}(I) = 2$, but
$\dim^{\acl}(I_{\bv}) \leq 1$ for every $\bv \in I$.
\end{remark}


\section{Matroids from dimensions}

Van den Dries in~\cite{dries89} gave a definition of dimension for definable
sets; we will show that his approach is almost equivalent to ours.
Let $\K$ be a first order structure.

\begin{definizione}
A \intro{dimension function} on $\K$ is a function $d$ from definable sets in
$\K$ to $\set{- \infty} \cup \N$, such that, for all $m \in \N$ and $S$, $S_1$
and~$S_2$ definable subsets of~$\K^m$, we have:
\begin{enumerate}[(Dim 1)]
\item $d(S) = - \infty$ iff $S = \emptyset$, $d(\set a) = 0$ for
every $a \in \K$, $d(\K) = 1$.
\item 
$d(S_1 \cup S_2) = \max \Pa{d(S_1), d(S_2)}$.
\item $d(S^\sigma) =d(S)$ for every permutation $\sigma$ of the
coordinates of~$\K^m$.
\item Let $U$ be a definable subset of $\K^{m+1}$, and, for $i = 0,1$, 
let $U(i) := \set{x \in \K^m: d(U_x) = i}$.  Then, $U(i)$ is definable
\emph{with the same parameters as~$U$}, and $d(U \cap \pi^{-1}(U(i))) =
d(U(i)) + i$, $i = 0, 1$, where $\pi := \Pi^{m+1}_m$.
\end{enumerate}
\end{definizione}

Notice that the axiom (Dim~4) is slightly stronger that the original axiom
in~\cite{dries89}; however, after expanding $\K$ by at most $\card{T}$ many
constants, the situation in~\cite{dries89} can be reduced to ours.

\begin{definizione}
Given a dimension function $d$ on~$\K$, for every $A \subset \K$ and $b \in
\K$ we define $b \in \mat^d(A)$ iff there exists $X \subseteq \K$ definable
with parameters in~$A$, such that $d(X) = 0$ and $b \in X$.
\end{definizione}

\begin{thm}
$\mat^d$ (more precisely, the extension of $\mat^d$ to a monster model)
is an existential matroid with definable dimension.
The dimension induced by $\mat^d$ is precisely~$d$.

Conversely, if $\mat$ is an existential matroid with definable dimension, then
$\dimat$ is a dimension function, and $\mat^{\dimat} = \mat$.
\end{thm}
\begin{proof}
The only non-trivial facts are that, if $d$ is a dimension function, then
$\mat^d$ is definable, satisfies EP and the existence axiom.
\smallskip

Definability)
Let $a \in \mat(B)$.
Let $X \subseteq \K$ be $B$-definable such that $d(X) = 0$ and $a \in X$.
Let $\phi(x, \bv)$ be the $B$-formula defining~$X$.
By (Dim~4), \wloG $d(\phi(\K, \y) \leq 0$ for every~$\y$.%
\footnote{Here it is important that in (Dim~4) we asked that the parameters of
$U(i)$ are the same as the parameters of~$U$.}
Hence, $\phi(x, \y)$ is an \xnarrow formula.
\smallskip

EP)
Let $a \in \mat(Bc) \setminus \mat(B)$.
Assume, for contradiction, that $c \notin \mat(Ba)$.
Let $X \subseteq \K^2$ be $B$-definable, such that $a \in X_c$ and 
$d(X_c) =0$.
Let $X' := X \cap \pi^{-1}(X(0))$, where $\pi := \Pi^2_1$.
By assumption, $\pair{c,a} \in X'$ and, by (Dim~4), $\dim(X') \leq 1$; \wloG,
$X = X'$.

Let $Z := \set{u \in \K: d(X^u) = 1}$.
Since $c \in X^a$ and $c \notin \mat(Ba)$, $a \in Z$.
Since $a \notin \mat(B)$, $d(Z) = 1$.
Hence, by (Dim~4) and (Dim~3), $d(X) = 2$, absurd.
\smallskip

Existence)
Immediate from Lemma~\ref{lem:cl-E}-5.
\end{proof}


\section{Expansions}

Remember that $\monster$ is a monster model of a complete $\Lang$-theory~$T$.
We are interested in the behaviour of definable matroids under expansions of~$\monster$.
In this section we 
assume that
$\mat$ is a definable closure operator on the monster model~$\monster$.

\begin{definizione}\label{def:cl-induced}
Given $X \subseteq \monster$, let the restriction $\mat^X: \Power(X) \to
\Power(X)$ and the relativisation $\mat_X: \Power(\monster) \to
\Power(\monster)$ of $\mat$ be defined as $\mat^X(Y) := \mat(Y) \cap X$
and $\mat_X(Y) := \mat(X Y)$.
\end{definizione}

\begin{lemma}
Given $X \subseteq \monster$, $\mat^X$ is a  closure operator on~$X$ and
$\mat_X$ is a closure operator on~$\monster$.
If moreover $\mat$ is a matroid, then both $\mat^X$ and $\mat_X$ are matroids,
$A \ind[\mbox{$\mat_X$}]_B C$ iff $A \indmat_{X B} C$,
and $\ind[\mbox{$\mat^X$}]\,$ is the restriction of $\indmat$ to the subsets 
of~$X$.
\end{lemma}

\begin{remark}
Given $B \subset \monster$ (with $\card B < \kappa$), 
let $\monster_B$ be the expansion of~$\monster$ with all constants from~$B$.
\begin{enumerate}
\item if $\mat$ is definable, then $\mat_B$ is also definable 
(see Remark~\ref{rem:general-closure}).
\item if $\mat$ is a matroid, then $\mat_B$ is also a matroid;
\item if $\mat$ is definable and satisfies Existence, 
then $\mat_B$ satisfies Existence too;
\item if $\mat$ is an existential matroid, then $\mat_B$ is also an
existential matroid, and  $\dimat$ and $\dim^{\mat_B}$ coincide (the definable
sets of $\monster$ and of $\monster_B$ are the same).
\end{enumerate}
\end{remark}

\begin{example}
In the above Remark, it is not true that, if $\mat$ is a definable matroid,
and $\mat_B$ satisfies Existence, then $\mat$ satisfies Existence.
For instance, let $B$ be any non-empty subset of $\monster$ (of cardinality
less than~$\kappa$), and $\mat = \mat^1$;
then, $\mat_B = \mat^0$ satisfies Existence, but $\mat$ does not.
\end{example}

\begin{lemma}\label{lem:cl-pair}
Let $X \subseteq \monster$.
Let $\monster'$ be the expansion of $\monster$ with a predicate $P$ for~$X$.
Assume that $\monster'$ is a monster model of $T(X)$, and
denote by $\mat_X'$ the closure operator $\mat_X'(Y) := \mat(X Y)$
on~$\monster'$ 
($\mat_X'$ coincides with $\mat_X$). 
\begin{enumerate}
\item If $\mat$ is definable, then $\mat_X'$~is definable on~$\monster'$;
\item if $\mat$ is a matroid, then $\mat_X'$~is a matroid;
\end{enumerate}
\end{lemma}

\begin{proof}
Let $D \subseteq X$ such that $\card D < \kappa$ and $\mat(X) = \mat(D)$.
\begin{enumerate}
\item
$b \in \mat'_X(A)$ iff $b \in \mat(A X)$ iff 
$\monster \models \phi(b, \av, \cv)$ for some \xnarrow formula 
$\phi(x, \y, \z)$ and some $\cv \in X^n$.
Define $\psi(x, \y) := \exists \z\ \Pa{P(\z) \et \phi(x, \y, \z)}$.
Notice that $\psi$ is an $\Lang(P)$-formula, and that, for every
$\av' \subset \monster$, $\psi(\monster', \av') \subseteq \mat'_X(\av')$.
\item 
Trivial.
\qedhere
\end{enumerate}
\end{proof}


\begin{lemma}\label{lem:cl-small-relative}
Let $\monster$, $X$ and $\monster'$ be as in the above lemma.
Let $\pair{\Bm, Y} \prec \pair{\monster, X}$; assume moreover that
$\mat$ is a definable closure operator on~$\monster$.
Then, for every $A \subseteq \Bm$, 
$\mat_{Y}(A) \cap \Bm = \mat_{X}(A) \cap \Bm$.
\end{lemma}
Hence, in the above situation, inside $\Bm$ we do not need to distinguish 
between $\mat_{X}$ and $\mat_{Y}$.

\begin{lemma}\label{lem:elementary-pair-independent}
Let $\mat$ be a definable matroid (not necessarily existential)
and $X$, $Y$, $X^*$, and $Y^*$ be elementary substructures of~$\monster$,
such that $X \subseteq X^* \cap Y$ and $X^* \cup Y \subseteq Y^*$.
Let $\Ltwo$ be the expansion of $\Lang$ with a new unary predicate~$P$,
and consider $\pair{Y, X}$ and $\pair{Y^*, X^*}$ as $\Ltwo$-structures.
Assume that $(Y, X) \preceq (Y^*, X^*)$.
Then, $X^* \indmat_X Y$.
\end{lemma}
\begin{proof}
Let $\x^* \subset X^*$; it suffices to prove that $\x^* \indmat_X Y$.
However, $\tp_{\Lang}(x^*/ Y)$ is finitely satisfied in~$X$, and we are done.
\end{proof}

Assume that $\monster$ expands a ring without 0 divisors.
Let $\monster'$ be an expansion of~$\monster$ to a larger language~$\Lang'$; 
assume that $\monster'$ is also a monster model
and that $\mat'$ is an existential matroid on~$\monster'$. 
We have seen that in this case $\mat'$ is the unique existential matroid on
$\monster'$, and that, for every $X$ definable subset of~$\monster'$,
$\dim'(X) = 0$ iff $F(X^4) \neq \monster'$ (where $\dim'$ is the dimension
induced by~$\mat'$).
It is clear that $\mat'$, in general, is not definable in~$\monster$.
However, the dimension function $\dim'$ \emph{is} definable in~$\monster$:
Hence, we can restrict the dimension function $\dim'$ to the sets definable in
$\monster$ (with parameters), and get a function~$\dim$.

\begin{lemma}\label{lem:cl-expansion}
Let $\monster$, $\monster'$, $\dim'$ and $\dim$ be as above.
Then, $\dim$ is a dimension function on $\monster$ (\ie, it satisfies the
axioms in Definition~\ref{def:dimension}).
The matroid $\mat$ induces by $\dim$ is characterized by
For every $A$ and $b$, $b \in \mat(A)$ iff
there exists $X \subseteq \monster$, definable in $\monster$ with 
parameters~$A$, such that $\dim'(X) = 0$.
\end{lemma}



\begin{corollary}
Assume that $\monster$ expands a ring without zero divisors.
Let $\monster'$ be an expansion of~$\monster$.
If $\monster'$ is geometric, then $\monster$ is also geometric.
\end{corollary}
Compare the above corollary with \cite[Corollary~2.38 and Example~2.40]{adler}.


\section{Extension to  imaginary elements}
\label{sec:imaginary}

Again, $\monster$ is a monster model of a complete theory~$T$, and $\mat$ is
an existential matroid on~$\monster$.
Let $\monstereq$ be the set of imaginary elements, and $\Teq$ be the theory of
$\monstereq$.
Our aim is to extend to~$\monstereq$ the matroid $\mat$ to a closure operator
$\mateq$  and the rank  $\rkmat$ to a ``rank function'' $\rkmatt$ (see
Definition~\ref{def:rank}). 

Notice that $\acleq$ is a closure operator on $\monstereq$
extending~$\acl$; 
however, if $\mat = \acl$, then in general $\mateq \neq \acleq$; 
hence, when $\mat = \acl$ , we will have to pay attention not to confuse the
two possible extensions of $\mat$ to~$\monstereq$ (\cf
Remark~\ref{rem:acleq}). 

On the other hand, by $\dcleq$ we will always denote the usual extension of
$\dcl$ to imaginary element: $a \in \dcl(b)$ if $\Xi(a/B) = \set{a}$.

We will start with the definition of $a \in \mateq(B)$ when $a$ is real and
$B$ is imaginary.

\begin{definizione}\label{def:matt-1}
Let $B$ be a set of imaginary elements (of cardinality less than~$\kappa$),
and $a$ be a real element.
We say that $a \in \matt(B)$ iff $\Xi(a/B)$ has finite $\rkmat$.
\end{definizione}

\begin{remark}
If $a$ and $B$ are real, then $a \in \matt(B)$ iff $a \in \mat(B)$.
\end{remark}

\begin{remark}
If $a$ is real a $B$ is imaginary, then $a \in \matt(B)$ iff 
$\rkmat(\Xi(a / B)) < \kappa$.
\end{remark}
\begin{proof}
Assume that $\rkmat(\Xi(a / B)) \geq \aleph_0$; we want to show that
$\rkmat(\Xi(a / B)) \geq \kappa$.
For every $\lambda < \kappa$, consider the type in $\lambda$ variables $p(\x)$
over~$B$, saying that, for every $i < \lambda$,  $x_i \elem_B a$,
and the $x_i$ are $\mat$-independent.
By assumption, $p(\x)$ is consistent, and hence satisfied in~$\monster$ by
$\lambda$-tuple $(a_i)_{i < \lambda}$; thus, the $a_i$ are $\lambda$-independent
elements in $\Xi(a / B)$. Therefore, $\rkmat(\Xi(a / B)) > \lambda$ for
every $\lambda < \kappa$; since $\kappa$ is a limit cardinal, we are done.
\end{proof}

Recall that $\monster$ has geometric elimination of imaginaries if every
for imaginary tuple $\av$ there exists a real tuple $\bv$ such that $\av$ and
$\bv$ are inter-algebraic.
If $\monster$ had geometric elimination of imaginaries, we could define
$\av \in \matt(B)$ iff there exists a real tuple $\cv$ such that
$\av \in \acleq(\cv)$ and $\cv \subset \matt(B)$.
In the general case, we need a more involved definition and some preliminary
lemmata.

\begin{lemma}[Exchange Principle {\cite[3.1]{gagelman}}]
$\matt$ satisfies the Exchange Principle for real points over imaginary
parameters.
That is, for $a$ and $b$ real elements and $C$ imaginary, 
if $a \in \matt(b C) \setminus \matt(C)$, then $b \in \matt(a C)$.
\end{lemma}

\begin{proof}
Let $a$, $b$, and $C$ as in the hypothesis, and assume, for contradiction,
that $b \notin \matt(a C)$.
Let $B$ be a real set (of cardinality less than~$\kappa$), 
such that $\Xi(a / b C) \subseteq \mat(B)$.
By enlarging~$B$, \wloG we can assume that $C \subseteq \dcleq(B)$.
Since $a \notin \matt(C)$, $\Xi(a / C) \nsubseteq \mat(B)$, and therefore
there exists $\sigma \in \aut(\monster/ C)$ such that 
$a' := a^\sigma \notin \mat(B)$.
Since $b \notin \matt(a C)$, we have $b^\sigma \notin \matt(a' C)$, 
and therefore there exists
$b' \elem_{a' C} b^\sigma$ such that
$b' \notin \mat(a' B)$.

Notice that $a' b' \elem_C a b$; let $\mu \in \aut(\monster / C)$ such that
$a' = a^\mu$ and $b' = b^\mu$, and define $B' := B^\mu$.
Since $\Xi(a / b C) \subseteq \mat(B)$, we have 
$\Xi(a' / b' C) \subseteq \mat(B')$. 
Moreover, since $C \subseteq \dcleq(B')$, we have 
$\Xi(a' / b' B') \subseteq \Xi(a' / b' C) \subseteq \mat(B')$.
Thus, $a' \in \mat(b' B) \setminus \mat(B)$; hence, since $a'$, $b'$, and $B$
are real, $b' \in \mat(a' B)$, absurd.
\end{proof}

It is relatively easy to also prove the following:
\begin{lemma}[Transitivity]
$\matt$ is transitive for real sets over imaginary parameters:
that is, if $A$ is a imaginary,$\bv$~is a tuple of reals, and $c$ is real, 
such that $\bv \subseteq \matt(A)$ and $c \in \matt(A \bv)$, 
then $c \in \matt(A)$.
\end{lemma}

\begin{definizione}\label{def:rank}
Let $X$ be a set and $\RK$ be a function from finite subsets of $X$ to~$\N$.
Let $\av$ and $\bv$ vary among the finite tuples of $X$ and $B$, $C$ among the
subsets of~$X$.
$\RK$ is a rank function if it satisfies the following conditions:
\begin{description}
\item[Finite character:]  for all $\av$ and $B$ there exists
$B' \subseteq B$ finite,  such that $\RK(\av / B') = \RK(\av /B)$.
\item[Additivity:] for every $\av$, $\bv$, and~$B$,
$\RK(\av \bv /B) = \RK(\av / \bv C) + \RK(\bv / C)$.
\item[Transitivity:]  for all $B \subseteq C$ and all~$a$,
$\RK(\av / C) \leq \RK(\av / B)$.
\end{description}
\end{definizione}

\begin{corollary}[\cite{gagelman}]\label{cor:rk-real-im}
\begin{itemize}
\item 
For any real tuple $\bv$ and imaginary set~$A$, any two maximally
$A$-$\matt$-independent sub-tuples of $\bv$ have the same cardinality; thus,
we may use the notation $\rkmat( - / -)$ accordingly, as long as the first
argument is real.
\item
The function $\rkmat$ defined above coincides with the usual one when both
arguments are real.
\item
The function $\rkmat$ has Finite Character, is Additive and Transitive, as long
as the first argument is real.
\item
$\rkmat$ satisfies Extension (for real first argument):
if $\av$ is real and $B$ and $C$ are imaginary sets,
then there exists $\av' \elem_C \av$ such that 
$\rkmat(\av'/ B C) = \rkmat(\av/ B)$.
\item
If $\av$ is real, then $\av \in \matt(C)$ iff $\rkmat(\av / C) = 0$.
\end{itemize}
\end{corollary}

\begin{definizione}
Let $A$ and $\av$ be imaginary.
Choose $\bv$ real, such that $\av \subseteq \acleq(A \bv)$.
Define $\rkmat(\av / A) := \rkmat(\bv / A) - \rkmat(\bv / \av A)$.
\end{definizione}

\begin{lemma}[{\cite[3.3]{gagelman}}]
The above definition does not depend on the choice of~$\bv$.
Moreover, $\rkmat(\av / A)$ is a natural number, and coincides with the one
given in Corollary~\ref{cor:rk-real-im} when $\av$ is real.
Finally, $\rkmat$ is a rank function on $\monstereq$.
\end{lemma}
\begin{proof}
Let's prove that $\rkmat$ does no depend on the choice of $\bv$.
Let $\bv$ and $\bv'$ be real tuples, such t hat $\av \subseteq \acleq(A \bv)$
and $\av \subseteq \acleq(A \bv')$.
\Wlog, $\bv \subseteq \bv'$ and hence $\bv' = \pair{\bv, \bv''}$.
We must prove that $\rkmat(\bv \bv'' /A) - \rkmat(\bv \bv'' / \av A)
= \rkmat(\bv /A) - \rkmat(\bv / \av A)$.
The above is equivalent to
$\rkmat(\bv \bv'' /A) - \rkmat(\bv /A) =
\rkmat(\bv \bv'' / \av A)  - \rkmat(\bv / \av A)$.
But the left hand side is equal to $\rkmat(\bv \bv'' / \bv A)$, 
while the right hand side is equal to
$\rkmat(\bv \bv'' / \av \bv A)$.
Since $\av \in \acleq(\bv A)$, the two sides are equal.

Let's prove additivity.
Let $\av$ and $\av'$ be imaginary tuples and $C$~be a set of imaginary
elements.
We want to prove that $\rkmat(\av \av' /C) = \rkmat(\av / \av' C) + \rkmat(\av'/C)$.
Let $\bv$ and $\bv'$ be real tuples, such that $\av \subset \acleq(C \bv)$ and
$\av' \subset \acleq(C \bv')$.
We have to show that
$\rkmat(\bv \bv'/C) - \rkmat(\bv \bv' / \av \av' C) =
\rkmat(\bv / \av'C) - \rkmat(\bv / \av \av' C) +
\rkmat(\bv' / C) - \rkmat(\bv' / \av' C)$, that is
\begin{multline*}
- \rkmat(\bv \bv'/C) + \rkmat(\bv \bv' / \av \av' C) +
\rkmat(\bv / \av'C) - \rkmat(\bv / \av \av' C) + \\
+\rkmat(\bv' / C) - \rkmat(\bv' / \av' C) = 0.
\end{multline*}
The above is equivalent to
\begin{equation}
- \rkmat(\bv / \bv' C) + \rkmat(\bv' / \bv \av \av' C) +
\rkmat(\bv / \av'C) - \rkmat(\bv' / \av' C) = 0.\label{eq:rank-1}
\end{equation}
Let $C' := C \av'$.
Since $\av \subset \acleq(C \bv)$ and $\av' \subset \acleq(C \bv')$,
\eqref{eq:rank-1} is equivalent to
\[
- \rkmat(\bv / \bv' C')  + \rkmat(\bv' / \bv C') +
\rkmat(\bv / C') - \rkmat(\bv' / C') = 0.
\]
Finally, $\rkmat(\bv' / \bv C') + \rkmat(\bv / C')
= \rkmat(\bv \bv' /C') = \rkmat(\bv / \bv' C') + \rkmat(\bv' / C')$, 
and we are done.
\end{proof}

Finally, we define $a \in \matt(B)$ iff $\rkmat(\av / B) = 0$, where $a$ and
$B$ can be either real of imaginaries.

\begin{lemma}
The operator $\matt$ defined above is a closure operator, coincides with
$\mat$ for real elements, and extends the operator defined in~\ref{def:matt-1}.
\end{lemma}

\begin{remark}\label{rem:acleq}
Assume that $\monster$ is pregeometric structure and $\mat = \acl$.
Given $\bv$ a real or imaginary tuple, 
we have $\acleq(\bv) \subseteq \matt(\bv)$ and
$\matt(\bv) \cap \monster = \acleq(\bv) \cap \monster$.
However, it is not true in general that
$\matt = \acleq$: more precisely, $\matt = \acleq$ iff $\monster$
is surgical~\cite{gagelman}.
For instance, if $\monster$ is a model of the theory of $p$-adic fields,
then $\monster$ is geometric but not surgical;
the imaginary sort $\Gamma$ corresponding to the value group has dimension~$0$
but it is infinite; therefore, 
$\Gamma \subset \matt(\emptyset) \setminus \acleq(\emptyset)$.
\end{remark}


 \section{Density}
\label{sec:density}

Again, $\monster$ is a monster model of a complete theory~$T$, and $\mat$ is
an existential matroid on~$\monster$.

\begin{definizione}\label{def:dense}
Let $\K \preceq \monster$, and $X \subseteq \K$.
We say that $X$ is \intro{\cldense{}} in $\K$ if, for every $\K$-definable subset $U$ of~$\K$, if $\dimat(U) = 1$, then $U \cap X \neq \emptyset$.
We define $\mat_\K(X) := \mat(X) \cap \K$, and
we say that $X$ is \intro{\clclosed{}} in~$\K$ if $\mat_\K(X) = X$.
\end{definizione}

\begin{examples}
\begin{enumerate}
\item If $\K$ is geometric, then $X$ is dense in $\K$ iff $X$ intersects every
infinite definable subset of~$\K$.  
\item If $\K$ is strongly minimal, then $X$ is dense in $\K$ iff $X$
is infinite.
\item If $\K$ is o-minimal and densely ordered, then $X$ is dense in $\K$ in
the sense of the above definition iff $X$ is topologically dense in~$\K$ (this
is our motivation for the choice of the term ``dense'').
See also \S\ref{sec:dmin} for a generalisation of this example.
\end{enumerate}
\end{examples}

\begin{remark}
If $X \subset \K$ is \cldense (in~$\K$), and $a \in X$, then $X \setminus \set
a$ is also \cldense.
\end{remark}
\begin{proof}
If $U \subseteq \K$ is definable and of dimension~1, then $U \setminus \set a$
is also definable and of dimension~1.
\end{proof}

\begin{lemma}\label{lem:density-closure}
Let $X \subseteq \K \preceq \monster$.
If $X$ is \clclosed and \cldense in~$\K$, then $X \preceq \K$.
\end{lemma}
\begin{proof}
Robinson's test.
Let $A \subseteq \K$ be definable, with parameters from~$A$: we must show that
$A \cap X \neq \emptyset$.
If $\dimat(A) = 1$, this is true because $X$ is \cldense in~$\K$.
If $\dimat(A) = 0$,  this is true because $X$ is \clclosed in~$\K$.
\end{proof}

\begin{lemma}\label{lem:Td-consistent-GCH}
Let $\K \preceq \monster$ 
be a saturated model of cardinality $\lambda > \card T$.
Then, there exists $X \subset \K$ such that $X$ is a \clbasis of $\K$ and
$X$ is \cldense in~$\K$.
Moreover, there exists $\F \prec \K$ such that $\F$ is \clclosed and \cldense
in $\K$ and $\F$ is not equal to~$\K$.
\end{lemma}

\begin{proof}
%
Let $(A_i)_{i < \lambda}$ be an enumeration of all subsets of $\K$ which are
definable (with parameters from~$\K$) and of dimension~1.
Build a \clindependent sequence $(a_i)_{i < \lambda}$ inductively:
for every $\mu < \lambda$, we make so that
$(a_i)_{i < \mu}$ is \clindependent, and, for every $i < \mu$ there exists
$j < \mu$ such that $a_j \in A_i$. 
Let $i_\mu$ be the smallest index such that $A_{i_\mu}$ does not contain any
$a_i$ for $i < \mu$.

\begin{claim}
$i_\mu$ exists.
\end{claim}
Otherwise, $\K$ would have a basis of cardinality less than~$\lambda$,
contradicting the saturation hypothesis.

\begin{claim}
There exists $a_\mu \in A_{\mu}$ such that $a_\mu$ is \clindependent from
$(a_i)_{i < \mu}$.
\end{claim}
Otherwise, $\rkmat(A_\mu) < \lambda$, absurd.

Define $a_\mu$ as in the above claim.

By construction, $X' := \set{a_i :i < \lambda}$ is \clindependent and \cldense
in~$\K$; we can complete it to a \clbasis~$X$, which is also \cldense.%
\footnote{Is it true or not that $X'$ is already a basis?}

Choose $a \in X$, let $Y := X \setminus \set a$, and $\F := \mat(Y)$.
Since $X$ is \cldense, $Y$~is also \cldense, and therefore $\F$ is \cldense 
in~$\K$.
Moreover, since $X$ is a \clbasis, $a \notin \F$.
Finally, by Lemma~\ref{lem:density-closure}, $\F \prec \K$.
\end{proof}

The proof of the above lemma shows the following stronger results.
\begin{corollary}
Let $\K$ be as in Lemma~\ref{lem:Td-consistent-GCH}.
Let $c \in \K \setminus \mat \emptyset$.
Then, there exists $\F \prec \K$ \clclosed and \cldense in~$\K$,
such that $c \notin \F$.
\end{corollary}

Given $\K \models T$, and $X$, $Y$ subsets of~$\K$,
we say that $X$ is \cldense in $\K$ \wrt $Y$ if for every subset $U$ of~$\K$
definable with parameters from~$X$, 
if $\dimat(U) = 1$, then $U \cap X \neq \emptyset$.
\begin{lemma}\label{lem:Td-consistent}
There exists $\F \prec \K \models T$ such that $\F$ is a proper \cldense and \clclosed
subset of~$\K$.
\end{lemma}
\begin{proof}
If $T$ has a saturated model of cardinality $> \card T$, we can apply
Lemma~\ref{lem:Td-consistent-GCH}.
Otherwise, let $\K_0 \prec \K_1 \prec \dots$ be an elementary chain of models
of~$T$, such that, for every $n \in \N$, $\K_{n+1}$ is 
$\Pa{\card{\K_n} + \card T} ^ +$-saturated, 
and let $\K := \bigcup_{n \in \N} \K_n$.
Proceeding as in Lemma~\ref{lem:Td-consistent-GCH}, for every $n\in \N$ 
we build a \clindependent set $\Afam_n$ of elements in $\K_{n+1}$, 
such that $\Afam_n \subseteq \Afam_{n+1}$
and $\Afam_n$ is \cldense in $\K_{n+1}$ \wrt $\K_n$.
Let $\Afam := \bigcup_n \Afam_n$.
Then, $\Afam$ is a \clindependent set of elements in $\K$, which is also
\cldense in~$\K$.
Conclude as in Lemma~\ref{lem:Td-consistent-GCH}.
\end{proof}


\section{Dense pairs}
\label{sec:dense-pairs}

Let $\Bm$ be a real closed field and $\Am$ a proper dense subfield of $\Am$,
such that $\Am$ is also real closed.
We call $\pair{\Bm, \Am}$ a dense pair of real closed fields, and
we consider its theory, in the language of ordered fields expanded with a
predicate for a (dense) subfield.
Robinson \cite{robinson74} proved that the theory of dense pairs of real
closed fields is complete.
Van den Dries \cite{vdd-dense} extended the results to o-minimal theories:
if $T$ is a complete o-minimal theory expanding the theory of (densely)
ordered Abelian groups, then  the theory of dense elementary pairs of models
of $T$ is complete.
Macintyre \cite{macintyre} introduced an abstract notion of density, in the
context of geometric theories, which for o-minimal theories specialises to the
usual topological notion, and proved various results; more recent work
has been done in the context of so called ``lovely pairs'' of geometric or
simple structures: see for instance \cite{berenstein, BPV}.

In \S~\ref{sec:density} we also proposed an abstract notion of density,
which for geometric theories specialises to the one given by Macintyre (and,
independently, by others): see Remark~\ref{rem:geometric-density}. 
However, it is not true in general that the theory of dense pairs of models of
$T$ is complete (unless $T$ is geometric): the main result of this section is
that if $T$ expands the theory of integral domains, 
and we add the additional condition that $\Am$ is \clclosed in~$\Bm$, we
obtain a complete theory, which we denote by~$\Td$.
(if $T$ is geometric, the condition is trivially true).
We will also show that $\Td$ admits an existential matroid 
(\S\ref{subsec:small-closure}).
For the proofs we will follow closely~\cite{vdd-dense}.

$\monster$ is a monster model of a complete theory~$T$, and $\mat$ is an
existential matroid on~$\monster$.
For this section, we will write $\dim$ instead of $\dimat$,
$\rk$ instead of~$\rkmat$, and $\ind$ instead of~$\indmat$.

\begin{definizione}
Let $\Ltwo$ be the expansion of $\Lang$ by a new unary predicate~$P$.
Let $\Ttwo$ be the $\Ltwo$-expansion of~$T$, whose models are the pairs
$\pair{\K, \F}$, with $\F \prec \K$, $\F \neq \K$, and $\F$ \clclosed in~$\K$.

Assume that $\dim$ is definable.
Let $\Td$ be the $\Ltwo$-expansion of $T$ saying that $\F$ is \clclosed and
\cldense in~$\K$
(we need definability of $\dim$ to express in a first-order way that $\F$ is
\cldense in~$\K$).
\end{definizione}
Notice that, by Lemma~\ref{lem:density-closure}, $\Td$ extends~$\Ttwo$.
Notice that if $\mat = \acl$, then $\Ttwo$ is the theory of pairs
$\pair{\K,\F}$, with $\F \prec \K \models T$; 
however, if $\mat \neq \acl$, then there exists 
$\F \prec \monster$ with $\F$ not \clclosed in~$\monster$.

\begin{lemma}
$\Td$ is consistent.
\end{lemma}
\begin{proof}
By Lemma~\ref{lem:Td-consistent}.
\end{proof}


\begin{proviso*}
For the remainder of this section, we assume that $T$ 
expands the theory of integral domains (and therefore $\dim$ is definable)
and that $\pair{\K, \F} \models \Td$.
\end{proviso*}

\begin{thm}\label{thm:Td-complete} 
$\Td$ is complete.
\end{thm}

\begin{definizione}
An $\Ltwo$ formula $\phi(\x)$ is \intro{\basic} if it is of the form
\[
\exists \y \Pa{U(\y) \et \psi(\x, \y)},
\]
where $\psi$ is an $\Lang$-formula,%
\footnote{\Basic formulae were called ``special'' in \cite{vdd-dense}.}
\end{definizione}

\begin{thm}\label{thm:Td-eq}
Each $\Ltwo$-formula $\psi(\x)$ is equivalent,
modulo~$\Td$, to a Boolean combination of \basic formulae,
with the same parameters as~$\psi$.
\end{thm}

Theorems \ref{thm:Td-complete} and \ref{thm:Td-eq} will be proved in \S\ref{sec:Td-proof}.


\subsection{Small sets}

In this subsection we will assume that $\pair{\K, \Am} \models \Ttwo$.

\begin{definizione}\label{def:small}
A subset $X$ of $\K$ is \intro{$\Am$-small} if $X \subseteq f(\Am^n)$, 
for some \Zapplication $f: \K^n \app \K$ which is definable in~$\K$.
\end{definizione}

\begin{definizione}
Let $X \subseteq \K^n$.
We say that $X$ is \intro{weakly dense} in $\K$ if, for every definable 
$U \subseteq \K^n$, if $X \subseteq U$, then $\dim(U) = n$.
\end{definizione}
For instance, if $\mat = \acl$, then $X$ is a weakly dense subset of $\K$ iff
$X$ is infinite.

\begin{remark}
If $X$ is a weakly dense subset of~$\K$, then $X^n$ is a weakly dense subset
of $\K^n$.
\end{remark}

\begin{lemma}\label{lem:weakly-dense}
If $\K \models T$ and $\K' \preceq \K$, then $\K'$ is weakly dense in~$\K$.
\end{lemma}
\begin{proof}
\Wlog, the pair $\pair{\K, \K'}$ is $\omega$-saturated.
Assume, for contradiction, that $U \subset \K$ is definable, with 
parameters~$\bv \in \K^n$, $\dim(U) = 0$, and $\K' \subseteq U$.
By saturation, $\rk(\K')$ is infinite; let $\cv \in \K'^{n+1}$ be independent
elements.
However, $\cv \in U$, and therefore $\cv \subset \mat(\bv)$, absurd.
\end{proof}

The following result is the most delicate one.
\begin{lemma}[{\cite[1.1]{vdd-dense}}]\label{lem:infinite-dimension}
Let $f: \K^{n + 1} \app \K$ be a \Zapplication $\Am$-definable in~$\K$, 
and let $b_0 \in \K \setminus \Am$.
For every $x \in \K$ and $\y = \pair{y_0, \dotsc, y_n} \in \K^{n+1}$ let
$p(\y, x) := y_0 + y_1 x + \dots + y_n x^n$,
Then, there exists $\av \Am^{n+1}$ such that
\[
p(\av, b_0) \notin f(\Am^n \times \set b).
\]
\end{lemma}
\begin{proof}
Otherwise there is for each $\av \in \Am^{n+1}$ a tuple $\cv \in \Am^n$ such
that $p(\av, b_0) \in f(\cv, b_0)$.
\Wlog, $f$~is definable without parameters.
For each $\y \in \K^{n+1}$ and $\z \in \K^n$ let 
$D(\y, \z) := \set{x \in \K: p(\y, x) \in f(\z, x}$.
Define $W := \set{\pair{\y, \z} := \dim(D(\y, \z) = 1)}$, and
$Y := \Pi^{2n +1}_{n+1}(W)$.
Since $b_0 \notin \Am$ and $\Am$ is \clclosed, $\Am^{n+1} \subseteq Y$.
Since $Y$ is definable, Lemma~\ref{lem:weakly-dense} implies 
that $\dim(Y) = n + 1$; therefore, $\dim(W) \geq n + 1$.
Let $Z := \set{\z \in \K^n: \dim(\set{\y: \pair{\y, \z} \in W}) \geq 1}$.
Since $\dim(W) \geq n + 1$ and $\dim(\K^n) = n$, we have that $Z$ is non-empty.

Choose $\cv \in Z$. Let $\av \in \K^{n+1}$ such that $\pair{\av, \cv} \in W$ and
$\rk(\av / \cv) \geq 1$.
By definition of~$W$, $\dim(D(\av, \cv) = 1$; choose $b \in D(\av, \cv)$ such
that $\rk(b/ \cv \av) = 1$.
Define $d := p(\av, b)$; remember that $d \in f(\cv, b)$, and therefore
$d \in \mat(\cv b)$.
Let $\av' \in \K^{n+1}$ such that $\av' \elem_{\cv b d} \av$
and $\av' \ind_{\cv b d} \av$.
Since $d \in \mat(\cv, b)$, we have$\av' \ind_{\cv b} \av$.
Moreover, $p(\av', b) = d$; therefore, $p(\av - \av', b) = 0$.

If $\av \neq \av'$, this implies that $b$ is algebraic over $\av - \av'$, and
therefore $b \in \mat(\av \av')$, contradicting the fact that 
$b \notin \cl(\av \cv)$ and $\av' \ind_{\cv b} \av$.

If instead $\av = \av'$, then $\av' \ind_{\cv b} \av$ implies that 
$\av \subset \mat(\cv b)$, contradicting the facts 
that $b \notin \mat(\cv \av)$ and $\rk(\av/ \cv) \geq 1$.
\end{proof}

Notice that the hypothesis of the above lemma can be weakened to:\\
$\K\models T$, $\Am$ is a proper \clclosed and weakly dense subset of~$\K$.

\begin{remark}[{\cite[1.3]{vdd-dense}}]\label{rem:pair-small}
Each $\Am$-small subset of $\K$ is a proper subset of~$\K$.
\end{remark}
\begin{proof}
Same as \cite[Corollary~1.3]{vdd-dense}.
\end{proof}

\begin{remark}\label{rem:pair-small-union}
A finite union of $\Am$-small subsets of $\K$ is also $\Am$-small.
\end{remark}

\begin{lemma}
Let $B \subseteq \K$ be a proper \clclosed subset.
Then, $B$ is co-\cldense in~$\K$;
that is, $\K \setminus B$ is \cldense in~$\K$.
\end{lemma}
\begin{proof}
Since $B$ is \clclosed, $F(B^4) \subseteq B$.
Assume, for contradiction, that there exists $U$ definable in~$\K$,
such that $\dim(U) = 1$ and $U \subseteq B$.
Then, $F(U^4) = \K$, and therefore $F(B^4) = \K$.
However, since $\K$ is \clclosed, $F(B^4) \subseteq B$ contradicting the
assumption that $B \neq \K$.
\end{proof}
The hypothesis in the above lemma can be weakened to:\\
$B$ proper subset of $\K \models T$, and $F(B^4) \subseteq B$.

\begin{lemma}[{\cite[Lemma 1.5]{vdd-dense}}]\label{cor:dense-cardinality}
If the pair $\pair{\K, \Am}$ is $\lambda$-saturated, where $\lambda$ is an
infinite cardinal and $\lambda > \abs T$, then $\dim(\K / \Am) \geq \lambda$.
Hence, if $\card X < \lambda$, then $\mat(\Am X)$~is co-\cldense in~$\K$.
\end{lemma}
\begin{proof}
Same as~\cite[Lemma 1.5]{vdd-dense}.
Let $E$ be a generating set for $\K/\Am$, 
and suppose that $\card E < \lambda$.
Let $\Gamma(v)$ be the set of $\Ltwo$-formulae of the form
\[
\forall y_1 \dots \forall y_n \Pa{U(\y) \rightarrow 
v \notin f(\y, e_1, \dotsc, e_p)},
\]
where $f(\y, \z)$ is a \Zapplication $\emptyset$-definable in~$\K$,
and $e_1, \dotsc, e_p$ are in~$E$.
By Remarks~\ref{rem:pair-small} and~\ref{rem:pair-small-union},
$\Gamma(v)$ is a consistent set of formulae, with less than $\lambda$~many parameters.
By saturation, there exists $b \in \K$ realising the partial type $\Gamma(v)$.
Thus $b \notin \mat(\Am E)$, absurd.
\end{proof}

Notice that in the original \cite[Lemma~1.5]{vdd-dense}, if $T$
expands~RCF,
then van den Dries' assumption that $\mathcal A$ is dense in $\mathcal B$
is superfluous; density is used if however $T$ expands only the theory of
ordered Abelian groups.

\subsection{Proof of Theorems~\ref{thm:Td-complete} and~\ref{thm:Td-eq}}
\label{sec:Td-proof}
The proof is similar to the one in \cite{boxall}.


\begin{definizione}
Let $\pair{\Bm, \Am} \models \Ttwo$ and $C \subseteq \Bm$.
Let $\cv$ be a tuple of elements from~$\eq{\Bm}$; 
the \Utype of~$\cv$, denoted by $\Utp(\cv)$, is the information which tells us
which members of $\cv$ are in $\Am$ (notice: the elements in $\cv$ are real or
iminaginary, but only real elements can be in~$\Am$). 
We say that $\cv$ is \Uindependent if $\cv \ind_{\Am \cap \cv} \Am$
(where, again, only the real element of $\cv$ can be in $\Am \cap \cv$).
\end{definizione}

We will use a superscript $1$ to denote model-theoretic notions for~$\Lang$,
and a superscript $2$ to denote those notions for~$\Ltwo$: for instance, we
will write $a \equiv^1_C a'$ if the $\Lang$-type of $a$ and $a'$ over $C$ is
the same, and $a \equiv^2_C a'$ the $\Ltwo$-type of $a$ and $a'$ over $C$ is
the same.

Both theorems are an immediate consequence of the following proposition.

\begin{proposition}\label{prop:back-and-forth}
Let $\pair{\Bm, \Am}$ and $\pair{\Bm', \Am'}$ be 
models of~$\Td$.
Let $\cv$ be a (possibly infinite) \Uindependent tuple in~$\eq{\Bm}$,
and $\cv'$ be an \Uindependent tuple in~$\eq{(\Bm')}$ of the same length and the same sorts. 
If $\cv \elem^1 \cv'$ and $\Utp(\cv) = \Utp(\cv')$, then
$\cv \elem^2 \cv'$.
\end{proposition}
\begin{proof}
Back-and-forth argument.
Let $\lambda < \kappa$ be a cardinal strictly greater than $\card T$ and the
length of~$\cv$.
\Wlog, we can assume that both $\pair{\Bm, \Am}$ and $\pair{\Bm', \Am'}$ are
$\lambda$-saturated.
Let $\ev$ (resp.~$\ev'$) be the subtuple of $\cv$ (resp. of~$\dv'$)
of non-real elements.
Let
\begin{multline*}
\Gamma := \bigl\{f: \tilde c \to \tilde c':\quad
 \cv \subset \tilde c \subset \eq{\Bm}, \quad
\cv' \subset \tilde c' \subset \eq{(\Bm')},\\
\tilde c \et \tilde c' \text{ of the same length less than } \lambda \text{
  and of the  same sorts}, \\
\text{with all non-real elements of } \tilde c \text{ in } \ev,\\
f \text { is a bijection},\\
\tilde c \et \tilde c' \text{ are \Uindependent},\quad
\tilde c \elem^1 \tilde c',\quad
\Utp(\tilde c) = \Utp(\tilde c')\bigr\}.
\end{multline*}
We want to prove that $\Gamma$ has the back-and-forth property.
So, let $f: \tilde c \to \tilde c'$ be in $\Gamma$, and $d \in \Bm \setminus \cv$;
we want to find $g \in \Gamma$ such that $g$ extends $f$ and $d$ is in the
domain of~$g$.
\Wlog, $\tilde c = \cv$ and $\tilde c' = \cv'$.
Let $\av := \cv \cap \Am$, and $\av' := \cv' \cap \Am'$.
Notice that $f(\av) = \av'$ and that 
$\Am \cap \mat(\cv) = \Am \cap \mat(\av)$, and similarly for $\cv'$.
We distinguish some cases.

\Case 1 $d \in \Am \cap \mat(\cv) = \Am \cap \mat(\av)$.
Notice that $\cv d \ind_{\av d} \Am$, and therefore $\cv d$ is \Uindependent.
There is a \xnarrow formula $\phi(x ,\y)$ such that $\Bm \models \phi(d, \av)$.
Choose $d' \in \Am'$ such that $\cv d \elem^1 \cv' d'$; therefore,
$d' \in \mat(\av')$, and thus $\cv' d'$ is also \Uindependent and has the
same \Utype as $\cv d$.
Thus, we can extend $f$ to $\cv d$ setting $g(d) = d'$.

\Case 2 $d \in \Am \setminus \mat(\cv) = \Am \setminus \mat(\av)$.
Since $\cv \ind_{\av} \Am$ and $c \in \Am$, we have
$\cv \ind_{\av d} \Am$, and therefore $\cv d$ is \Uindependent.
Let $q(x) := \tp^1(d /\cv)$, and $q' := f(q) \in S^1_1(\cv')$.
Notice that $q \ind_{\av} \cv$ (because $d \ind_{\av} \cv$), and
therefore $q' \ind_{\av'} \cv'$.
Since $\Am'$ is \emph{dense} in $\Bm'$ and $\pair{\Bm', \Am'}$ is
$\lambda$-saturated, there exists $d' \in \Am'$ realizing $q'$.
It is now easy to see that $\cv' d'$ is \Uindependent, and that we can
extend $f$ to $\cv d$ setting $g(d) = d'$.

\Case 3 $d \in \mat(\cv \Am) \setminus \Am$.
Let $\av_0 \in \Am^n$ such that $d \in \mat(\bv \av_0)$ 
($\av_0$ exists because $\mat$ is finitary).
By applying $n$ times cases~1 or~2, we can extend $f$ to $f' \in \Gamma$ such
that $\av_0$ is a subset of the domain of~$f'$.
By substituting $f$ with $f'$, we are reduced to the case that 
$d \in \mat(\cv) \setminus \Am$.
Hence, $\cv \ind_{\av} \Am$ and $d \in \mat(\cv)$: therefore, 
$\cv d \ind_{\av} \Am$, and hence $\cv d$ is \Uindependent.
Let $d' \in \Bm'$ such that $d' \cv' \elem^1 d \cv$.
For the same reason as above, $\cv' d'$ is also \Uindependent.
It remains to show that $\cv d$ and $\cv' d'$ have the same \Utype,
that is that $d' \notin \Am'$.
If, for contradiction, $d' \in \Am' \cap \mat(\cv)$, then $d' \in \mat(\av)$;
therefore, there would be a \xnarrow-formula witnessing it, 
and thus $d \in \mat(\av) \subseteq \Am$, absurd.

\Case 4 $d \notin \mat(\cv \Am)$.
Let $\av_0 \subset \Am$ be of cardinality less than $\lambda$
such that $d\ind_{\av_0 \av} \Am$ 
($\av_0$ exists because $\ind$ satisfies Local Character).
By applying cases 2 and 3 sufficiently many times,
we can extend $f$ to  $f' \in \Gamma$ such that $\av_0$ is contained in the
domain of $f'$; thus, \wloG, $d \ind_{\av} P$.
Let $d' \in Am'$ such that $d' \cv' \elem^1 d \cv$; moreover, 
by Lemma~\ref{cor:dense-cardinality} we can also assume that 
$d' \ind_{\av'} \Am'$.
We need only to show that $d' \notin \Am'$.
Assume, for contradiction, that $d' \in \Am'$ and $d' \ind_{\av'} \Am'$; then,
$d' \ind_{\av'} d'$, thus $d' \in \mat(\av')$, and hence $d \in \mat(\av)$, 
absurd.
\end{proof}

\subsection{Additional facts}
Reasoning as in~\cite[2.6--2.9]{vdd-dense}, 
from Theorems~\ref{thm:Td-complete} and~\ref{thm:Td-eq}, 
and Proposition~\ref{prop:back-and-forth}, we can deduce the following facts.
We are still assuming that $T$ expands an integral domain.
We will say that $A$ is $\Ttwo$-algebraically closed if $A$ is a subset
of some pair $\pair{\Bm, \Cm} \models \Td$ and $A$ is algebraically closed \wrt
the language~$\Ltwo$, and similarly for $\Ttwo$-definably closed.
We denote by $\acl^1$ the algebraic closure in~$T$, and by $\acl^2$ the
algebraic closure in~$\Ttwo$.
We denote by $\elem^1$ elementary equivalence \wrt the language~$\Lang$, and
by We denote by $\elem^2$ elementary equivalence \wrt the language~$\Ltwo$.
Similarly, $\tp^1(\bv / X)$ will denote the $\Lang$-type of $\bv$ over~$X$,
while $\tp^2(\bv/X)$ will denote the $\Ltwo$ type.

\begin{corollary}[{\cite[2.6]{vdd-dense}}]\label{cor:vdd-26}
Let $\pair{\Bm, \Am}$ be a model of~$\Td$.
Suppose $Y \subseteq \Bm^n$ is
$A_0$-definable in $\pair{\Bm, \Am}$, for some $A_0 \subset \Am$.
Then $Y \cap \Am^n$ is $A_0$-definable in~$\Am$.
\end{corollary}

\begin{corollary}[{\cite[2.7]{vdd-dense}}]\label{cor:vdd-27}
Let $\pair{\Bm, \Am}$ and $\pair{\Bm', \Am'}$ be models of~$\Td$,
such that $\pair{\Bm', \Am'} \subseteq \pair{\Bm, \Am}$
and $\Bm'$ and $\Am$ are \clindependent over $\Am'$.
Then, $\pair{\Bm', \Am'} \preceq \pair{\Bm, \Am}$.
In particular, if $\Am \prec \Bm' \preceq \Bm$, with $\Am \neq \Bm'$,
then $\pair{\Bm', \Am} \preceq \pair{\Bm, \Am}$.
\end{corollary}

\begin{corollary}[{\cite[2.8]{vdd-dense}}]\label{cor:Td-2extensions}
Let $A \subseteq B \subset \monster$ be substructures.
Assume that $\pair{B, A}$ have extensions 
$\pair{\Bm_1, \Am_1} \models \Td$ and $\pair{\Bm_2, \Am_2} \models \Td$,
such that $B \ind_{A} \Am_k$ and $B \cap \Am_k = A$, $k = 1, 2$.
Then, $\pair{\Bm_1, \Am_1} \elem^2_{B} \pair{\Bm_2, \Am_2}$, that is
$\pair{\Bm_1, \Am_1}$ and $\pair{\Bm_2, \Am_2}$ satisfy the same
$\Ltwo$-formulae with parameters from~$B$.
More generally, for every $\av_1 \in (\Am_1)^n$ and $\av_2 \in (\Am_2)^n$,
if $\av_1 \elem^1_B \av_2$, then $\av_1 \elem^2_B \av_2$; that is, if 
$\av_1$ and $\av_2$ realise the same $\Lang$-types over $B$ in $\Bm_1$ and
$\Bm_2$ respectively, then they realise the same $\Ltwo$-type over $B$ in 
$\pair{\Bm_1, \Am_1}$ and $\pair{\Bm_2, \Am_2}$ respectively.
\end{corollary}
Notice that the hypothesis of the above Corollary implies that $A$ is
\clclosed (but not necessarily \cldense) in~$B$.
\begin{proof}
Let $\cv_k := B \av_k$. 
$\cv_1$ and $\cv_2$ have the same \Utype, they are both \Uindependent,
and $\cv_1 \elem^1 \cv_2$;
the conclusion follows from Proposition~\ref{prop:back-and-forth}.
\end{proof}

\begin{corollary}[{\cite[2.9]{vdd-dense}}]\label{cor:Td-type-1}
Let $\pair{\Bm_1, \Am_1} \models \Td$ and $\pair{\Bm_2, \Am_2} \models \Td$,
and let $A$ be a common subset of $\Am_1$ and~$\Am_2$.
Suppose that $b_1 \in \Bm_1 \setminus \Am_1$ and 
$b_2 \in \Bm_2 \setminus \Am_2$
realise the same $\Lang$-types over~$A$ in $\Bm_1$ and $\Bm_2$ respectively,
that is $b_1 \elem^1_A b_2$.
Then, they realise the same $\Ltwo$-types over~$A$ in 
$\pair{\Bm_1, \Am_1}$ and $\pair{\Bm_2, \Am_2}$ respectively,
that is $b_1 \elem_A^2 b_2$.
\end{corollary}
\begin{proof}
Let $\cv_i := b_i \Am_i$, $i = 1, 2$.
By assumption, $\cv_1 \elem^1 \cv_2$, they have the same \Utype, and
they are both \Uindependent.
The conclusion follows from Proposition~\ref{prop:back-and-forth}.
\end{proof}

For the remainder of this section, we will assume that $\pair{\Bm, \Am}$
is a model of~$\Td$, and that $\lambda$ is a cardinal number such that
$\kappa > \lambda  > \max(\card T, \card{\Bm})$.

\begin{lemma}[{\cite[Theorem~2]{vdd-dense}}]
\label{lem:thm-2}
Let $\bv \subset \Bm$ be \Uindependent.
Given a set $Y \subset \Am^n$, \tfae:
\begin{enumerate}
\item 
$Y$ is $\Ttwo$-definable over~$\bv$;
\item 
$Y = Z \cap \Am^n$ for some set $Z \subseteq \Bm^n$ that is $T$-definable 
over~$\bv$.
\end{enumerate}
\end{lemma}

\begin{proof}
$(1 \Rightarrow 2)$ 
is as in \cite[Theorem~2]{vdd-dense}.
$(2 \Rightarrow 1)$ is obvious.
\end{proof}

\begin{lemma}[{\cite[3.1]{vdd-dense}}]
$\Am$ is $\Ttwo$-algebraically closed in $\pair{\Bm, \Am}$.
\end{lemma}
\begin{proof}
As in \cite[3.1]{vdd-dense}: let $b \in \Bm \setminus \Am$.
Let $\pair{\Bm^*, \Am^*} \succeq \pair{\Bm, \Am}$ be a monster model.
Since $\mat$ is existential, and $b \notin \mat(\Am)$,
there exists infinitely many distinct
$b' \in \Bm^*$ such that 
$b \elem^1_{\Am} b'$. 
By Corollary~\ref{cor:Td-type-1}, $b \elem^2_{\Am} b'$.
Thus, $b$~is not $\Ttwo$-$\Am$-algebraic in $\pair{\Bm^*, \Am^*}$, 
and therefore not $\Ttwo$-$\Am$-algebraic in $\pair{\Bm, \Am}$.
\end{proof}

\begin{lemma}[{\cite[3.2]{vdd-dense}}]\label{lem:3.2}
Let $A_0 \subseteq \Am$ be $T$-algebraically closed ($T$-definably closed).
Then $A_0$ is $\Ttwo$-algebraically closed (resp.\ $\Ttwo$-definably closed).
\end{lemma}
\begin{proof}
Assume $A_0$ is $T$-algebraically closed.
Let $c \in \acl^2(A_0)$, and $C := \set{c_1, \dotsc, c_n}$ be the set of
$\Ltwo$-conjugates of $c/ A_0$.
By definition, $C$~is $A_0$-definable in $\pair{\Bm, \Am}$, and, by the above
Lemma, $C \subset \Am$.
Hence, by Corollary~\ref{cor:vdd-26}, $C$~is $A_0$-definable in~$\Am$.
The case when $A_0$ is $T$-definably closed is similar.
\end{proof}


\begin{lemma}\label{lem:orbit-dense}
Let $\pair{\Bm, \Am}$ be a $\kappa$-saturated model of~$\Td$.
Let $D \subset \Bm$ such that $\card D < \lambda$,
and $c \in \Bm \setminus \mat(D)$.
Define $C := \set{c' \in \Bm: c' \equiv^1_D c} \cap A$.
Then, $\card{C} \geq \lambda$.
\end{lemma}
\begin{proof}
Consider the following partial $\Ltwo$-type over~$D$:
\[
p(x_i: i < \lambda) :=
\PA{\bigwedge_i x_i \elem^1_D c} \et \PA{\bigwedge_i U(x_i)} \et 
\PA{\bigwedge_{i < j} x_i \neq x_j}.
\]

\begin{claim}
$p$ is consistent.
\end{claim}
If not, there exist $\dv \subset D$, $\bv \subset \Bm$,
$\phi(x, \dv) \in \tp^1(c/ D)$, such that
$\phi(\Bm, \dv) \setminus \Am = \bv$.
Let $X := \phi(\Bm, \dv) \setminus \bv$: notice that $X$ is definable
in~$\Bm$, and $X \subseteq \Am$.
Hence, since $\Am$ is co-dense in~$\Bm$, we conclude that $\dim(X) \leq 0$, and
therefore $\dim(\phi(\Bm, \dv)) \leq 0$.
Thus, $c \in \mat(\dv) \subseteq \mat(D)$, absurd.

The conclusion follows immediately from the claim.
\end{proof}

\begin{proposition}[{\cite[3.3]{vdd-dense}}]\label{prop:3.3}
Let $\bv \subset \Bm$ be \Uindependent.
Then, $\dcl^2(\bv) = \dcl^1(\bv)$, and similarly for the algebraic closure.
Let $c \in \eq{\Bm}$ (\ie, $c$ is an imaginary element for the
structure~$\Bm$).
Then, $c \in \dcl^2(\bv)$ iff $c \in \dcl^1(\bv)$, and similarly for the
algebraic closure. 
\end{proposition}

\begin{proof}[Idea for proof]
\Wlog, we can assume that $\pair{\Bm, \Am}$ is $\lambda$-saturated and that
$\bv$ has finite length.
So, let $c$ be a $\Bm$-imaginary such that
$c \in \acl^2(\bv)$.  We have to prove that $c \in \acl^1(\bv)$.

Let $\Bm_1 := \mat^{\Bm}(\Am \bv)$; by Corollary~\ref{cor:vdd-27},
$\pair{\Bm_1, \Am} \preceq \pair{\Bm, \Am}$, and in
particular $\Bm_2$ is $\Ttwo$-algebraically closed in $\pair{\Bm, \Am}$, and
therefore $c \in \eq{\Bm_1}$.
Let $n \geq 0$ minimal and $\av \in \Am^n$ such that $c \in \mat(\bv \av)$.

\begin{claim}\label{cl:33-2}
$c \in \mat(\bv)$, \ie $n = 0$.
\end{claim}
If $n > 0$, by substituting $\bv$ with 
$\bv a_1, \dotsc, a_{n-1}$, and proceeding by induction on~$n$, 
we can reduce to the case $n = 1$;  let $a := a_1$.
Consider the following partial $\Lang$-type over $\bv a$:
\[
q(x) := (x \elem^1_{\bv} a) \et (x \ind_{\bv} a).
\]
Since $\ind$ satisfies Existence, $q$~is consistent.
Let $d \in \Bm$ be any realisation of~$q$.
Since $d \ind_{\bv} a$, we conclude that either $d \notin \mat(\bv a)$ 
or $d \in \mat(\bv)$. 
However, the latter cannot happen, since $d \elem^1_{\bv} a \notin
\mat(\bv)$; thus, $d \notin \mat(\bv a)$, and therefore $\dim(q) = 1$.
Hence, since $\Am$ is dense in $\Bm$ and $\pair{\Bm, \Am}$ is
$\omega$-saturated, there exists $a' \in \Am$ satisfying~$q$.
Reasoning in the same way, we can show that there exists
$(a_2, a_3, a_4, \dotsc)$ a Morley sequence in $q$ contained in~$\Am$.
By Corollary~\ref{cor:Td-2extensions}, $a_i \elem^2_{\bv} a$ for every~$i$.
Let $c_1$, $c_2, \dotsc, c_m$ be all the $\Ltwo$-conjugates of $c$ over~$\bv$
(there are finitely many of them), and let $\phi(x, y, \z)$ be an \xnarrow
$\Lang$-formula without parameters 
such that $\Bm \models \phi(c, a, \bv)$.

The $\Lang$-formula (in~$y$, with parameters in $\bv c_1, \dotsc, c_m$) 
$\bigvee_i \phi(c_i, y, \bv)$ is equivalent to an $\Ltwo$-formula in $y$ 
\emph{with parameters $\bv$}; hence, every $a_i$ satisfies it 
(because $a_i \elem^2_{\bv} a$).
Hence, \wloG $c_1 \in \mat(\bv a_2) \cap \mat(\bv a_3) = \mat(\bv)$
(because $a_2 \ind_{\bv} a_3$).
Therefore, $c \in \mat(\bv)$.

It remains to show that $c \in \acl^1(\bv)$.
Let $c_2 \in \eq{\Bm}$ such that $c_2 \elem^1_{\bv} c$.
Since $\Bm$ is $\omega$-saturated, 
it suffices to prove that there are only finitely many such~$c_2$.
Since $c \in \acl^2(\bv)$, it suffices to prove that $c_2 \elem^2_{\bv} c$.
Let $\bv_1 := \bv c$,  $\bv_2 := \bv c_2$, and
$\av := \bv \cap \Am$.
By assumption, $\bv_1 \elem^1 \bv_2$.
By Claim~\ref{cl:33-2}, we have $\bv_1 \subseteq \mat(\bv)$, and therefore,
since $\bv \ind_{\av} \Am$, $\bv_1$ is \Uindependent.
Claim~\ref{cl:33-2} also implies that $\bv_2 \subseteq \mat(\bv)$, and hence
$\bv_2$ is also \Uindependent.
It remains to show that $\bv_1$ and $\bv_2$ have the same \Utype.
Assume \eg that $c \in \Am$.
Since $\bv \ind_{\av} \Am$, we have that $c \in \mat(\av)$, 
and therefore $c_2 \in \mat^{\Bm}(\av) \subseteq \Am$.
%
%


The other assertions are proved in the same way.
\end{proof}

\subsection{The small closure}
\label{subsec:small-closure}

We will are still assuming that $T$ expands an integral domain.
Let $\monster^* := \pair{\Bm^*, \Am^*}$ be a $\kappa$-saturated and strongly
$\kappa$-homogeneous monster model of~$\Td$,
and $\pair{\Bm, \Am} \prec \monster^*$, with $\card{\Bm} < \kappa$.
Notice that $\rk(\Bm^* / \Am^*) \geq \kappa$.

\begin{definizione}
For every $X \subseteq \Bm^*$ we define the \intro{small closure} of $X$ as
\[
\scl(X) := \mat(X \Am^*).
\]
\end{definizione}
For lovely pairs (\eg, dense pairs of o-minimal structures), 
the small closure was already defined in~\cite{berenstein}.

\begin{remark}
$\scl$ is a definable matroid (on~$\monster^*$).
\end{remark}
\begin{proof}
$\scl$ coincides with the operator $\mat_{\Am^*}$ in Lemma~\ref{lem:cl-pair}.
\end{proof}

Notice that we can apply Lemma~\ref{lem:cl-small-relative}: 
$\scl^{\Bm} = (\mat^{\Bm})_{\Am}$: we can ``compute''
the small closure of a subset of $\Bm$ inside $\Bm$ using $\Am$ 
instead of~$\Am^*$.

We want to prove that $\scl$ is existential; we will need a preliminary lemma.

\begin{lemma}
Let $b \in \Bm^* \setminus \Am^*$.
Define $\monster^*_b$ the expansion of $\monster^*$ with a constant for~$b$, 
and $\scl_b(X) := \scl(b X) = \mat(X \Am^* b)$.
Then, $\scl_b$ is an existential matroid on~$\monster^*_b$.
\end{lemma}
\begin{proof}
That $\scl_b$ is a definable matroid follows from Lemma~\ref{lem:cl-pair},
applied to~$\scl$.
Let $X \subseteq \monster^*$, and $Y := \scl_b(X)$.
\begin{claim}
$Y \prec \monster^*$ (as an $\Ltwo$-structure).
\end{claim}
By Lemma~\ref{lem:density-closure}, $Y$~is an elementary $\Lang$-substructure
of~$\Bm^*$.
By Theorems~\ref{thm:Td-complete}, and~\ref{thm:Td-eq} and, again, 
Lemma~\ref{lem:density-closure}, it suffices to show
that $\Am^*$ is a \clclosed, \cldense, and proper subset of~$Y$, which is
trivially true.

The lemma then follows from the above claim and
Lemma~\ref{lem:cl-Skolem-existence}; non-triviality follows 
from the fact that $\zrk(\Bm^*/\Am^*) \geq \kappa$.
\end{proof}

\begin{lemma}\label{lem:scl}
$\scl$ is an existential matroid.
\end{lemma}
\begin{proof}
The only thing that needs proving is Existence.
Define $\Xi^2(a /C)$ as the set of conjugates of $a$ over $C$ in~$\monster$.
Assume that $\Xi^2(a/C) \subseteq \scl(C D)$.
We want to prove that $a \in \scl(C)$.
Choose $b, b' \in \Bm^*$ which are $\mat$-independent over~$\Am^* C$.
By  applying the previous lemma to $\scl_b$ and $\scl_{b'}$, we see that
\[
a \in \scl_b(C) \cap \scl_{b'}(C) = \mat(\Am^* C b) \cap \mat(\Am^* C b')
= \mat(\Am^* C) = \scl(C). 
\qedhere
\]
\end{proof}

Hence, we can define the dimension induced by~$\scl$, and denote it by~$\sdim$.

Notice that, by Theorem~\ref{thm:cl-unique}, $\scl$~is the only existential matroid on~$\Td$.

\begin{lemma}
Let $X \subseteq (\Bm^*)^n$ be definable in~$\Bm^*$.
Then $\sdim(X) = \dim(X)$.
\end{lemma}
\begin{proof}
From $\mat \subseteq \scl$ follows immediately that $\sdim(X) \leq \dim(X)$.
For the opposite inequality, we proceed by induction on $k := \dim(X)$.
Assume, for contradiction, that $\sdim(X) < k$.
\Wlog, $\dim(\Pi^n_k(X)) = k$; therefore, \wloG $k = n$.
If $k = 1$, then $\sdim(X) = 0$, and therefore $F^4(X) \neq {\Bm^*}^n$,
contradicting $\dim(X) = 1$.
For the inductive step, assume $k = n > 1$, and let 
$U := \set{a \in \Bm^n: \dim(X_a) = n - 1}$.
$U$ is definable in~$\Bm$, and therefore, by inductive hypothesis,
$\sdim(U) = \dim(U) = n - 1$.
By the case $k = 1$, for every $a \in \K^{n-1}$, $\dim(X_a) = \sdim(X)$, and
therefore $\sdim(X_a) = 1$ for every $a \in U$.
Thus, $\sdim(X) = n$.
\end{proof}

\begin{definizione}
Let $X \subseteq (\Bm^*)^n$ be definable in $\pair{\Bm^*, \Am^*}$.
We say that $X$ is small if $\sdim(X) = 0$.
Let $Y \subseteq \Bm^n$ be definable in $\pair{\Bm, \Am}$.
We say that $Y$ is small if $\sdim(Y^*) = 0$, where $Y^*$ is the
interpretation of $Y$ inside $\pair{\Bm^*, \Am^*}$.
\end{definizione}

Notice that, if $X \subset \Bm^n$ is $\Am$-small (in the sense of
Definition~\ref{def:small}), then $X$ is also small in the above sense.
The next lemma shows that the converse is also true.

\begin{lemma}\label{lem:pair-small}
Let $\pair{\Bm, \Am} \preceq \pair{\Bm^*, \Am^*}$ and $X \subseteq \Bm^n$ be
definable in $\pair{\Bm, \Am}$. 
Let $X^*$ be the interpretation of $X$ inside $\pair{\Bm^*, \Am^*}$. 
Let $\cv \in \Bm^k$ be the parameters of definition of~$X$.
\Tfae:
\begin{enumerate}
\item $X$ is small;
\item $X^*$ is small;
\item 
$X^* \subseteq \scl(\bv)$ for some finite tuple $\bv \subset \Bm^*$;
\item
$X^* \subseteq \scl(\cv)$;
\item
$X^* \subseteq \mat(\cv \Am^*)$;
\item $X$ is $\Am$-small: that is,
there exists a \Zapplication $f^*: (\Bm^*)^m \app (\Bm^*)^n$, 
definable in~$\Bm^*$ with parameters, 
such that $f^*\Pa{{\Am^*}^m} \supseteq X^*$;
\item $X^*$ is $\Am^*$-small: that is,
there exists a  \Zapplication $f: \Bm^m \app \Bm^n$, 
definable in~$\Bm$ with parameters~$\cv$, 
such that $f^*\Pa{{\Am^*}^m} \supseteq X^*$, 
where $f^*$ is the interpretation of $f$ in~$\Bm^*$;
\item
there exists a  \Zapplication $g: \Bm^{m + k} \app \Bm^n$, definable 
in~$\Bm$ without parameters, such that 
$g^*\Pa{{\Am^*}^m \times \set{\cv}} \supseteq X^*$;
\item
there exists a  \Zapplication $f: \Bm^m \app \Bm^n$, definable in~$\Bm$
without parameters, such that $f(\Am^m \times \set{\cv}) \supseteq X$. 
\end{enumerate}
\end{lemma}
\begin{proof}
The only non-trivial implication is $(5 \Rightarrow 7)$, which is proved by a
compactness argument using Remark~\ref{rem:Zapplication}.
\end{proof}

\begin{conjecture}[{\cite[3.6]{vdd-dense}}]
Let $f: \Am^n \to \Am$ be $\Ttwo$-definable with
parameters~$\bv$.
Let $\av \subset \Am^m$ such that $\bv \ind_{\av} \Am$ and
$\dcl^1(\bv \av) \cap \Am = \dcl^1(\av)$.
Then, $f$~is given piecewise by functions definable in $\Am$ with
parameters~$\av$. 
\end{conjecture}
\begin{proof}[Idea for proof]
By compactness, it suffices to show that, given an elementary extension
$\pair{\Bm^*, \Am^*}$ of $\pair{\Bm, \Am}$ and a point
$a^* \in (\Am^*)^n$, we have $f(a^*) \in \Am' := \dcl^1(\av a^*)$.
Let $\Bm' := \dcl^1(\bv \av a^*)$.
Since $\bv \ind_{\av} \Am$, we have
$\Bm' \ind_{\Am'} \Am$.
Hence, by Proposition~\ref{prop:3.3}, $\Bm'$ is $\dcl^2$-closed, and therefore
$f(a^*) \in \Bm'$. 
Hence,
\[
f(a^*) \in \dcl^1(\bv \av a^*) \cap \Am^* = \mat(\bv \av a^*) \cap \Am^* \cap
\dcl^1(\bv \av a^*) 
= \mat(\av a^*) \cap \dcl^1(\bv \av a^*).
\]
It remains to show that $\mat(\av a^*) \cap \dcl^1(\bv \av a^*) = \dcl^1(\av a^*)$.
\end{proof}

\begin{lemma}[{\cite[6.1.3]{boxall}}]
Let $f: \Am^n \to \Bm$ be $\Ttwo$-definable with parameters~$\bv$.
Assume that $\bv$ is \Uindependent.
Then, there exists $g: \Bm^n \to \Bm$ which is $T$-definable with
parameters~$\bv$, and such that $f = g \rest \Am^n$.
\end{lemma}
\begin{proof}
Let $\pair{\Bm^*, \Am^*}$ be an elementary extension of $\pair{\Bm, \Am}$ and 
$a^* \in (\Am^*)^n$. 
By Proposition~\ref{prop:3.3}, there exists a function $g_i: \Bm^n \to \Bm$
which is $T$-definable with parameters~$\bv$, such that $f(a) = g_i(a)$.
By compactness, finitely many $g_i$ will suffice.
The conclusion then follows from Lemma~\ref{lem:thm-2}.
\end{proof}

\begin{proposition}[{\cite[3.5]{vdd-dense}}]\label{prop:3.5}
Let $\bv \in \Bm^k$ and $\av \in \Bm^{k'}$ such that $\bv \ind_{\av} \Am$ and
$\bv \cap \Am \subseteq \av$.
Let $X \subseteq \Bm^n$ be $T$-definable (possibly, inside an imaginary sort) 
with parameters~$\bv$, such that $dim(X) = d$.
Let $Y \subseteq X$ be $T^2$-definable, with parameters~$\bv$.
Then, there exist $S \subset X$ which is $T^2$-definable with
parameters~$\bv$,
and $Z \subseteq X$ which is $T^2$-definable with parameters $\bv \av$,
such that $Z \sdiff Y \subseteq S$ and $\sdim(S) < d$.

In particular, if $\dim(X) = 0$, then every $T^2$-definable subset of $X$ is
already $T$-definable.
\end{proposition}
\begin{proof}
The proof is a variant of Beth's definability theorem, using
Proposition~\ref{prop:back-and-forth}.
\Wlog, $\pair{\Bm, \Am}$ is $\lambda$-saturated for some cardinal $\lambda$
such that $\card T < \lambda < \kappa$.

Let $W := \set{p \in S^2_X(\av \bv): \sdim(p) = d}$.
Notice that $W$ is a closed subset of $S^2_X(\av \bv)$ (the Stone space of
$\Ttwo$-types over $\av \bv$ containing the formula ``$x \in X$'').
Let $\theta: S^2_X(\av \bv) \to S^1_X(\av \bv)$ be the restriction map: notice
that $\theta$ is a continuous homomorphism, and therefore
$V := \theta(W)$ is closed in $S^1_X(\av \bv)$.
Let $\rho := \theta \rest W$.
\begin{claim}
$\rho$ is injective (and therefore $\rho$ is a homeomorphism between $W$ and $V$).
\end{claim}
We have to prove that, for every $\cv$ and $\cv' \in X$, 
if $\srk(\cv / \av \bv) = \srk(\cv' / \av \bv) =d$ and $\cv \elem^1_{\av \bv} \cv'$,
then $\cv \elem^2_{\av \bv} \cv'$.
Let $\dv := \av \bv \cv$ and $\dv' := \av \bv \cv'$.
By Proposition~\ref{prop:back-and-forth}, it suffices to prove that
$\dv$ and $\dv'$ are both \Uindependent and have the same \Utype.
Since $\srk(\cv / \av \bv) = d$ and $\cv \in X$, we have that
$\srk(\cv / \av \bv) = \rk(\cv / \av \bv)$, which is equivalent to
$\cv \ind_{\av \bv} \Am$, and hence (since $\bv \ind_{\av} \Am$)
$\dv \ind_{\av} \Am$, that is $\dv$ is \Uindependent, and similarly for
$\dv'$.
It remains to show that $\dv$ and $\dv'$ have the same \Utype.
Let $d_i \in \Am$; we have to prove that $d_i' \in \Am$.
Since $\dv \ind_{\av} \Am$, we have $d_i \in \mat(\av)$, and hence
$d_i' \in \mat^{\Bm}(\av') \subseteq \Am$.

Let $U := S^2_Y(\av \bv) \cap W$; since $Y$ is definable, $U$ is clopen
in~$W$, and since $\rho$ is a homeomorphism, $\rho(U)$ is clopen in~$V$.
Hence, there exists $Z$ subset of~$X$, such that $Z$ is
$T$-definable over $\av \bv$ 
and $V \cap S^1_{Z_1}(\av \bv) = \rho(U)$. 
\begin{claim}
There exists $S \subset X$ which is $T_2$-definable over~$\bv$, such that
$\sdim(S) < d$ and $Y \sdiff Z \subseteq S$.
\end{claim}
Assume not. Then, the following partial type over $\av \bv$ is consistent:
\[
\Phi(\x) := \x \in X \et \x \in Y \sdiff Z \et \x \notin S,
\]
where $S$ varies among the subsets of $X$ 
which are $T^2$-definable over $\bv$, with $\sdim(S) < d$.
Let $\cv \in X$ be a realization of~$\Phi$ and $p := \tp^2(\cv/\av \bv) \in
S^2_X(\av \bv)$.
By assumption, $\sdim(\cv/ \av \bv) = d$, and therefore
$p \in W$.
Hence, $\rho(p) = \tp^1(\cv/ \av \bv) \in V$.
Since $\rho$ is injective, we have
\[
\rho(p) \in \rho \Pa{S^2_Y(\av \bv) \cap W} \sdiff 
\rho\Pa{S^2_Z(\av \bv) \cap W} \subseteq
S^1_Z(\av \bv) \sdiff S^1_Z(\av \bv) = \emptyset,
\]
absurd.
\end{proof}

In general, given $\bv \in \Bm^n$, it is always possible to find
$\av \in \Am^{n'}$ such that $\bv \ind_{\av} \Am$.
However, there are some examples when $\Bm$ is o-minimal, but $\av$ cannot be
found inside $\dcl^2(\bv)$.%
\footnote{Thanks to J.~Ramakrishnan for pointing this out.}

\begin{corollary}[{\cite[3.4]{vdd-dense}}]\label{cor:3.4}
Let $\bv$ and $\av$ be as in the above Proposition.
Let $\Gamma$ be a $T$-definable set (possibly, in some imaginary sort)
over~$\bv$, and let the function $f: \Bm^n \to \Gamma$ be $\Ttwo$-definable
with parameters~$\bv$.
Then, there exist $S \subseteq \Bm^n$, which is $T^2$-definable over $\bv$ and
with $\sdim(S) < n$, and $\hat f: \Bm^m \to \Gamma$, which is $T$-definable
over $\bv \av$, such that $f$ agrees with $\hat f$ outside~$S$.
\end{corollary}
\begin{proof}
\Wlog, $\pair{\Bm, \Am}$ is $\kappa$-saturated.

\begin{claim}
There exists a set $S\subset \Bm^n$ which is $T^2$-definable with
parameters~$\bv$, with $\sdim(S) < n$,
and finitely many functions $g_1, \dots, g_k: \Bm^n \to \Gamma$ that are
$T$-definable with parameters $\bv\av$, such that $f$ agrees off $S$ with some
of the $g_i$.
\end{claim}
Assume that the claim does not hold.
Hence, for every $S$ and every $g$ as in the claim, there exists
$\cv \in \Bm^n$ such that $\cv \notin S$ and $f(\cv) \neq g(\cv)$.
Thus, the following partial $\Ltwo$-type over~$\bv\av$ is consistent:
\[
p(\x) := \Pa{\x \in \Bm^n \setminus \scl(\bv)} \et \Pa{f(\x) \neq g(\x)},
\]
where we let $g: \Bm^n \to \Gamma$ vary among the functions that are
$T$-definable with parameters $\bv\av$.
Let $\cv$ be a realisation of~$p$.
Notice that the choice of $\av$ and the fact that $\srk(\cv / \av \bv ) =n$
imply that $\cv \bv\av \ind_{\av} \Am$.
Hence, by Proposition~\ref{prop:3.3}, $f(\cv) \in \dcl^1(\cv \bv \av)$.
Hence $f(\cv) = g(\cv)$ for
some function $g: \Bm^n \to \Bm$  which is $T$-definable 
with parameters $\bv\av$, absurd.

The above Claim plus Proposition~\ref{prop:3.5} imply the conclusion.
\end{proof}

The above Corollary gives a way to find the parameters of
definition of $\hat f$ (and of~$S$) starting from the parameters $\bv$ of~$f$.
\begin{example}
In general, $\hat f$ cannot be defined using only $\bv$ as parameters.
Consider $a_1$ and $a_2$ in $\Am$ which are independent over the
empty set, $b_1 \in \Bm \setminus \Am$, and $b_2 := a_1 + b_1 a_2 \in \Bm
\setminus \Am$.
Let $\av := \pair{a_1, a_2}$ and $\bv := \pair{b_1, b_2}$.
Notice that $\rk (\av \bv) = 3$, while $\srk(\av \bv) = 2$.
Let $f$ be the constant function~$a_1$.
Then, $f$ is $\Ttwo$-definable over $\bv := \pair{b_1, b_2}$, but is not
$T$-definable over~$\bv$.
\end{example}

\begin{question}\label{question:dmin-scl-1}
Assume that $T$ is \dminimal (see \S\ref{sec:dmin}).
Is it true that, for every $X \subseteq \Bm^*$, $\scl(X) = \acl(\Am^*  X)$?
\end{question}

\begin{conjecture}[J. Ramakrishnan]
Assume that $T$ is o-minimal.
Then, for every $X \subset \Bm$,
\[
\acl^2(X) = \acl^1\Pa{X \cup (\acl^2(X) \cap \Am) }.
\]
\end{conjecture}



\subsection{Elimination of imaginaries}
Let $\mat$ be an existential matroid on~$\monster$.
Remember that element $e \in \monstereq$ is an equivalence class $X \subseteq
\monster^n$ for some $\emptyset$-definable equivalence relation $E$
on~$\monster^n$. If $\cv \in X$ we say that $\cv$ \intro{represents}~$e$.

\begin{definizione}
We say that $\monster$ has \intro{\clelimination of imaginaries} if, for
every $e \in \monstereq$, there exists $\cv$ representing~$e$, such that
$\cv \in \matt(e)$.
Given $\bv \subset \monster$, we say that $\monster$ has
\clelimination of imaginaries \intro{modulo $\bv$} if, for
every $e \in \monstereq$, there exists $\cv$ representing~$e$, such that
$\cv \in \matt(e \bv)$.

If $\K \preceq \monster$ we say that $\K$ has \clelimination of imaginaries
(modulo some $\bv \subset \K$) if $\monster$ has it.
\end{definizione}
Compare the above notion with weak elimination of imaginaries (see~\cite{CF}).

\begin{proposition}\label{prop:cl-elimination}
Let $\bv \subset \monster$.
Assume that $\mat(\bv)$ is \cldense in $\monster$.
Then, $\monster$ has \clelimination of imaginaries modulo~$\bv$.
\end{proposition}

\begin{corollary}
Let $\monster$ be geometric. 
Assume that $\acl(\emptyset)$ is $\acl$-dense in $\monster$ (\eg, $\monster$
is an algebraically closed field).
Then, $\monster$ has weak elimination of imaginaries.
If moreover $\monster$ expands a field, then $\monster$ has elimination of 
imaginaries.
\end{corollary}

\begin{corollary}
Assume that $\monster$ expands an integral domain.
Let $\pair{\Bm, \Am} \models \Td$.
Let $b \in \Bm \setminus \Am$.
Then, $\pair{\Bm, \Am}$ has $\scl$\hyph elimination of imaginaries modulo~$b$.
\end{corollary}
\begin{proof}
$\scl(b)$ is $\scl$-dense in  $\pair{\Bm, \Am}$ for every 
$b \in \Bm \setminus \Am$.
\end{proof}
In the situation of the above corollary, it is not true that $\pair{\Bm, \Am}$
has $\scl$\hyph elimination of imaginaries (modulo~$\emptyset$).
For instance, let $X := \Bm \setminus \Am$.
Then, $X \cap \scl^{eq}(\canon X) = \emptyset$.

Before proving the Proposition~\ref{prop:cl-elimination}, we need some
preliminaries. 
Let $X \subseteq \monster^n$ be a subset definable with parameters~$\bv$.
Let $\monster'$ be the expansion of $\monster$ with a new predicate
denoting~$X$. Notice that $\monster$ and $\monster'$ have the same definable
sets. 
However, $\mat$ is no longer an existential matroid on~$\monster'$: 
for instance, if $X = \set{b}$ is a singleton, and $b \notin \mat(\emptyset)$,
then $b \in \acl_{\monster'}(\emptyset) \setminus \mat(\emptyset)$, and
therefore $\mat$ is not existential on~$\monster'$.
Notice that $\indmat$ satisfies all the axioms of a symmetric independence
relation on~$\monster'$, except the Extension axiom.

Let $e := \canon X \in \monstereq$ be the canonical parameter for~$X$.
For every $Z \subseteq \monster$, define $\mat_e(Z) := \matt(e Z) \cap
\monster$ (notice that, if $e = \emptyset$, then $\mat_e = \mat$).

\begin{lemma}
$\mat_e$ is an existential matroid on~$\monster'$.
\end{lemma}
\begin{proof}
We only need to check that $\mat_e$ satisfies Existence.
Let $B$ and $C$ be subsets of $\monster$ such that $a \notin \mat_e(B)$, that
is $a \notin \matt(e B)$.
Let $a' \elem^{\monster}_{e B} a$, such that $a' \indmat_e B C$. 
Then, $a' \elem^{\monster'}_{B} a$ and $a' \notin \matt( e B C) = \mat_e(B C)$.
\end{proof}

\begin{proof}[Proof of Proposition~\ref{prop:cl-elimination}]
\Wlog, $\bv = \emptyset$.
Let $e \in \monstereq$. 
Let $E$ be a $\emptyset$-definable equivalence relation on
$\monster$ and $X$ be an equivalence class of~$E$, such that $e = X$.
Let $\mat_e$ be defined as above.
Since $\mat(\emptyset)$ is dense in $\monster$ and $\mat \subseteq \mat_e$,
$\K := \mat_e(\emptyset)$ is $\mat_e$-dense in~$\monster'$.
Hence, by Lemma~\ref{lem:density-closure}, $\K \preceq \monster'$.
Thus, since $X$ is $\emptyset$\hyph definable in $\monster'$, 
there exists $\cv \in X \cap \K$.
\end{proof}


\section{D-minimal topological structures}
\label{sec:dmin}

In this section we will introduce d-minimal structures.
They are topological structures whose definable sets are particularly simple
from the topological point of view; they generalise o-minimal structures.
We will show that for d-minimal structure the topology induces a canonical
existential matroid, which we denote by $\zcl$.
Moreover, the abstract notion of density introduced in \S\ref{sec:density}
coincides with the usual topological notion.
Finally, if $T$ is a complete d-minimal theory expanding the theory of fields,
then in $\Td$ the condition that the smaller structure is \clclosed is
superfluous. 
Our definition of d-minimality extends an older definition by
C.~Miller~\cite{miller}, that applied only to linearly ordered structures.

Let $\K$ be a first-order topological structure in the sense
of~\cite{pillay87}.
That is, $\K$ is a structure with a topology, such that a basis of the
topology is given by $\set{\Phi(\K, \av): \av \in\K^m}$ for a certain formula
without parameters $\Phi(x, \y)$; fix such a formula $\Phi(x, \y)$, and denote
$B_{\av} := \Phi(\K, \av)$.
Examples of topological structures are valued fields, or ordered structures.
On $\K^n$ we put the product topology.%
\footnote{Allowing a different topology (\eg Zariski topology) might be a
better choice.}
Let $\monster \succeq \K$ be a monster model.
Given $X \subseteq \K^n$, we ill denote by $\cl X$ the topological closure of
$X$ inside~$\K^n$.

\begin{definizione}
$\K$ is 
\dminimal if: 
\begin{enumerate}
\item it it is $T_1$ (\ie, its points are closed);
\item it has no isolated points;
\item
for every $X \subseteq \monster$ definable subset 
(with parameters in~$\monster$), if $X$ has empty interior, 
then $X$ is a finite union of discrete sets.
\item 
for every $X \subset \K^n$ definable and discrete, 
$\Pi^n_1(X)$ has empty interior;
\item 
given $X \subseteq \K^2$ and $U \subseteq \Pi^2_1(X)$ definable sets,
if $U$ is open and non-empty, and $X_a$ has non-empty interior for every 
$a \in U$, then $X$ has non-empty interior.
\end{enumerate}
\end{definizione}

\begin{lemma}
Assume that $\K$ is 
\dminimal.
Let $Z \subset \K^2$ be definable, such that $\Pi^2_1(Z)$ has empty interior,
and $Z_x$ has empty interior for every $x \in \K$.
Then, $\theta(Z)$ has empty interior, where $\theta$ is the projection onto
the second coordinate.
\end{lemma}
\begin{proof}
By assumption, \wloG $\Pi^2_1(Z)$ is discrete and, for every $x \in \K$,
$Z_x$ is also discrete.
Therefore, $Z$~is discrete, and hence $\theta(Z)$ has empty interior.
\end{proof}

\begin{definizione}
Given $A \subset \monster$ and $b \in \monster$,
we say that $b \in \zcl(A)$ if there exists $X \subset \monster$ $A$-definable
such that $b \in A$ and $A$ has empty interior (or, equivalently, $A$~is discrete).
\end{definizione}


\begin{examples}\
\begin{enumerate}
\item $p$-adic fields and algebraically closed valued fields are \dminimal;
\item densely ordered o-minimal structures are \dminimal.
\end{enumerate}
\end{examples}

\begin{example}\label{ex:dmin}
A structure $\K$ is \intro{definably complete} if it expands a linear 
order~$\pair{K, <}$, and every $\K$-definable subset of $K$ has a supremum in 
$K \sqcup \set  {\pm   \infty}$.
C.~Miller defines a \dminimal structure as a definably complete structure~$\K$
such that, given $\K'$ an $\aleph_0$-saturated
elementary extension of~$\K$, every $\K'$-definable subset of $\K'$ 
is the union of an open set and finitely many discrete sets.
In particular, o-minimal structures and ultra-products of o-minimal structures
are \dminimal in Miller's sense.
If $\K$ expands a field and is a \dminimal structures in the sense of Miller,
then $\K$ is \dminimal in our sense (the proof will be given elsewhere).
Conversely, any definably complete structure which is \dminimal in our sense
is also \dminimal in Miller's sense.
\end{example}

\begin{proviso*}
For the remainder of this section,
we assume that $\K$ is \dminimal, and $T$ is the theory of~$\K$.
\end{proviso*}


\begin{remark}
If $A$ and $B$ are definable subsets of $\K$ with empty interior, then
$A \cup B$ has empty interior.
Hence, for every $C \subseteq \K$ definable, $\bd(C)$ has empty interior.
\end{remark}

\begin{lemma}\label{lem:open-orbit}
If $c \notin \zcl(A)$, then $\Xi(c /A)$ has non-empty interior.
\end{lemma}
\begin{proof}
Let $X \subseteq \monster$ be an $A$-definable set containing~$c$.
Since $c \notin \zcl(A)$, $c \in \inter X$.
Consider the partial type over~$c A$
\[
\Gamma(\y) := \set{\Phi(c, \y) \et \Phi(\monster, \y) \subseteq X},
\]
where $X$ varies among the $A$-definable sets containing~$c$.
By the above consideration, $\Gamma$ is consistent;
let $\bv \subset \monster$ be a realisation of~$\Gamma$.
\begin{claim}
$c \in \Phi(\monster, \bv) \subseteq \Xi(c/ A)$.
\end{claim}
In fact, if $c \in X$, where $X$ is $A$-definable, then, by definition,
$c \in \Phi(\monster, \bv) \subseteq X$, and therefore any $x' \in
\Phi(\monster, \bv)$ satisfies all the $A$-formulae satisfied by~$c$.
\end{proof}

\begin{thm}
$\zcl$ is an existential matroid.
\end{thm}
\begin{proof}
Finite character, extension and monotonicity are obvious.\\
The fact that $\zcl$ is definable is also obvious.
\smallskip\\ 
Idempotence) Let $\bv := \pair{b_1, \dotsc, b_n}$, $a \in \zcl(\bv \cv)$ and 
$\bv \subset \zcl(\cv)$. 
We must prove that $a \in \zcl(\cv)$.
Let $\phi(x, \y, \z)$ and $\psi_i(y, \z)$ be formulae, $i =1, \dotsc, n$,
such that $\phi(\K, \y, \z)$ and $\psi_i(\K, \z)$ are discrete for every $\y$
and~$\z$, and $\K \models \phi(a, \bv, \cv)$ and $\K \models \psi_i(b_i,
\cv)$, $i = 1, \dotsc, n$.
Let
\[
Z := \set{\pair{x, \y}: \K \models \phi(x, \y, \cv) \et 
\bigvee_{i = 1}^n   \psi_i(y_i, z_i)},
\]
and $W := \Pi^{n+1}_1{Z}$.
By hypothesis, $Z$~is a discrete subset of $\K^{n + 1}$, and therefore, by
Assumption~(4), $W$~has empty interior.
Moreover, $W$ is $\cv$-definable and $a \in W$, and hence $a \in \zcl(\cv)$.
Notice that weak \dminimality suffices.
\smallskip \\ 
EP)
Let $a \in \zcl(b \cv) \setminus \zcl(\cv)$.
We must prove that $b \in \zcl(a \cv)$.
Assume not.
Let $Z \subset \K^2$ be $\cv$-definable, such that $\pair{a, b} \in Z$ and
$Z^y$ is discrete for every $y \in \K$.
Since $b \in Z_a$ and $b \notin \zcl(a \cv)$, $b \in \interior(Z_a)$;
hence, \wloG $Z_x$ is open for every $x \in \K$.
Let $U := \Pi^2_1(Z)$.
Since $a \in U$ and $a \notin \zcl(\cv)$, $a \in \inter U$.
Hence, by Condition~(5), $Z$~has non-empty interior; but this contradict the
fact $Z^y$ is discrete for every $y \in \K$.
\smallskip \\
Existence follows from Lemma~\ref{lem:open-orbit}.
\smallskip \\
Non-triviality)
Consider the following partial type over the empty set:
\[
\Lambda(x) := \set{x \notin Y},
\]
where $Y$ varies among the discrete $\emptyset$-definable sets.
Since $\K$~has no isolated points, $\Lambda$~is finitely satisfiable;
if $a \in \monster$ is a realisation of~$\Lambda$, 
then $a \notin \zcl(\emptyset)$.
\end{proof}

We will denote by $\zrk$ and $\dim$ the rank and dimension on $\monster$
according to~$\zcl$.
We will say \Zclosed instead of $\zcl$-closed.

\begin{remark}
Let $X \subseteq \K^n$ be definable.
If $X$ has non-empty interior, then $dim(X) = n$.
If $\Pi^n_d(X)$ has non-empty interior, then $\dim(X) \geq d$.
\end{remark}

\begin{conjecture}
Let $X \subseteq \K^n$ be definable.
Then, $\dim(X) = d$ iff, after a permutation of variables,
$\Pi^n_d(X)$ has non-empty interior.
\end{conjecture}

\begin{conjecture}
For every $X \subseteq \K^n$ definable,
$\dim(\cl X) = \dim X$.
\end{conjecture}
However, it is not true that $\dim(\partial X) < \dim(X)$ if $X$ is definable
and non-empty.

\begin{lemma}
$X$ is \cldense in $\K$ according to definition~\ref{def:dense} (\wrt~$\zcl$)
iff $X$ is topologically dense in~$\K$.
\end{lemma}
\begin{proof}
Assume that $X$ is \cldense in~$\K$ according to~$\zcl$.
Let $A \subseteq \K$ be an open definable set; thus, $\dim(A) = 1$, and
therefore $A \cap X \neq \emptyset$.
Conversely, if $X$ is topologically dense in~$\K$, let $A \subseteq \K$ be
definable and of dimension~1. 
Thus, $A$~has non-empty interior, and therefore $A \cap X \neq \emptyset$.
\end{proof}

\begin{proviso*}
For the remainder of this section, will assume that $\K$ is \dminimal and
expands an integral domain, that $+$ and $-$ are continuous (and
therefore $\pair{\K, + }$ is a topological group),  
and that $T$ is the theory of~$\K$.
\end{proviso*}
Notice that an algebraically closed field with the Zariski topology is not a
topological group, because $+$ is not continuous.
Notice also that, since we required that points are closed, $\K$~is a 
regular topological space.

\begin{remark}\label{rem:dense-sum}
Let $X \subseteq \K$ be  dense (but not necessarily definable).
Then, for every $b \in \K$ and every $V$ neighbourhood of~$0$, there exists 
$a \in X$ such that $b \in a + V$.
\end{remark}
\begin{proof}
Since $-$ is continuous, there exists $V'$ neighbourhood of $0$ such that
$V' = - V'$ and $V' \subseteq V$.
Since $X$ is dense, there exists $a \in X$ such that $a \in b + V'$.
Hence, $b \in a - V' \subseteq a + V$.
\end{proof}

\begin{corollary}\label{cor:dense-dminimal}
$\Td$ is complete.
Besides, $\Td$ is the theory of pairs $\pair{\K, \F}$ such that
$\F \prec \K \models T$ and $\F$ is a (topologically) dense 
proper subset of~$\K$.
\end{corollary}
\begin{proof}
We must show that if $\F \preceq \K$ is topologically dense in~$\K$, then $\F$
is \clclosed in~$\K$.
\Wlog, the pair $\pair{\K, \F}$ is $\omega$-saturated.
Let $b \in \mat^{\K}(\F)$; we must prove that $b \in \F$.
Let $Z \subset \K$ be $\F$-definable and discrete, such that $b \in Z$.
Let $U'$ be a definable neighbourhood of~$b$, such that $Z \cap U' = \set{b}$.
Define $U := U' - b$; since $\K$ is a topological group, 
$U$~is a neighbourhood of~$0$ in~$\K$, and there exists $V$ open
neighbourhood of $0$ definable in $\K$, such that $V = - V$ and 
$V + V \subseteq U$.


\begin{claim}
There exists and $\F$-definable open neighbourhood $W$ of $0$ such that
$W \subseteq V$.
\end{claim}
Suppose the claim is not true.
Since $\F$ is a regular space, there exists $X$ definable open neighbourhood
of~$0$ such that $\cl{X} \subseteq V$.
Let $X_\F := X \cap \F$.
$X_\F$ is a neighbourhood of $0$ in~$\F$; thus, since the topology has a 
definable basis, there exists $W_\F \subseteq X_\F$ such that 
$W_\F$ is $\F$-definable and $W_\F$ is an open neighbourhood of~$0$.
Let $W$ be the interpretation of $W_\F$ in~$\K$.
Since $W$ is open and $\F$ is dense in~$\K$, $W_\F$ is dense in~$W$;
therefore, $W \subseteq \cl{W_\F} \subseteq \cl{X} \subseteq V$.

By Remark~\ref{rem:dense-sum}, there exists $a \in \F$
such that $b \in W'$, where $W' := a + W$.

\begin{claim}
$W' \subseteq U'$.
\end{claim}
The claim is equivalent to $a + W \subseteq b + U$, 
that is $W + (a - b) \subseteq U$.
By assumption, $b \in a + W$, and therefore $a - b \in - W$.
Hence, $W + (a - b) \subseteq W - W \subseteq V - V \subseteq U$.

However, $W'$ is $\F$-definable, 
and $b \in W' \cap Z \subseteq V \cap Z = \set{b}$.
Hence, $b$ is $\F$-definable, and therefore $b \in \F$.
\end{proof}

Denote $\ind := \ind[$\zcl$]$.
Given $\av := \pair{\av_1, \dotsc, \av_n} \in \monster^{n \times m}$ and
$\bv \in \monster^n$, denote 
\[
B_{\av} + \bv := (B_{\av_1} + b_1) \times \dots \times (B_{\av_n} + b_n)
\subseteq \monster^n.
\]

\begin{lemma}\label{lem:neighbourhood-0}
Let $d \in \monster$, $V$~be a definable neighbourhood of~$d$,
and $C\subset \monster$.
Then, there exists $\av \in \monster^m$ such that $\av \ind_d C$ and
$d \in B_{\av} \subseteq V$.
\end{lemma}
\begin{proof}
Let $X := \set{\av \in \monster^n: d \in B_{\av}}$.
Let $\leq$ be the quasi-ordering on $X$ given by reverse inclusion:
$\av \leq \av'$ if $B_{\av} \supseteq B_{\av'}$.
Fix $\bv \in X$ such that $B_{\bv} \subseteq V$.
Since $(X, \leq)$ is a directed set, by Lemma~\ref{lem:order-ind}, there
exists $\av \in X$ such that $\av \indmat_d C$ and $B_{\av} \subseteq B_{\bv}
\subseteq V$.
\end{proof}

\begin{lemma}\label{lem:neighbourhood-1}
Let $\dv \in \monster^n$, $V$~be a definable neighbourhood of~$\dv$, and
$C \subset \monster$.
Then, there exist $\av \in \monster^{m \times n}$ and $\bv \in \monster^n$
such that $\dv \in B_{\av} + \bv \subseteq V$ and $\av \bv \ind C \dv$.
\end{lemma}
\begin{proof}
Proceeding by induction on~$n$, it suffices to treat the case $n = 1$.
Let $V_0 := V - d$; it is a definable neighbourhood of~$0$.
Since $\monster$ is a topological group, there exists $V_1$ definable and
open, such that $0 \in V_1$, $V_1 = - V_1$, and $V_1 + V_1 \subseteq V_0$.
By Lemma~\ref{lem:neighbourhood-0}, there exists $\av \in \monster^m$ such
that $\av \ind C d$ and $0 \in B_{\av} \subseteq V_1$.
Let $W := d - B_{\av}$.
Since $\dim(W) = 1$, there exists $b \in W$ 
such that $b \notin \zcl(C \av d)$.

\begin{claim}
$d \in B_{\av} + b$.
\end{claim}
In fact, $b \in - B_{\av} + d$, and therefore 
$d - b \in B_{\av}$

\begin{claim}
$a \bv \ind C d$.
\end{claim}
By construction, $b \ind C \av d$, and therefore $b \ind_{\av} C d$, and hence
$\av b \ind_{\av} C d$.
Together with $\av \ind C d$, this implies the claim.
\end{proof}

\begin{corollary}
Let $X \subseteq \monster^n$ be a definable set, and $k \in \N$.
Assume that, for every $\x \in X$, there exists $V_{\x}$ definable neighbourhood
of~$\x$, such that $\dim(V_{\x} \cap X) \leq k$.
Then, $\dim(X) \leq k$.
\end{corollary}
\begin{proof}
Let $C$ be the set of parameters of~$X$.
By Lemma~\ref{lem:neighbourhood-1}, for every $x \in X$ there exist
$\av \in \K^{n \times m}$ and $\bv \in \K^n$ such that
$\av \bv \ind C x$ and $\x \in B_{\av} + \bv \subseteq V_{\x}$;
notice that $\dim(X \cap (B_{\av} + \bv)) \leq k$
Hence, by Lemma~\ref{lem:neighbourhood-abstract}, $\dim(X) \leq k$.
\end{proof}

We do not know if the above Corollary remains true if we drop the assumption
that $\monster$ expands a group.

\begin{corollary}
Let $C \subset \monster$ and $p \in S_n(C)$.
Then, $p$ is stationary iff $p$ is realised in $\dcl(C)$.
\end{corollary}
\begin{proof}
Assume for contradiction, that $p$ is stationary, but $\dim(p) > 0$.
Let $\av_0$ and $\av_1$ be realisations of $p$ independent over~$C$.
Since $\dim(p) > 0$, $\av_0 \neq \av_1$.
Since $\monster$ is Hausdorff, Lemma~\ref{lem:neighbourhood-1} implies that
there exists $V$ open neighbourhood of $\av_0$, definable with
parameters~$\bv$, such that $\av_1 \notin V$ and $\bv \ind C\av_0 \av_1$.
Hence, by Lemma~\ref{lem:ind-ternary}, $\av_0 \ind_{C \bv} \av_1$.
Since $p$ is stationary, Lemma~\ref{lem:indipendent-morley-seq} implies 
$\av_0 \equiv_{\bv} \av_1$, contradicting the fact that $\av_0 \in V$, while
$\av_1 \notin V$.
\end{proof}


\section{Cl-minimal structures}

Let $\monster$ be a monster model, $T$~be the theory of~$\monster$,
and $\mat$ be an existential matroid
on~$\monster$.
We denote by $\dim$ and $\rk$ the dimension and rank induced by~$\mat$.

\begin{definizione}
$p \in S_n(A)$ is a \intro{generic type} if $\dim(p) = n$.
$\monster$ is \clminimal if, for every $A \subset \monster$,
there exists only one generic $1$-type over~$A$.
\end{definizione}

\begin{lemma}
For every $0 < n \in \N$ and $A \subset \monster$, there exists at least one
generic type in~$S_n(A)$.
If $\monster$ is \clminimal, then for every $n$ and $A$ there exists exactly
one generic type in~$S_n(A)$.
\end{lemma}

\begin{lemma}
If $\monster$ is \clminimal, then $\dim$ is definable.
\end{lemma}
\begin{proof}
Notice that, given $\x := \pair{x_1, \dotsc, x_n}$ and a formula~$\phi(\x, \y)$, the set $U_\phi^n := \set{\av:  \dim(\phi(\K,  \av)) = n}$ is always type-definable (Lemma~\ref{lem:type-def}).
By the above Lemma, $\K^n \setminus U^n_\phi = U^n_{\neg \phi}$,  and therefore
$U^n_\phi$ is both type-definable and ord-definable, and hence definable.
\end{proof}

\begin{remark}
$\monster$ is \clminimal iff, for every $n > 0$ and every $X$ definable subset
of~$\K^n$, exactly one among $X$ and $\K^n \setminus X$, has dimension~$n$.
\end{remark}

\begin{remark}\label{rem:clminimal-fixed}
If $\K \preceq \monster$ and $\dim$ is definable, then $\K$ is \clminimal iff,
for every $X$ definable subset of~$\K$, either $\dim(X) = 0$, or
$\dim(\K \setminus X) = 0$; that is, we can check \clminimality directly on~$\K$.
\end{remark}


\begin{examples}
\begin{enumerate}
\item $\monster$ is strongly minimal iff $\acl$ is a matroid and $\monster$ is
$\acl$-minimal.
\item 
Consider Example~\ref{ex:group}.
In that context, a type is generic in our sense iff it is generic in the sense
of stable groups.
Hence, $\Gm$~is \clminimal iff it has only one generic type iff it is
connected (in the sense of stable groups).
\end{enumerate}
\end{examples}



\begin{lemma}\label{lem:minimal-pair}
Assume that $T$ is \clminimal.
Then, $\Td$ is also \clminimal.
Moreover, $\Td$~coincides with~$\Ttwo$.
\end{lemma}
\begin{proof}
Let $\pair{\Bm^*, \Am^*}$ be a monster model of $\Td$.
Let $C \subset \Bm^*$ with $\card C < \kappa$.
Define $\Am := \mat(\Am^* C)$,
and $q_C(x)$ the partial $\Ltwo$-type over $C$ given by 
\[
q_C(x) := x \notin \Am.
\]
It is clear that every generic 1-$\Td$-type over $C$ expands~$q_C$.
Hence, it suffices to prove that $q_C$ is complete.
Let $b$ and $b' \in \Bm^*$ satisfy~$q_C$.
By Corollary~\ref{cor:vdd-27}, $\pair{\Bm^*, \Am^*} \preceq \pair{\Bm^*, \Am}$.
By assumption, $b$ and $b'$ are not in~$\Am$; hence, since $T$ is \clminimal,
they satisfy the same generic 1-$T$-type $p_{\Am}$; thus, by
Corollary~\ref{cor:Td-type-1}, $b \equiv^2_{\Am} b'$.
\end{proof}


\section{Connected groups}

Let $\monster$ be a monster model, and $\mat$ be an existential matroid on it.
Denote $\dim := \dimat$, $\rk := \rkmat$, and $\ind := \indmat$.

\begin{definizione}
Let $X \subseteq \monster^n$ be definable (with parameters).
Assume that $m := \dim(X) > 0$.
We say that $X$ is \emph{connected} if, for every $Y$ definable subset of~$X$,
either $\dim(Y) < n$, or $\dim(X \setminus Y) < n$.
\end{definizione}
For instance, if $\monster$ is \clminimal and $X = \monster$, then $X$ is
connected.

\begin{remark}
If $X$ is connected, then, for every $l \geq 0$, $X^l$ is also connected.
\end{remark}

\begin{remark}
Let $X \subseteq \monster^n$ be definable, of dimension $m > 0$.
Then, $X$ is connected iff for every $A \subset \monster$ containing the
parameters of definition of~$X$, there exists exactly
one $n$-type over $A$ in $X$ which is generic (\ie, of dimension~$m$).
\end{remark}

\begin{lemma}\label{lem:semigroup}
Let $G \subseteq \monster^n$ be definable and connected.
Assume that $G$ is a semigroup with left cancellation.
Assume moreover that $G$ has either right cancellation or right identity.
Then $G$ is a group.
\end{lemma}
Cf.\ \cite[1.1]{poizat87}.
\begin{proof}
Assume not. 
Let $m := \dim(G)$. 
\Wlog, $G$~is definable without parameters.
For every $a \in G$, 
let $a \cdot G := \set{a \cdot x: a \in G}$.
Since $G$ has left cancellation, we have $\dim(a \cdot \monster) = m$.

Let $F := \set{a \in G: a \cdot G = G}$.
Our aim is to prove that $F = G$.

It is easy to see that $G$ is multiplicatively closed.

First, assume that $G$ has a right identity element~$1$.
The following claim is true for any abstract semigroup with left cancellation
and right identity~$1$.
\begin{claim}
$1$ is also the left identity.
\end{claim}
In fact, for every $a, b \in G$, $a \cdot b = (a \cdot 1) \cdot b = a \cdot (1 \cdot b)$.
Since we have left cancellation, we conclude that $b = 1 \cdot b$ for every~$b$, and we are done.

Obviously, $1 \in F$.
For every $a \in F$, denote by $a^{-1}$ the (unique) element of $G$ such that
$a \cdot a^{-1} = 1$.

\begin{claim}
$a^{-1} \cdot a = 1$.
\end{claim}
In fact, $a \cdot (a^{-1} \cdot a) = 1 \cdot a = a \cdot 1$, and the claim
follows from left cancellation.

\begin{claim}
$F$ is a group.
\end{claim}
We have already seen that $F$ is multiplicatively closed and $1 \in F$.
Let $a \in F$. Then, for every $g \in G$, $a^{-1} \cdot (a \cdot g) = g$, and
therefore $a^{-1} \in F$.

\begin{claim}
$\dim(F) < m$.
\end{claim}
Assume, for contradiction, that $\dim(F) = m$.
Let $a \in G \setminus F$.
Then, $F \cap (a \cdot F) \neq \emptyset$; let $u, v \in F$ such that
$u = a \cdot v$.

Since $u \in F$ and $F$ is a group, there exists $w \in F$ such that 
$v \cdot w = 1$; hence, $u \cdot w = a \cdot 1 = a$, 
and therefore $a \in F$, absurd.

Choose $a, b \in G$ independent (over the empty set).
Since $\dim(a \cdot G) = \dim(b \cdot G) = m$, we have
$a \in b \cdot G$ and $b \in a \cdot G$.
Let $u, v \in G$ such that $b = a \cdot u$ and $a = b \cdot v$.
Hence, $a = a \cdot u \cdot v$.

Since $a \cdot 1 = a \cdot u \cdot v$, we have $1 = u \cdot v$.
Hence, both $u$ and $v$ are in~$F$.
However, since $\dim(F) < m$ and $b = a \cdot u$, 
we have $\rk(b/a) \leq \rk(u) < m$, absurd.

If instead $G$ has right cancellation, it suffices, by symmetry, to
show that $G$ has a left identity.
Reasoning as above, we can show that there exists $a$ and $b$ in $G$
such that $a \cdot b = a$.
We claim that $b$ is a left identity.
In fact, for every $c \in G$, we have
$a \cdot b \cdot c = a \cdot c$, and therefore $b \cdot c = c$, and we are done.
\end{proof}

\begin{proviso*}
For the remainder of this section, $\pair{G, \cdot}$ is a definable connected
group, of dimension~$m > 0$, with identity~$1$.
\end{proviso*}
If $G$ is Abelian, we will also use $+$ instead of $\cdot$ and $0$ instead 
of~$1$.

Hence, if $G$ expands a ring without zero divisors, then, by applying the
above lemma to the multiplicative semigroup of~$G$, we obtain that 
$G$ is a division ring.

\begin{remark}
Let $X \subseteq G$ be definable, such that $X \cdot X \subseteq X$.
Then, $\dim(X) = m$ iff $X = G$.
\end{remark}
\begin{proof}
Assume that $\dim(X) = m$.
Let $a \in G$.
Then, $X \cap (a \cdot X^{-1}) \neq \emptyset$; choose $u, v \in X$ such that
$u = a \cdot v^{-1}$.
Hence, $a = u \cdot v \in X \cdot X = X$.
\end{proof}


\begin{lemma}\label{lem:homo}
Let $f: G \to G$ be a definable homomorphism.
If $\dim(\ker f) = 0$, then $f$ is surjective.
\end{lemma}
\Cf \cite[1.7]{poizat87}.
\begin{proof}
Let $H := f(G)$ and $K := \ker(f)$; notice that $H < G$ and $K < G$.
Moreover, by additivity of dimension, $m = \dim(H) + \dim(K)$.
Hence, if $\dim(K) = 0$, then $\dim(H) = m$, therefore $H = G$ and $f$ is surjective.
\end{proof}

\begin{example}
$\pair{\Z, + }$ cannot be \clminimal, because the homomorphism $x \mapsto 2x$
has trivial kernel but is not surjective.
\end{example}

\begin{lemma}\label{lem:quotient}
Let $H < G$ be definable, with $\dim(H) = k < m$.
Then, $G / H$ is connected, and $\dim(G/H) = m - k$.
\end{lemma}
\begin{proof}
That $\dim(G/H) = m - k$ is obvious.
Let $X \subseteq G/H$ be definable of dimension $m - k$.
We must prove that $\dim(G/H \setminus X) < m$.
Let $\pi: G \to G/H$ be the canonical projection, and $Y := \pi^{-1}(X)$.
Then, $\dim(Y) = m$, and therefore $\dim(G \setminus Y) < m$.
Thus, $\dim(G/H \setminus X) = \dim(\pi(Y)) < m - k$.
\end{proof}

\begin{conjecture}
If $m = 1$, then $G$ is Abelian.
\Cf Reineke's Theorem \cite[3.10]{poizat87}.
\end{conjecture}
\begin{proof}[Idea for proof]
Assume for contradiction that $G$ is not Abelian.
Let $Z := Z(G)$ and $\overline G := G/Z$.
Since $Z < G$ and $Z \neq G$, we have $\dim(Z) = 0$, and therefore
$\overline G$ is also connected and of dimension~$1$.

For every $a \in G$, let  $U_a$ be the set of conjugates of~$a$.

\begin{claim}\label{cl:conjugate}
If $a \notin Z$, then $\dim(U_a) = 1$.
\end{claim}
By general group theory, $U_a \equiv G/C(a)$, where $C(a)$ is the centraliser
of~$a$.
Since $a \notin Z$, $C(a)$ is not all of~$G$; moreover, $C(a) < G$, 
therefore $\dim(C(a)) = 0$, and thus $\dim(U_a) = 1$, and similarly for~$U_b$.

\begin{claim}\label{cl:conjugate-2}
For every $a, b \in G \setminus Z$, $a$ is a conjugate of~$b$.
\end{claim}
In fact, by connectedness and the above claim, $U_a \cap U_b \neq \emptyset$, and thus $U_a = U_b$.


\begin{claim}
For every $x, y \in \overline G \setminus \set 1$, $x$ is a conjugate of~$y$.
\end{claim}
In fact, $x = \bar a$ and $y = \bar b$ for some $a, b \in G \setminus Z$.
By Claim~\ref{cl:conjugate-2}, $a$ and $b$ are conjugate in $G$, and thus
$x$ and $y$ are conjugate in $\overline G$.

Thus, $\overline G$ is a definable (in the imaginary sorts) connected group of
dimension~1, such that any two elements different
from the identity are in the same conjugacy class
(and therefore $\overline G$ is torsion-free and has trivial centre). 

One should now prove that such a group $\overline G$ cannot exists.
\end{proof}

Notice that the above conjecture is false if $m > 1$.


\begin{lemma}
Assume that $m = 1$ and $G$ is Abelian.
Let $p$ be a prime number.
Then, either $p G = 0$, or $G$ is divisible by~$p$.
\end{lemma}
\begin{proof}
Let $H := p G$ and $K := \set{x \in G: p x = 0}$.
If $\dim(H) = 1$, then $G = H$ and therefore $G$ is $p$-divisible.
If $\dim(H) = 0$, then $\dim(K) = 1$, thus $G = K$ and $p G = 0$.
\end{proof}

Notice that the above lemma needs the hypothesis that $m = 1$.
For instance, let $\monster$ be the algebraic closures of
$\mathbb F_p$, and let $G := \monster \times \monster^*$ (where $\monster^*$
is the multiplicative group of~$\monster$).

\begin{thm}
Assume that $G$ expands an integral domain (and $\pair{G, +}$ is connected).
Then, $G$~is an algebraically closed field.
\end{thm}
\Cf Macintyre's Theorem \cite[3.1, 6.11]{poizat87}.
\begin{proof}
Let $\pair{G^*, \cdot}$ be the multiplicative semigroup of~$G$.
By Lemma~\ref{lem:semigroup}, $\pair{G^*, \cdot}$ is a group, and
therefore $G$ is a field.
For every $n \in \N$, consider the map $f_n: G^* \to G^*$
$x \mapsto x^n$.
Since $f_n$ has finite kernel, Lemma~\ref{lem:homo} implies that $f_n$ is
surjective, and therefore every element of $G$ has an $n$th root in~$G$.
In particular, $G$~is perfect.

Let $p := \charact(G)$.
If $p > 0$, consider the map $h: G \to G$, $x \mapsto x^p + x$.
Notice that $h$ is an additive homomorphism with finite kernel; hence, $h$ is
surjective.

Since $G^l$ is also connected for every $0 <l \in \N$,
the above is true not only for~$G$, but also for every finite-degree extension
$G_1$ of~$G$.

The rest of the  proof is the same as in \cite[3.1]{poizat87}:
$G$ contains all roots of~$1$ (because $G^*$ is divisible), 
and, if $G$ were not algebraically closed, there would exists a finite
extension $G_1$ and normal finite extension $L$ of~$G_1$, such that the Galois
group of $L/G_1$ is cyclic and of prime order~$q$.
If $q \neq p$, then $L/G_1$ is a Kummer extension, absurd.
If $q = p$, then $L/G_1$ is an Artin-Schreier extension, also absurd.
\end{proof}

In the above theorem it is essential that $G$ is connected.
For instance, if $\monster$ is a formally $p$-adic field, then $\monster$
itself is a non-algebraically closed field (of dimension~$1$).
Notice also that the first step in the proofs of~\cite[3.1, 6.11]{poizat87} is
showing that $G$ is connected.

\begin{question}
Can we weaken the hypothesis in the above theorem from ``$G$~expands an
integral domain'' to ``$G$~expands a ring without zero divisors''?
\end{question}


\section{Ultraproducts}
Let $I$ be an infinite set, and $\mu$ be an ultrafilter on~$I$.
For every $i \in I$, let $\pair{\K_i, \mat_i}$ be a pair given by first-order
$\Lang$-structure $\K_i$ and an existential matroid $\mat_i$ on~$\K_i$.
Let $\Kfam$ be the family $\Pa{\pair{\K_i,  \mat_i}}_{i \in I}$, and
$\K := \Pi_i \K_i / \mu$ be the corresponding ultraproduct.

We will give some sufficient condition on the family~$\Kfam$, such that
there is an existential matroid on $\K$ induced by the family of~$\mat_i$.
Denote by $d_i$ the dimension induced by~$\mat_i$.

\begin{definizione}
We say that the dimension is uniformly definable (for the family~$\Kfam$)
if, for every formula $\phi(\x, \y)$ without parameters, 
$\x = \pair{x_1, \dotsc, x_n}$, $\y = \pair{y_1, \dotsc, y_m}$, 
and for every $l \leq n$,  there is a formula $\psi(\y)$, 
also without parameters, such that, for every $i \in I$,
\[
\set{\y \in \K_i^m: d_i\Pa{\phi(\K_i, \y)} = l} = \psi(\K_i).
\]
We denote by $d_\phi^l$ the formula~$\psi$.
\end{definizione}

\begin{remark}
The dimension is uniformly definable
if, for every formula $\phi(x, \y)$ without parameters, 
$\y = \pair{y_1, \dotsc, y_m}$,  there is a formula $\psi(\y)$, 
also without parameters, such that, for every $i \in I$,
\[
\set{\y \in \K_i^m: d_i\Pa{\phi(\K_i, \y)} = 1} = \psi(\K_i).
\]
\end{remark}

For instance, if every $\K_i$ expands a ring without zero divisors, then the
dimension is uniformly definable: given $\psi(x, \y)$, define $\psi(\y)$ by 
\[
\forall z\ \exists x_1, \dotsc x_4\ \Pa{z = F^4(x_1, \dotsc, x_4) \et
\bigwedge_{i = 1}^4 \phi(x_i, \y)}.
\]

For the remainder of this section, we assume that the dimension is uniformly
definable for~$\Kfam$.

\begin{definizione}
Let $d$ be the function from definable sets in $\K$ to
$\set{- \infty} \cup \N$ defined in the following way:\\
given  a $\K$-definable set $X = \Pi_{i \in I} X_i/ \mu$ and $l \in \N$,
$d(X) = l$ if, for $\mu$-almost every $i \in I$, $d(X_i) = l$.
\end{definizione}

\begin{thm}
$d$~is a dimension function on~$\K$.
Let $\mat$ be the existential matroid induced by~$d$.
Then, $a \in \mat(\bv)$ implies that, for $\mu$-almost every $i \in I$,
$a_i \in \mat_i(\bv_i)$, but the converse is \emph{not} true.
\end{thm}

\begin{remark}
Let $X \subseteq \K^n$ be definable with parameters~$\cv$; let $\phi(\x, \cv)$
be the formula defining~$X$.  
Given $l \in N$, $d(X) = l$ iff, for $\mu$-almost every $i \in I$, 
$\K_i \models d_\phi^l(\cv_i)$.
\end{remark}

\begin{lemma}
If each $\K_i$ is \clminimal, then $\K$ is also \clminimal.
\end{lemma}
\begin{proof}
By Remark~\ref{rem:clminimal-fixed}.
\end{proof}

\begin{examples}
The ultraproduct $\K$ of strongly minimal structures is not
strongly minimal in general (it will not even be a pregeometric structure), 
but if each structure expands a ring without zero
divisors, then $\K$ will have a (unique) existential matroid, 
and will be \clminimal.

It is easy to find a family $\Kfam = \Pa{\K_i}_{i \in \N}$ of strongly minimal
structures expanding a field, such that any non-principal ultraproduct of $\K$
$\Kfam$ is not pregeometric, does satisfy the Independence Property, and has
an infinite definable subset with a definable linear ordering.
Moreover, one can also impose that the trivial chain condition for uniformly
definable subgroups of $\pair{\K, +}$ fails in~$\K$ \cite[1.3]{poizat87}.
However, $\K$ will satisfy the following conditions:
\begin{enumerate}
\item Every definable associative monoid with left cancellation is a group
\cite[1.1]{poizat87};
\item Given  $G$ a definable group acting in a definable way on a definable
set~$E$, if $E$ is a definable subset of $A$ and $g \in G$ such that
$g \cdot A \subseteq A$, then $g \cdot A = A$
\cite[1.2]{poizat87}.
\end{enumerate}
\end{examples}

We do not know if conditions (1) and (2) in the above example are true for an
arbitrary \clminimal structure expanding a field.

\begin{remark}
Assume that each $\K_i$ is a first-order topological structure, and that the
definable basis of the topology of each $\K_i$ is given by the same function
$\Phi(x, \y)$. 
Then, $\K$ is also a first-order topological structure, and $\Phi(x, \y)$
defines a basis for the topology of~$\K$.
If each $\K_i$ is \dminimal, then $\K$ has an existential matroid,  
but it needs not be~\dminimal.

Assume that each $\K$ is \dminimal and satisfies the additional condition
\begin{itemize}
\item[(*)] Every definable subset of $\K_i$ of dimension $0$ is discrete.
\end{itemize}
Then, $\K$ is also \dminimal and satisfies condition~(*).
\end{remark}

\begin{example}
The ultraproduct of o-minimal structures is not necessarily
o-minimal, but it is \dminimal, and satisfies condition~(*).
\end{example}


\section{Dense tuples of  structures}
In this section we assume that $T$ expands the theory of integral domains.
We will extend the results of \S\ref{sec:dense-pairs} to dense tuples of
models of~$T$.

\begin{definizione}
Fix $n \geq 1$.  Let $\Ln$ be the expansion of $\Lang$ by $(n-1)$ new unary
predicates $P_1, \dotsc, P_{n - 1}$.  
Let $\Tn$ be the $\Ln$-expansion of $T$, whose models are sequences $\K_1
\prec \dots \prec \K_{n - 1} \prec \K_n \models T$, where each $\K_i$ is a
proper \clclosed elementary substructure of $\K_{i+1}$.  
Let $T^{nd}$ be the expansion of $T^{n+1}$ saying that $\K_1$ is \cldense
in~$\K_n$.
We also define $T^{0d} := T$.
\end{definizione}
For instance, $T^1 = T$, $\Ttwo$ is the theory we already defined in
\S\ref{sec:dense-pairs}, and $T^{1d} = \Td$.

\begin{lemma}
If $T$ is \clminimal, then $\Tn$ is complete for every $n \geq 1$ (and
therefore coincides with $T^{(n - 1)d}$).
Moreover, $\Tn$ has a (unique) existential matroid~$\mat^n$:
given $\pair{\K_n, \dotsc, \K_1} \models \Tn$, we have
$b \in \mat^n(A)$ iff $b \in \mat(A \K_{n-1})$.
Finally, $\Tn$ is $\mat^n$-minimal.
\end{lemma}
\begin{proof}
By induction on~$n$: iterate $n$ times Lemma~\ref{lem:minimal-pair}.
\end{proof}

\begin{corollary}
Assume that $T$ is strongly minimal.
Then, $\Tn$ is complete, and coincides with the theory of tuples
$\K_1 \prec \dots \prec \K_n \models T$.
\end{corollary}
\begin{proof}
One can use either the above Lemma, or reason as in~\cite{keisler}, using 
Lemma~\ref{lem:infinite-dimension}.
\end{proof}

\begin{remark}\label{rem:pair-rank}
Let $\pair{\Bm, \Am}$ be a $\kappa$-saturated model of $T^d$.
Let $U \subseteq \Bm$ be $\Bm$-definable and of dimension~$1$.
Then 
$\rk(U \cap \Am) \geq \kappa$.
\end{remark}

\begin{thm}
$T^{nd}$ is complete.
There is a (unique) existential matroid on~$T^{nd}$.
\end{thm}

\begin{proof}
By induction on~$n$, we will prove
$T^{nd}$ coincides with $(\dots (\Td)^d \dots)^d$ iterated $n$ times.
This implies both that $T^{nd}$ is complete, and that it has an existential
matroid. 

It suffices to treat the case $n = 2$.
Notice that $\pair{\K_2, \K_1} \prec \pair{\K_3, \K_1} \models \Td$.
It suffices to show that $\K_2$ is $\scl$-dense in $\pair{\K_3, \K_1}$.
\Wlog, we can assume that $\pair{\K_3, \K_2, \K_1}$ is $\kappa$-saturated.

Let $X \subseteq \K_3$ be definable in $\pair{\K_3, \K_1}$ (with parameters
from~$\K_3$), such that $\sdim(X) = 1$.
We need to show that $X$ intersects~$\K_2$.
By Corollary~\ref{prop:3.5}, there exists $U$ and $S$ subsets of $\K_3$,
such that $X$ is definable in~$\K_3$, $S$~is definable in $\pair{\K_3,\K_1}$
and small, and $X = U \sdiff S$. Therefore, $\dim(U) = 1$.
If, by contradiction, $X \cap \K_2 = \emptyset$, then
$\K_2  \cap U \subseteq S$; therefore, $\srk(\K_2 \cap U) < \omega$
(where $\srk$ is the rank induced by~$\scl$),
contradicting Remark~\ref{rem:pair-rank}.
\end{proof}

\begin{corollary}
Assume that $T$ is \dminimal.
Then, $T^{nd}$ coincides with the theory of $(n+1)$-tuples
$\K_1 \prec \dots \prec \K_n \prec \K_{n+1} \models T$, such that
$\K_1$ is (topologically) dense in $\K_{n+1}$.
\end{corollary}
\begin{proof}
Notice that if $\pair{\K_n, \dotsc, \K_1}$ satisfy the assumption, 
then, by Corollary~\ref{cor:dense-dminimal}, each $\K_i$ is \clclosed in $\K_n$.
\end{proof}


\subsection{Dense tuples of topological structures}
Assume that $T$ expands the theory of integral domains.
Assume that $\monster$ has both an existential matroid~$\mat$ and a definable
topology (in the sense of~\cite{pillay87}.
Let $\phi(x, \y)$ be a formula such that the family of sets
\[
B_{\bv} := \phi(\monster, \bv),
\]
as $\bv$ varies in $\monster^k$, is a basis of the topology of~$\monster$.
If $\bv = \pair{\bv_1, \dotsc, \bv_m}$, 
we denote by $B^n_{\bv} := B_{\bv_1} \times \dots \times B_{\bv_m} \subseteq \monster^m$.

Assume the following conditions:
\begin{hypothesis*}
\begin{enumerate}[\upshape I.]
\item Every definable non-empty open subset of $\monster$ has dimension~$1$.
\item For every $m \in \N$, every $U$ open subset of $\monster^m$,
and every $\av \in U$, the set
$\set{\bv: \av \in B_{\bv}}$ has non-empty interior.
\end{enumerate}
\end{hypothesis*}

\begin{remark}
Assumption~I implies that a definable subset of $\monster^m$ with non-empty
interior has dimension $m$ (but the converse is not true: there can be
definable subsets of dimension $m$ but with empty interior).
Moreover, it implies that a $\mat$-dense subset of $\monster^m$ is also
topologically dense (but, again, the converse is not true).
\end{remark}

\begin{example}
\begin{enumerate}
\item 
If $\monster$ is either a valued field (with the valuation topology) or a
linearly ordered field (with the order topology), then it satisfies
Assumption~II.
\item
If $\monster$ is a d-minimal structure, then it satisfies Assumption~I.
\item
Let $\monster$ be either a formally $p$-adic field, or an algebraically closed
valued field, or a d-minimal expansion of 
a linearly ordered definably complete field (\cf Example~\ref{ex:dmin}).
Then, $\monster$ satisfies both assumptions.
\end{enumerate}
\end{example}

We have two notions of closure and of density on~$\monster$: 
the ones given by the topology and the ones given by the matroid;
to distinguish them, we will speak about topological closure and
\clclosure respectively (and similarly for density).

The following theorem follows easily from~\cite{BH} (we are assuming that the
Hypothesis holds).
\begin{thm}[{\cite[Corollary~3.4]{BH}}]
Let $\Cm := \pair{\Bm, \Am_{n-1}, \dotsc, \Am_1} \models T^{n d}$.
Let $\cv \subset \Bm$ be \clindependent over $\cv \cap \Am_{n-1}$.
Let $U \subseteq \Bm^m$ be open and definable in~$\Cm$, with parameters~$\cv$.
Then, $U$ is definable in~$\Bm$, with parameters~$\cv$.
\end{thm}
In the terminology of~\cite{DMS}, the above theorem proves that $\Bm$ is the
\intro{open core} of~$\Cm$.
\begin{proof}
By induction on $n$, it suffices to do the case when $n = 2$,
\ie when $\Dm = \pair{\Bm, \Am}$.
\Wlog, $\Dm$ is $\lambda$-saturated and $\lambda$-homogeneous, for some
$\abs T < \lambda < \kappa$
We want to verify that the hypothesis of \cite[Corollary~3.1]{BH} is satisfied.
Let $D_m := \set{\bv \in \Bm^m: \srk(\bv/\cv) = m}$.
\begin{enumerate}
\item 
If $V \subseteq \Bm^n$ is $\Bm$-definable and of dimension~$m$, then 
$V \cap D_m$ is non-empty: therefore, $D_m$ is topologically dense in $\Bm^m$.
\item
Let $\dv \in D_m$, and $U \subseteq \monster^m$ be open, such that
$p := \tp^1(\dv/\cv)$ is realized in~$U$.
We have to show that $p$ is realized in $U \cap D_m$.
Let $\dv' \in U$ be such a realization, and let $\bv \subset \Bm$ such
that $\dv' \in B_{\bv}$.
Since $\dv' \elem^1_{\cv} \dv$, we have that $\dv'$ is \clindependent 
over~$\cv$.
By changing $\bv$ if necessary, we can also assume that $\dv' \ind \bv \cv$
(\cf the proof of Lemma~\ref{lem:neighbourhood-1}), and thus $\dv'$ is
\clindependent over $\bv \cv$.
Finally, since $\Am$ is $\mat$-dense in $\Bm$, there exists
$\dv'' \elem^1_{\bv \cv} \dv$ such that $\dv''$ is \clindependent over
$\bv \cv \Am$.
\item
By Proposition~\ref{prop:back-and-forth}, for every $dv \in D_m$, 
$\tp^2(\dv / \cv)$ is determined by $\tp^1(\dv / \cv)$ in conjunction with 
``$\dv \in D_m$''.
\end{enumerate}
Hence, we can apply \cite[Corollary~3.1]{BH} and we are done.
\end{proof}



\section{The (pre)geometric case}

Remember that $\monster$ is a pregeometric structure if $\acl$ satisfies EP.
If moreover $\monster$ eliminates the quantifier $\exists^\infty$, then
$\monster$ is geometric.

In this section we gather various results about (pre)geometric structures,
mainly in order to clarify and motivate the general case of structures with an
existential matroid. 

Remember that $\monster$ has geometric elimination of imaginaries if every
for imaginary tuple $\av$ there exists a real tuple $\bv$ such that $\av$ and
$\bv$ are inter-algebraic.

\begin{remark}
$T$~is pregeometric iff $T$ is a real-rosy theory of real \th-rank~1.
Moreover, if $T$ is pregeometric and has geometric elimination of imaginaries,
then $\tind = \ind[$\acl$]$,  and $\dim^{\acl}$ is equal to the \th-rank: see
\cite{EO} for definitions and proofs.
\end{remark}

\begin{remark}
The model-theoretic algebraic closure $\acl$ is a definable
closure operator.
\end{remark}

\begin{fact}
Let $\monster$ be a definably complete and \dminimal expansion of a field.
Then, $\monster$ has elimination of imaginaries an \DSF; moreover,
$\monster$ is rosy iff it is o-minimal.
In particular, an ultraproduct of o-minimal structures expanding a field is
rosy iff it is o-minimal.
\end{fact}
The proof of the above fact will be given elsewhere.

For the remainder of this section, $\monster$ is pregeometric (and $T$ is its
theory). 

\begin{remark}
$\acl$ is an existential matroid on $\monster$.
The induced independence relation $\ind[$\acl$]$ coincides with
real \th-independence $\tind$ and with the $M$-dividing
notion $\mind$ of~\cite{adler}.
A formula is \xnarrow (for $\acl$) iff it is algebraic in~$x$.
\end{remark}

\begin{remark}
Let $X \subseteq \monster^n$ be definable.
$\dim^{\acl}(X) = 0$ iff $X$ is finite.
\end{remark}

\begin{remark}
$\monster$ is geometric iff $\dim^{\acl}$ is definable.
\end{remark}

\begin{remark}
$\monster$ is $\acl$-minimal iff it is strongly minimal.
\end{remark}

In \S \ref{sec:imaginary} we defined $\tilde \acl$, the extension of $\acl$ to
the imaginary sorts.

\begin{remark}
If $a$~is real and $B$ is imaginary, then
$a \in \tilde\acl(B)$ iff $a \in \acleq(B)$.
\end{remark}

\begin{remark}
\Tfae:
\begin{enumerate}
\item $\acleq$ coincides with the extension of $\acl$ defined in \S\ref{sec:imaginary};
\item $T$ is superrosy of \th-rank~1 \cite{EO};
\item $T$ is surgical \cite{gagelman}.
\end{enumerate}
\end{remark}

\begin{remark}\label{rem:geometric-density}
$X$ is dense in $\monster$ iff for every $U$ infinite definable subset 
of~$\monster$,  $U \cap X \neq \emptyset$.
If $\F \preceq \K$, then $\F$ is $\acl$-closed in~$\K$.
\end{remark}

\begin{remark}
Assume that $T$ is geometric.
Then, $\Ttwo$ is  the theory of pairs $\pair{\K,\F}$, with $\F \prec \K
\models T$, and  $\Td$ is the theory of pairs $\pair{\K,\F} \models \Ttwo$, 
such that $\F$ is dense in~$\K$.
For every $X \subseteq \K$, $\scl(X) = \acl(\F X)$ (\cf
Question~\ref{question:dmin-scl-1}).
\end{remark}

For more on the theory $\Td$ in the case when $T$ is geometric, and in
particular when $T$ is o-minimal, see \cite{berenstein}.

\section*{Acknowledgments}
I thank H.~Adler, A.~Berenstein,
G.~Boxall, J.~Ramakrishnan, K.~Tent, and M.~Ziegler for
helping me to understand the subject of this article.



\begin{thebibliography}{9999}

\bibitem[Adler05]{adler} H.~Adler.
\newblock{Explanation of Independence.}
\newblock Dissertation
zur Erlangung des Doktorgrades der Fakult\"at f\"ur Mathematik und Physik der 
Albert-Ludwigs-Universit\"at Freiburg im Breisgau, June 2005.

\bibitem[BPV03]{BPV}
I. Ben-Yaacov,  A. Pillay, E. Vassiliev. 
\newblock {\em Lovely pairs of models}.
\newblock Pure Appl. Logic  \textbf{122}  (2003),  no.~1--3, 235--261.


\bibitem[Berenstein07]{berenstein} A.~Berenstein.
\newblock{\em Lovely pairs and dense pairs of o-minimal structures.}
\newblock Preprint, 2007.

\bibitem[Boxall09]{boxall} G. Boxall.
\newblock Lovely pairs and dense pairs of real closed fields.
\newblock Dissertation, University of Leeds, July 2009.

\bibitem[BH10]{BH} G.~Boxall and P.~Hieronymi.
\newblock{\em Expansions which introduce no new open sets.}
\newblock Preprint, 2010.

\bibitem[CF04]{CF} E.~Casanovas, Enrique and  R.~Farr\'e. 
\newblock{\em Weak forms of elimination of imaginaries.}
\newblock MLQ Math.\ Log.0 Q.\  50  (2004),  no.~2, 126--140.


\bibitem[CDM92]{cdm} Z.~Chatzidakis,  L.~van den Dries, and A.~Macintyre. 
\newblock{\em Definable sets over finite fields.}
 J.~Reine Angew.\ Math.\  427  (1992), 107--135.


\bibitem[DMS08]{DMS}
A.~Dolich, C.~Miller, and C.~Steinhorn.
\newblock{\em Structures having o-minimal open core.}
\newblock Preprint.

\bibitem[Dries89]{dries89} L.~van den Dries.
\newblock{\em Dimension of definable sets, algebraic boundedness and Henselian
 fields.}
\newblock Stability in model theory, II (Trento, 1987).
 Ann.\ Pure Appl.\ Logic  45  (1989),  no.~2, 189--209.


\bibitem[Dries98]{vdd-dense} L.~van den Dries.
\newblock{\em Dense pairs of o-minimal structures.}
\newblock Fundamenta Mathematicae 157 (1998), 61--78.


\bibitem[EO07]{EO} C.~Ealy, and  A.~Onshuus. 
\newblock{\em Characterizing rosy theories.}
\newblock J.~Symbolic Logic 72  (2007),  no.~3, 919--940.


\bibitem[Gagelman05]{gagelman} J.~Gagelman. 
\newblock{\em Stability in geometric theories.}
Ann.\ Pure Appl.\ Logic  132  (2005),  no.~2-3, 313--326.

\bibitem[HP94]{hp} E.~Hrushovski, A.~Pillay. 
\newblock{\em Groups definable in local fields and pseudo-finite fields.}
Israel J.~Math.\  85  (1994),  no.~1-3, 203--262.

\bibitem[Keisler64]{keisler} H.J.~Keisler.
\newblock{\em Complete theories of algebraically closed fields with
  distinguished  subfields.}
\newblock Michigan\ Math.\ J.~11  1964 71--81.

\bibitem[Macintyre75]{macintyre} A.~Macintyre. 
\newblock{\em Dense embeddings. I. A theorem of Robinson in a general setting.}
\newblock  Model theory and algebra (A memorial tribute to Abraham Robinson), 
pp.~200--219. Lecture Notes in Math., Vol. 498, Springer, Berlin,  1975. 


\bibitem[Miller05]{miller} C.~Miller. 
\newblock Tameness in expansions of the real field.
\newblock Logic Colloquium '01,  281--316, Lect. Notes Log., 20, 
Assoc. Symbol. Logic, Urbana, IL,  2005. 

		
\bibitem[Pillay87]{pillay87}
A. Pillay. 
\newblock{\em First order topological structures and theories.}
\newblock J.~Symbolic Logic 52 (1987),  no.~3, 763--778.

\bibitem[Poizat85]{poizat85} B.~Poizat.
\newblock Cours de th\'eorie des mod\`eles.
Une introduction \`a la logique math\'ematique contemporaine. 
\newblock Nur al-Mantiq wal-Ma'rifah. 
Bruno Poizat, Lyon,  1985. vi+584~pp.

\bibitem[Poizat87]{poizat87} B.~Poizat. 
\newblock Stable groups.
\newblock Mathematical Surveys and Monographs, 87.
American Mathematical Society, Providence, RI,  2001. xiv+129~pp.

\bibitem[Robinson74]{robinson74}
A. Robinson.
\newblock{\em A note on topological model theory.}
\newblock Fund.\ Math.\ 81 (1973/74), no.~2, 159--171. 

\bibitem[TZ08]{ZT} K.~Tent and M.~Ziegler.
\newblock Model theory.
\newblock Preprint, 2009-05-27. V.~0.2.414. 

\bibitem[Wagner00]{wagner} F.O.~Wagner.
\newblock  Simple theories.
\newblock Mathematics and its Applications, 503. 
Kluwer Academic Publishers, Dordrecht, 2000. xii+260~pp.


\bibitem[Wood76]{wood} C.~Wood. 
\newblock{\em The model theory of differential fields revisited.}
\newblock Israel J.~Math.\ 25  (1976),  no.~3-4, 331--352.


\end{thebibliography}
\end{document}